\documentclass[10 pt, reqno]{amsart}
\usepackage{color}
\usepackage{verbatim}
\usepackage{amssymb}
\usepackage{amscd}
\usepackage{amsmath}
\usepackage{amsthm}
\usepackage{eucal}
\usepackage{setspace}
\usepackage{url}
\usepackage[utf8]{inputenc}
\usepackage{booktabs}

\usepackage[hyperfootnotes=false, colorlinks, linkcolor={blue}, citecolor={magenta}, filecolor={blue}, urlcolor={blue}]{hyperref}

\allowdisplaybreaks
\renewcommand{\Im}{\mathrm{Im}}
\renewcommand{\Re}{\mathrm{Re}}
\renewcommand{\H}{\mathbb H}
\newcommand{\leftexp}[2]{{\vphantom{#2}}^{#1}%
      \kern-\scriptspace%
      {#2}}
\newcommand{\A}{{\mathbb A}}
\newcommand{\E}{{\mathcal E}}

\newcommand{\Tr}{\mathrm{Tr}}
\newcommand{\Q}{{\mathbb Q}}
\newcommand{\Z}{{\mathbb Z}}

\newcommand{\R}{{\mathbb R}}
\newcommand{\C}{{\mathbb C}}
\newcommand{\bs}{\backslash}

\newcommand{\p}{\mathfrak p}
\newcommand{\g}{\mathfrak g}
\newcommand{\U}{\mathcal U}

\newcommand{\GL}{{\rm GL}}

\newcommand{\SL}{{\rm SL}}

\newcommand{\GSp}{{\rm GSp}}
\newcommand{\Sp}{{\rm Sp}}

\newcommand{\Aut}{{\rm Aut}}

\newcommand{\vol}{{\mathrm {vol}}}

\newcommand{\diag}{{\rm diag}}

\newcommand{\Hom}{{\rm Hom}}

\newcommand{\sym}{{\rm sym}}
\newcommand{\vl}{{\rm vol}}

\newcommand{\Lambdap}{\Lambda^+}

\renewcommand{\AA}{{\mathcal A}}
\newcommand{\HH}{{\mathbb H}}
\newcommand{\mat}[4]{{\setlength{\arraycolsep}{0.5mm}\left[\begin{array}{cc}#1&#2\\#3&#4\end{array}\right]}}

\newtheorem{lemma}{Lemma}[section]
\newtheorem{theorem}[lemma]{Theorem}

\newtheorem{corollary}[lemma]{Corollary}
\newtheorem{proposition}[lemma]{Proposition}
\newtheorem{definition}[lemma]{Definition}
\theoremstyle{remark}
\newtheorem{remark}[lemma]{Remark}

\begin{document}

\bibliographystyle{plain}

\title[Lowest weight modules and nearly holomorphic forms]{Lowest weight modules of $\Sp_4(\R)$ and nearly holomorphic Siegel modular forms}

\thanks{A.S.\ is partially supported by EPSRC grant EP/L025515/1. A.P.\ and R.S.\ are supported by NSF grant DMS--$1100541$.}
\author{Ameya Pitale}
\address{Department of Mathematics
\\ University of Oklahoma\\ Norman\\
   OK 73019, USA}
\email{apitale@math.ou.edu}

\author{Abhishek Saha}
\address{School of Mathematical Sciences \\
  Queen Mary University of London\\
 London E1 4NS \\
  UK} \email{abhishek.saha@qmul.ac.uk}

\author{Ralf Schmidt}
\address{Department of Mathematics
\\ University of Oklahoma\\ Norman\\
   OK 73019, USA}
\email{rschmidt@math.ou.edu}

\begin{abstract}
We undertake a detailed study of the lowest weight modules for the Hermitian symmetric pair $(G,K)$, where $G=\Sp_4(\R)$ and $K$ is its maximal compact subgroup. In particular, we determine $K$-types and composition series, and write down explicit differential operators that navigate all the highest weight vectors of such a module starting from the unique lowest-weight vector. By rewriting these operators in classical language, we show that the automorphic forms on $G$ that correspond to the highest weight vectors are exactly those that arise from \emph{nearly holomorphic} vector-valued Siegel modular forms of degree $2$.

Further, by explicating the algebraic structure of the relevant space of \emph{$\mathfrak{n}$-finite}  automorphic forms, we are able to prove  a \emph{structure theorem} for the space of nearly holomorphic vector-valued Siegel modular forms of (arbitrary) weight $\det^\ell \sym^m$ with respect to an arbitrary congruence subgroup of $\Sp_4(\Q)$. We show that the \emph{cuspidal} part of this space is the \emph{direct sum} of subspaces obtained by applying explicit differential operators to \emph{holomorphic} vector-valued cusp forms of weight $\det^{\ell'} \sym^{m'}$ with $(\ell', m')$ varying over a certain set. The structure theorem for the space of \emph{all modular forms} is similar, except that we may now have an additional component coming from certain nearly holomorphic forms of weight $\det^{3}\sym^{m'}$ that cannot be obtained from holomorphic forms.

As an application of our structure theorem, we prove several arithmetic results concerning nearly holomorphic modular forms that improve previously known results in that direction.
\end{abstract}

 \maketitle

\section{Introduction}
\subsection{Motivation}
In a series of influential works~\cite{Shimura1986, Shimura1987, Shimuraholoproj, shimura2000}, Shimura defined the notion of a \emph{nearly holomorphic function} on a K\"{a}hler manifold $\mathbb{K}$ and proved various properties of such functions. Roughly speaking, a nearly holomorphic function on such a manifold is a polynomial of some functions $r_1, \ldots r_m$ on $\mathbb{K}$ (determined by the K\"{a}hler structure), over the ring of all holomorphic functions. For example, if $\mathbb{K} = \H_n$, the symmetric space for the group $\Sp_{2n}(\R)$, then $r_i$ are the entries of $\Im(Z)^{-1}$. When there is a notion of holomorphic modular forms on $\mathbb{K}$, one can define nearly holomorphic (scalar or vector-valued) modular forms by replacing holomorphy by near-holomorphy in the definition of modular forms.

The prototype of a nearly holomorphic modular form in the simplest case when $\mathbb{K}$ equals the complex upper-half plane $\H$ is provided by the function
\begin{equation}\label{prototype}
 f(z):= \bigg(\sum_{(c,d) \neq (0,0)}(cz+d)^{-k} |cz +d|^{-2s} \bigg)_{s=0}.
\end{equation}
Here $k$ is a positive even integer. The function $f$ transforms like a modular form of weight $k$ with respect to $\SL_2(\Z)$. If $k >2$, the function is holomorphic, but the case $k=2$ involves a non-holomorphic term of the form $\frac{c}{y}$, where $c$ is a constant.

More generally, special values of Eisenstein series\footnote{The typical situation is as follows. Let $E(z,s)$ be an appropriately normalized Eisenstein series on some Hermitian symmetric space that converges absolutely for $\Re(s)>s_0$ and transforms like a modular form in the variable $z$. Suppose that $E(z,k)$ is holomorphic for some $k \in \Z$. Then $E(z,s')$ is typically a nearly holomorphic modular form for all $s'$ such that $s_0 <s' \le k$, $s' \in \Z$; see~\cite[Thm. 4.2]{Shimura1987}.}, and their restrictions to lower-dimensional manifolds, provide natural examples of nearly holomorphic modular forms. On the other hand, such restrictions of Eisenstein series appear in the theory of $L$-functions via their presence in integrals of Rankin-Selberg type. Thus, the arithmetic theory of nearly holomorphic forms is closely related to the arithmetic theory of $L$-functions. The theory was developed by Shimura in substantial detail and was exploited by him and other authors to prove algebraicity and Galois-equivariance of critical values of various $L$-functions. We refer the reader to the papers~\cite{bluher, bochpilot96, heimboch, sah2, ShimuraHilbert, shimura2000} for some examples. The theory of nearly holomorphic modular forms and the differential operators related to them has also been very fruitful in the study of $p$-adic measures related to modular $L$-functions~\cite{bochschmidt, panchishkin2004, panchishkin2005} and in the derivation of various arithmetic identities~\cite{gradechi, lanphiernearholo}.

From now on, we restrict ourselves to the symplectic case, and we assume further that the base field is $\Q$. The relevant manifold $\mathbb{K}$ is then the degree $n$ Siegel upper half space $\H_n$ consisting of  symmetric $n$ by $n$ matrices $Z =X+iY$ with $Y>0$. For each non-negative integer $p$, we let $N^p(\H_n)$ denote the space of all polynomials of degree $\le p$ in the entries of $Y^{-1}$ with holomorphic functions on $\H_n$ as coefficients. The space $N(\H_n)=\bigcup_{p\ge 0}N^p(\H_n)$ is the space of nearly holomorphic functions on $\H_n$. Note that $N^0(\H_n)$ is the space of holomorphic functions on $\H_n$.

Given any congruence subgroup $\Gamma$ of $\Sp_{2n}(\Q)$ and any irreducible finite-dimensional rational representation $(\eta, V)$ of $\GL_n(\C)$, we let $N_{\eta}^p(\Gamma)$ denote the space of functions $F: \H_n \rightarrow V$ such that
\begin{enumerate}
 \item $F \in N^p(\H_n)$,
 \item $F(\gamma Z) =\eta(CZ+D) (F(Z))$ for all $\gamma = \mat{A}{B}{C}{D} \in \Gamma$.
 \item $F$ satisfies the ``no poles at cusps" condition. This condition states that the Fourier expansion of $F$ at any cusp is supported  on the positive semi-definite matrices.\footnote{If $n>1$, the ``no poles at cusps" condition follows automatically from the previous two conditions by the Koecher principle.}
\end{enumerate}

The set $N_{\eta}^p(\Gamma)$ (which is clearly a complex vector-space) is known as the space of nearly holomorphic vector-valued modular forms of weight $\eta$ and nearly holomorphic degree $p$ for $\Gamma$. In the special case $(\eta, V) = (\det^k, \C)$, we denote the space $N_{\eta}^p(\Gamma)$ by $N_{k}^p(\Gamma)$. We let $N_{\eta}^p(\Gamma)^\circ \subset N_{\eta}^p(\Gamma)$ denote the subspace of cusp forms (the cusp forms can be defined in the usual way via a vanishing condition at all cusps for degenerate Fourier coefficients). We also denote $M_\eta(\Gamma) = N_\eta^0(\Gamma)$, $S_\eta(\Gamma) = N_\eta^0(\Gamma)^\circ$, $N_\eta(\Gamma)= \bigcup_{p\ge 0} N^p_\eta(\Gamma)$ and $N_\eta(\Gamma)^\circ= \bigcup_{p\ge 0} N^p_\eta(\Gamma)^\circ$.

In the case $n=1$, Shimura proved~\cite[Thm.\ 5.2]{Shimura1987} a complete \emph{structure theorem} that describes the set $N_{k}^p(\Gamma)$ precisely for every weight $k$ and every congruence subgroup $\Gamma$ of $\SL_2(\Z)$. For simplicity, write $N_k(\Gamma) =\bigcup_{p\ge 0} N_k^p(\Gamma)$. Let $R$ denote the classical \emph{weight-raising operator} on $\bigcup_{k}N_k(\Gamma)$ that acts on elements of $N_k(\Gamma)$ via the formula $\frac ky+2i\frac{\partial}{\partial z}$. It can be easily checked that $R$ takes  $N_{k}^p(\Gamma)$ to  $N_{k+2}^{p+1}(\Gamma)$. Then a slightly simplified version of the structure theorem of  Shimura says that $N_0(\Gamma)=\C$, and for $k>0$,
\begin{equation}\label{structuredegree1}N_k(\Gamma)=R^{\frac{k-2}2}(\C E_2) \:\oplus\:\bigoplus_{\substack{\ell\ge 1}} R^{\frac{k-\ell}2}\left(M_\ell(\Gamma)\right), \qquad N_k(\Gamma)^\circ
      =\bigoplus_{\substack{\ell\ge 1}} R^{\frac{k-\ell}2}\left(S_\ell(\Gamma)\right),
\end{equation}
where we understand $R^{v} = 0$ if $v \notin \Z_{\ge0}$, and where $E_2$ denotes the weight 2 nearly holomorphic Eisenstein series obtained by putting $k=2$ in~\eqref{prototype}. For the refined structure theorem taking into account the nearly holomorphic degree, we refer the reader to~\cite{pssdeg1}, where we reprove Shimura's results using representation-theoretic methods.

Shimura used his structure theorem to prove that the cuspidal holomorphic projection map from $N_k(\Gamma)$ to $S_k(\Gamma)$ has nice $\Aut(\C)$-equivariance properties, and he even extended these results to the half-integral case~\cite[Prop.\ 9.4]{ShimuraHilbert}. As an application, Shimura obtained many arithmetic results for ratios of Petersson norms and critical values of $L$-functions.

In the case $n >1$, Shimura showed \cite[Prop.\ 14.2]{shimura2000} that if the lowest weight of $\eta$ is ``large enough" compared to the nearly holomorphic degree, then the space $N^p_{\eta}(\Gamma)$ is \emph{spanned} by the functions obtained by letting differential operators act on various spaces $M_{\eta'}(\Gamma)$. Using this, he was able to construct an analogue of the projection map under some additional assumptions. But the arithmetic results thus obtained are weaker than those for $n=1$.

There is another aspect in which the state of our understanding of nearly holomorphic modular forms is unsatisfactory, namely that the precise meaning of these objects in the modern language of automorphic forms on reductive groups, \`{a} la Langlands, has not been worked out. Most work done so far for nearly holomorphic forms has been in the classical language. There has been some work in interpeting these forms from the point of view of vector bundles and sheaf theory, see \cite{harrisarithmetic85,harrisarithmetic86,nappari,Urban2013}. There has also been some work on interpreting the differential operators involved in the language of Lie algebra elements, but this has been carried out explicitly only in the case $n=1$~\cite{gradechi, harris3}. A detailed investigation from the point of view of automorphic representations has so far been lacking in the case $n>1$.

The objective of this paper is to address the issues discussed above in the case $n=2$, i.e., when $\Gamma$ is a  congruence subgroup of $\Sp_4(\Q)$. The relevant $\eta$'s in this case are the representations $\det^\ell \sym^m$ for integers $\ell, m$ with $m\ge 0$, and it is natural to use $N_{\ell, m}(\Gamma)$ to denote the corresponding space of nearly holomorphic forms. We achieve the following goals.

 \begin{itemize}

 \item We prove a structure theorem for  $N_{\ell, m}(\Gamma)$ that is (almost) as complete and explicit as the $n=1$ case. As a consequence, we are able to prove arithmetic results for this space (as well as for certain associated ``isotypic projection" maps, and ratios of Petersson inner products) that  improve previously known results in this direction.

  \item We make a detailed study of the spaces $N_{\ell, m}(\Gamma)$ in the language of $(\g, K)$-modules and automorphic forms for the group $\Sp_4(\R)$. We analyze the $K$-types, weight vectors and composition series, write down \emph{completely explicit} operators from the classical as well as Lie-theoretic points of view, explain exactly how nearly holomorphic forms arise in the Langlands framework, and describe the automorphic representations attached to them.
 \end{itemize}
In the rest of this introduction we explain these results and ideas in more detail.

\subsection{The structure theorem in degree 2}

Let $\Gamma$ be a congruence subgroup of $\Sp_4(\Q)$. In order to prove a structure theorem for $N_{\ell, m}(\Gamma)$, it is necessary to have suitable differential operators that generalize the weight-raising operator considered above. In fact, it turns out that one needs four operators, which we term $X_+$, $U$, $E_{+}$ and $D_{+}$.

Each of these four operators acts on the set $\bigcup_{\ell, m}N_{\ell, m}(\Gamma)$. They take the subspace $N_{\ell, m}^p(\Gamma)$ to the subspace $N_{\ell_1, m_1}^{p_1}(\Gamma)$, where the integers $\ell_1, m_1, p_1$ are given by the following table.

\label{Nplusoperatorsnearholotableintro}
\begin{equation}\label{introtable}\renewcommand{\arraystretch}{1.5}\renewcommand{\arraycolsep}{0.4cm}
 \begin{array}{ccccc}
  \toprule
   \text{operator}&\text{ $\ell_1$}&\text{ $m_1$}&\text{$p_1$}\\
  \toprule
   X_+&\ell&m+2&p+1\\
  \midrule

   U&\ell+2&m-2&p+1\\
  \midrule

   E_{+}&\ell+1&m&p+1\\
  \midrule

   D_+&\ell+2&m&p+2\\

  \bottomrule
 \end{array}
\end{equation}
\medskip

Note that, in the above list, $E_+$ is the only operator that changes the parity of $\ell$. For the explicit formulas for the above differential operators, see~\eqref{operatorsonfunctions1} -- \eqref{operatorsonfunctions8} of this paper. We note that the operator $D_+$ was originally studied by Maass in his book~\cite{maassbook} in the case of scalar-valued forms. The operator $X_+$ (for both scalar and vector-valued forms) was already defined in \cite{BochererSatohYamazaki1992}, where it was called $\delta_{\ell+m}$. Also, the operator $U$ was considered by Satoh~\cite{satoh} (who called it $D$) in the very special case $m=2$. To the best of our knowledge, explicit formulas  for the operators (except in the cases mentioned above) had not been written out before this work.

More generally, if $\mathcal{X}_+$ denotes the free monoid consisting of finite strings of the above four operators, then each element $X \in  \mathcal{X}_+$ takes $N_{\ell,m}^p(\Gamma)$ to $N_{\ell_1,m_1}^{p_1}(\Gamma)$ for some integers $\ell_1, m_1, p_1$ (uniquely determined by $\ell$, $m$, $p$ and $X$) that can be easily calculated using the above table. In particular, the non-negative integer $v= p_1-p$ depends only on $X$; we call it the degree of $X$. For example, the operator $D_+^rU^s \in \mathcal{X}_+$ takes the space $N_{\ell,2s}^p(\Gamma)$ to $N_{\ell+2s+2r,0}^{2r+s+p}(\Gamma)$ and has degree $2r+s$.

Let $X$, $\ell$, $m$, $\ell_1$, $m_1$, $v$ be as above. We show that $X$ has the following properties.
\begin{enumerate}
 \item (Lemma~\ref{slashliecommute}) For all $\gamma \in \GSp_4(\R)^+$, we have
  $$
   (XF)|_{\ell_1, m_1}\gamma = X(F|_{\ell,m}\gamma).
  $$
 \item (Lemma~\ref{preserveeissiegel}) $X$ takes  $N_{\ell, m}(\Gamma)^\circ$ to $N_{\ell_1, m_1}(\Gamma)^\circ$ and takes the orthogonal complement of $N_{\ell, m}(\Gamma)^\circ$ to the orthogonal complement of $N_{\ell_1, m_1}(\Gamma)^\circ$.
 \item (Proposition \ref{peterssonequivprop}) There exists a constant $c_{\ell, m, X}$ (depending only on $\ell$, $m$, $X$) such that for all $F, G$ in $S_{\ell, m}(\Gamma)$,
  $$
   \langle X F, XG \rangle  = c_{\ell, m, X}\langle F , G\rangle.
  $$
 \item (Proposition \ref{proparithmeticitydiff}) For all $\sigma \in \Aut(\C)$, we have
  $$
   \leftexp{\sigma}( (2 \pi)^{-v}XF ) =  (2 \pi)^{-v}X(\leftexp{\sigma}F).
  $$
\end{enumerate}
We now state a coarse version of our structure theorem for cusp forms.
\begin{theorem}[Structure theorem for cusp forms, coarse version]\label{structuretheorem}
 Let $\ell, m$ be integers with $m \ge 0$. For each pair of integers $\ell', m'$, there is a (possibly empty\footnote{Indeed, $\mathcal{X}_{\ell', m'}^{\ell, m}$ is empty unless $m'\ge 0$, $0 \le \ell' \le \ell$, $0 \le \ell'+m' \le \ell+m$, and some additional parity conditions are satisfied. Moreover, $\mathcal{X}_{\ell, m}^{\ell, m}$ is always the singleton set consisting of the identity map whenever $\ell, m$ are non-negative integers.}) finite subset $\mathcal{X}_{\ell', m'}^{\ell, m}$ of $\mathcal{X}_+$ such that the following hold.

  \begin{enumerate}
   \item Each element $X \in  \mathcal{X}_{\ell', m'}^{\ell, m}$ acts injectively on $M_{\ell', m'}(\Gamma)$ and takes this space to $N_{\ell, m}(\Gamma)$.
   \item We have an orthogonal direct sum  decomposition
    $$
     N_{\ell,m}(\Gamma)^\circ=\bigoplus_{\ell'=1}^\ell
 \bigoplus_{m'=0}^{\ell+m-\ell'}\;\sum_{X \in \mathcal{X}_{\ell', m'}^{\ell, m}} X(S_{\ell',m'}(\Gamma)).
    $$
 \end{enumerate}
\end{theorem}
For the refined version of this result, see Theorem~\ref{cuspidalstructuretheorem}, which contains an exact description of the sets  $\mathcal{X}_{\ell', m'}^{\ell, m}$. We also formulate a version of this theorem for scalar valued cusp forms (Corollary~\ref{cuspidalstructuretheoremcor2}), as well as deduce a result for forms of a fixed nearly holomorphic degree (Corollary~\ref{cuspidalstructuretheoremcor3}).

Next, we turn to a structure theorem for the whole space, including the non-cusp forms. This situation turns out to be more complicated. Indeed, we need to now also include certain non-holomorphic objects among our building blocks. This is to be expected from the $n=1$ situation, where the nearly holomorphic Eisenstein series $E_2$ appears in the direct sum decomposition~\eqref{structuredegree1}.

For each $m \ge 0$, we define a certain subspace $M^*_{3,m}(\Gamma)$  of $N^1_{3,m}(\Gamma)$ consisting of forms that are annihilated by two differential operators that we call $L$ and $E_{-}$ (see Section \ref{s:nearholofn} for their explicit formulas). From the definition, it is immediate that $M^*_{3,m}(\Gamma)$ contains $M_{3,m}(\Gamma)$. However, it may potentially contain more objects.
These extra elements in $M^*_{3,m}(\Gamma)$ cannot exist if $M_{1,m}(\Gamma)=\{0\}$ (which is the case, for instance, when $\Gamma = \Sp_4(\Z))$; moreover, if they exist, they cannot be cuspidal, must lie inside $N^1_{3,m}(\Gamma)$, and cannot be obtained by applying our differential operators to holomorphic modular forms of any weight. Furthermore, we can prove that the space $M^*_{3,m}(\Gamma)$ is $\Aut(\C)$-invariant.

Now, we may state our general structure theorem as follows.

\begin{theorem}[Structure theorem for all modular forms, coarse version]\label{structuretheoremgeneralinto}
 Let $\ell, m$ be integers with $\ell>0$ and $m \ge 0$. For each pair of integers $\ell', m'$, let $\mathcal{X}_{\ell', m'}^{\ell, m}$ be as in Theorem~\ref{structuretheorem}. Then we have a  direct sum  decomposition
 $$
  N_{\ell,m}(\Gamma)=\bigoplus_{\substack{\ell'=1\\ \ell' \neq 3 }}^\ell\bigoplus_{m'=0}^{\ell+m-\ell'}\;\sum_{X \in \mathcal{X}_{\ell', m'}^{\ell, m}} X(M_{\ell',m'}(\Gamma))\oplus\bigoplus_{m'=0}^{\ell+m-3}\;\sum_{X \in \mathcal{X}_{\ell', m'}^{\ell, m}} X(M^*_{3,m'}(\Gamma)).
 $$
 This decomposition is orthogonal in the sense that forms lying in different constituents, and such that at least one of them is cuspidal, are orthogonal with respect to the Petersson inner product.
\end{theorem}
For a refined version of this result, see Theorem~\ref{noncuspidalstructuretheorem}. We note that the restriction to $\ell>0$ is not serious, since the only nearly holomorphic modular forms with $\ell\leq0$ are the constant functions.

\subsection{Lowest weight modules and \texorpdfstring{$\mathfrak{n}$}{}-finite automorphic forms}

We now describe the representation-theoretic results that form the foundation for Theorems \ref{structuretheorem} and \ref{structuretheoremgeneralinto}. We hope that they are of independent interest, as they explain nearly holomorphic forms from the point of view of representation theory.

Let $\g$ be the Lie algebra of $\Sp_4(\R)$, and let $\g_\C$ be its complexification. We fix a basis of the root system of $\g_\C$, and let $\mathfrak{n}$ be the space spanned by the non-compact negative roots. It is well known that vector-valued holomorphic modular forms $F$ correspond to (scalar-valued) automorphic forms\footnote{Here, and elsewhere in this paper, we use the term ``automorphic form" in the sense of Borel-Jacquet \cite[1.3]{BorelJacquet1979}; in particular, our automorphic forms are always \emph{scalar-valued} functions on $\Sp_4(\R)$. For the precise correspondence between (nearly holomorphic) vector-valued modular forms for $\Gamma$, and automorphic forms on $\Sp_4(\R)$ with respect to  $\Gamma$, see Lemma~\ref{FPhilemma} and Proposition~\ref{Nplmautformprop} of this paper.} $\Phi$ on $\Sp_4(\R)$ that are annihilated by $\mathfrak{n}$. The $(\g,K)$-module $\langle\Phi\rangle$ generated by such a $\Phi$ is a lowest weight module, and $\Phi$ is a lowest weight vector in this module. In fact, it will follow from our results that $\langle\Phi\rangle$ is always an \emph{irreducible} module (see Proposition~\ref{holomorphicsimplemodule}).

We define a vector $v$ in any representation of $\g_\C$ to be \emph{$\mathfrak{n}$-finite}, if the space $\mathcal{U}(\mathfrak{n})v$ is finite-dimensional; here $\mathcal{U}(\mathfrak{n})$ is the universal enveloping algebra of $\mathfrak{n}$, which in our case is simply a polynomial ring in three variables. Applying this concept to the space of automorphic forms on $\Sp_4(\R)$, we arrive at the notion of \emph{$\mathfrak{n}$-finite automorphic form}, which is central to this work. Let $\AA(\Gamma)_{\mathfrak{n}\text{-fin}}$ be the space of $\mathfrak{n}$-finite automorphic forms on $\Sp_4(\R)$ with respect to a fixed congruence subgroup $\Gamma$. Finiteness results from the classical theory imply that $\AA(\Gamma)_{\mathfrak{n}\text{-fin}}$ is an \emph{admissible} $(\g,K)$-module.

Clearly, the lowest weight module $\langle\Phi\rangle$ considered above is contained in $\AA(\Gamma)_{\mathfrak{n}\text{-fin}}$. We will prove the following:
\begin{itemize}
 \item Every automorphic form in $\AA(\Gamma)_{\mathfrak{n}\text{-fin}}$ gives rise to a vector-valued nearly holomorphic modular form on $\HH_2$. See Lemma \ref{FPhilemma}, Proposition \ref{AAnfindecomp2prop} and the discussion following it.
 \item Conversely, the automorphic form corresponding to a vector-valued nearly holomorphic modular form on $\HH_2$ lies in $\AA(\Gamma)_{\mathfrak{n}\text{-fin}}$. See Proposition \ref{Nplmautformprop}.
\end{itemize}
In other words, the $\mathfrak{n}$-finite automorphic forms correspond \emph{precisely} to nearly holomorphic modular forms. The lowest weight vectors in irreducible submodules of $\AA(\Gamma)_{\mathfrak{n}\text{-fin}}$ correspond precisely to holomorphic modular forms.

The structure theorems \ref{structuretheorem} and \ref{structuretheoremgeneralinto} are reflections of the fact that, in a lowest weight module appearing in $\AA(\Gamma)_{\mathfrak{n}\text{-fin}}$, we can navigate from the lowest weight vector to any given $K$-type using certain elements $X_+$, $U$, $E_+$ and $D_+$ in $\mathcal{U}(\g_\C)$ that correspond to the differential operators given in table \eqref{introtable}. See Proposition \ref{navigateKtypesprop} for the precise statement.

To prove statements like Proposition \ref{navigateKtypesprop}, we need rather precise information about $K$-types and multiplicities occurring in lowest weight modules. Such information is in principle available in the literature, but it requires some effort to obtain it from general theorems. It turns out that \emph{category $\mathcal{O}$} provides a framework well-suited for our purposes. More precisely, we will work in a parabolic version called category $\mathcal{O}^\p$, whose objects consist precisely of the finitely generated $(\g,K)$-modules in which all vectors are $\mathfrak{n}$-finite. This category thus contains all the lowest weight modules relevant for the study of $\mathfrak{n}$-finite automorphic forms.

Basic building blocks in category $\mathcal{O}^\p$ are the parabolic Verma modules $N(\lambda)$ and their unique irreducible quotients $L(\lambda)$; here, $\lambda$ is an integral weight.\footnote{The automorphic forms corresponding to elements in $M_{\ell, m}(\Gamma)$ generate the lowest weight module $L(\ell+m,\ell)$. We note that in previous papers, we have used the notation $\mathcal{E}(\ell+m,\ell)$ instead of $L(\ell+m, \ell)$.} We determine which of the $N(\lambda)$ are irreducible (Proposition \ref{Mlambdairredprop}), composition series in each reducible case (Proposition \ref{Mlambdaexactsequencesprop}), and which of the $L(\lambda)$ are square-integrable, tempered, or unitarizable (Proposition \ref{Mlambdaunitaryprop}). This is slightly more information than needed for our applications to automorphic forms, but we found it useful to collect all this information in one place.

By general principles, the admissible $(\g,K)$-module $\AA(\Gamma)_{\mathfrak{n}\text{-fin}}$ decomposes into a direct sum of indecomposable objects in category $\mathcal{O}^\p$. The subspace of cusp forms $\AA(\Gamma)_{\mathfrak{n}\text{-fin}}^\circ \subseteq \AA(\Gamma)_{\mathfrak{n}\text{-fin}}$ decomposes in fact into a direct sum of irreducibles $L(\lambda)$, due to the presence of an inner product. The multiplicities with which each $L(\lambda)$ occurs is given by the dimension of certain spaces of holomorphic modular forms. We can thus determine the complete algebraic structure of the space $\AA(\Gamma)_{\mathfrak{n}\text{-fin}}^\circ$ in terms of these dimensions. See Proposition \ref{AA0nfindecomp2prop} for the precise statement, which may be viewed as a precursor to Theorem \ref{structuretheorem}.

One cannot expect that the entire space $\AA(\Gamma)_{\mathfrak{n}\text{-fin}}$ also decomposes into a direct sum of irreducibles. This is already not the case in the degree $1$ situation, where the modular form $E_2$ generates an indecomposable but not irreducible module. Sections \ref{initialdecompsec} and \ref{nontemperedsec} are devoted to showing that only a very limited class of indecomposable but not irreducible modules can possibly occur in $\AA(\Gamma)_{\mathfrak{n}\text{-fin}}$. These modules account for the presence of the spaces $M^*_{3,m'}(\Gamma)$ in Theorem \ref{structuretheoremgeneralinto}. The algebraic structure of the entire space $\AA(\Gamma)_{\mathfrak{n}\text{-fin}}$ in terms of dimensions of spaces of modular forms is given in Proposition \ref{AAnfindecomp2prop}. As in the cuspidal case, this proposition is a precursor to the structure theorem.

\subsection{Applications of the structure theorem}
The significance of the structure theorem is twofold. On the one hand, it builds up the space of nearly holomorphic forms from holomorphic forms using differential operators. As the differential operators have nice arithmetic properties, this essentially reduces all arithmetic questions about nearly holomorphic  forms to the case of holomorphic forms. Since there is considerable algebraic geometry known for the latter, powerful results  can be obtained. For example, in Section~\ref{s:arithmeticity}, we show that the ``isotypic projection" map from $N_{\ell, m}(\Gamma)$ to $\sum_{X \in \mathcal{X}^{\ell, m}_{\ell', m'}}X(M_{\ell',m'}(\Gamma))$ (this is commonly called the ``holomorphic projection" map when $\ell=\ell'$, $m=m'$) obtained from our structure theorem is $\Aut(\C)$-equivariant (see Propositions~\ref{arithmeticityproj} and \ref{arithmeticitycusppro}). This is a considerable generalization of results of Shimura. We also prove a result on the arithmeticity of ratios of Petersson inner products (Proposition \ref{pterssonratio}) that will be of importance in our subsequent work.

On the other hand, sometimes one prefers to deal with modular forms of scalar weight, rather than vector-valued objects. The structure theorem gives an explicit and canonical way to start with an element of $M_{\ell,m}(\Gamma)$ (with $m$ even and $\ell\ge 2$) and produce a non-zero element of $N_{\ell+m}^{m/2}(\Gamma)$ lying in the same representation. (This does not work if $\ell=1$.)  On a related note, this will also allow one to write down a \emph{scalar-valued} nearly holomorphic lift in cases where previously only holomorphic vector-valued lifts have been considered (e.g., the Yoshida lift of two classical cusp forms $f$ and $g$ both of weight bigger than 2).

Both these points of view will be combined in  a forthcoming work where we will prove results in the spirit of Deligne's conjecture for the standard $L$-function attached to a holomorphic vector valued cusp form with respect to an arbitrary congruence subgroup of $\Sp_4(\Q)$. Such results have so far been proved (in the vector-valued case) only for forms of full level. The main new ingredient of this forthcoming work will be to consider an integral representation consisting only of scalar-valued nearly holomorphic vectors. The results of this paper will be key to doing that.

There are many other potential applications of this work, some of which we plan to pursue elsewhere. For example, one can use our structure theorems to produce exact formulas for the dimensions of spaces of nearly holomorphic modular forms; to the best of our knowledge, no such formulas are currently known in degree 2. One could try to see if our explicit formulas could be used to deal with problems related to congruences or the construction of $p$-adic measures for vector-valued Siegel modular forms, similar to what was done in the scalar-valued case in \cite{panchishkin2004}. One could apply our results to the study of nearly overconvergent modular forms for congruence subgroups of $\Sp_4(\Z)$, following the general framework of~\cite{Urban2013}. One could also explore applications of our work to arithmetic and combinatorial identities, \`{a} la \cite{gradechi}.

Finally, it would be interesting to generalize the results of this paper to the case $n>2$ and possibly to other groups. We hope to come back to this problem in the future.

\subsection*{Acknowledgements}
We would like to thank Siegfried B\"ocherer, Marcela Hanzer, Michael Harris, and Jonathan Kujawa for helpful discussions. We would also like to thank Shuji Horinaga, who had kindly informed us about a mistake in an earlier version of this paper.

\subsection*{Notation}
\begin{enumerate}
\item The symbols $\Z$, $\Z_{\ge0}$, $\Q$, $\R$, $\C$, $\Z_p$ and $\Q_p$ have the usual meanings. The symbol $\A$ denotes the ring of adeles of $\Q$, and $\A^\times$ denotes its group of ideles. We let $\mathfrak{f}$ denote the set of finite places, and $\A_{\mathfrak{f}}$ the subring of $\A$ with trivial archimedean component.

\item For any commutative ring $R$ and positive integer $n$, let $M_n(R)$ denote the ring of $n\times n$ matrices with entries in $R$, and let $M_n^{\rm sym}(R)$ denote the subset of symmetric matrices. We let $\GL_n(R)$ denote the group of invertible elements in $M_n(R)$, and we use $R^{\times}$ to denote $\GL_1(R)$. If $A\in M_n(R)$, we let $^t\!A$ denote its transpose.

\item Define $J_n\in M_n(\Z)$  by
  $J_n=\left[\begin{smallmatrix}0 & I_n\\-I_n & 0 \end{smallmatrix}\right]$, where $I_n$ is the $n$ by $n$ identity matrix. Let $\GSp_4$ and $\Sp_4$  be the algebraic groups whose $\Q$-points are given by
 \begin{align}
  \GSp_4(\Q)&=\{g \in \GL_4(\Q)\;|\; ^tgJ_2g = \mu_2(g)J_2,\:\mu_2(g)\in \Q^{\times} \},  \label{HFdefinition}\\
  \Sp_4(\Q)&=\{g \in \GSp_4(\Q)\;|\; \mu_2(g) = 1\}. \label{Gjdefinition}
 \end{align}
 Let $\GSp_4(\R)^+ \subset \GSp_4(\R)$ consist of the matrices with $\mu_2(g)>0$.

\item For $\tau=x+iy$, we let $$\frac{\partial}{\partial\tau}=\frac12\left(\frac{\partial}{\partial x}-i\frac{\partial}{\partial y}\right), \quad\frac{\partial}{\partial\bar\tau}=\frac12\left(\frac{\partial}{\partial x}+i\frac{\partial}{\partial y}\right)$$ denote the usual Wirtinger derivatives.

 \item The Siegel upper half space of degree $n$ is defined by
  $$
   \H_n= \{ Z \in M_n(\C) \ | \ Z = {}^tZ,\:i( \overline{Z} - Z) \text{ is positive definite}\}.
  $$
For $g=
\left[\begin{smallmatrix} A&B\\ C&D \end{smallmatrix}\right] \in \GSp_4(\R)^+$, $Z\in \H_2$, define $J(g,Z) = CZ + D.$
We let $I$ denote the element $\left[\begin{smallmatrix}i&\\&i\end{smallmatrix}\right]$ of $\H_2$.

 \item We let $\mathfrak{g}=\mathfrak{sp}_4(\R)$ be the Lie algebra of $\Sp_4(\R)$ and $\g_\C =\mathfrak{sp}_4(\C)$ the complexified Lie algebra.  We let $\mathcal{U}(\mathfrak{g}_\C)$ denote the  universal enveloping algebra and let $\mathcal{Z}$ be its center.
 We use the following basis for $\g_\C$.
\begin{alignat*}{2}
 & Z=-i\left[\begin{smallmatrix}0&0&1&0\\0&0&0&0\\-1&0&0&0\\0&0&0&0
 \end{smallmatrix}\right]
  ,\qquad\qquad&&Z'=-i\left[\begin{smallmatrix}0&0&0&0\\0&0&0&1\\0&0&0&0
  \\0&-1&0&0\end{smallmatrix}\right],\\
 &N_+=\frac12\left[\begin{smallmatrix}0&1&0&-i\\-1&0&-i&0\\
  0&i&0&1\\i&0&-1&0\end{smallmatrix}\right]
  ,\qquad\qquad&&N_-=\frac12\left[\begin{smallmatrix}
  0&1&0&i\\-1&0&i&0\\0&-i&0&1\\-i&0&-1&0\end{smallmatrix}\right],\\
 &X_+=\frac12\left[\begin{smallmatrix}1&0&i&0\\0&0&0&0\\i&0&-1&0\\
  0&0&0&0\end{smallmatrix}\right]
  ,\qquad\qquad&&X_-=\frac12\left[\begin{smallmatrix}
  1&0&-i&0\\0&0&0&0\\-i&0&-1&0\\0&0&0&0\end{smallmatrix}\right],\\
 &P_{1+}=\frac12\left[\begin{smallmatrix}0&1&0&i\\1&0&i&0\\
  0&i&0&-1\\i&0&-1&0\end{smallmatrix}\right]
  ,\qquad\qquad&&P_{1-}=\frac12\left[\begin{smallmatrix}
  0&1&0&-i\\1&0&-i&0\\0&-i&0&-1\\-i&0&-1&0\end{smallmatrix}\right],\\
 &P_{0+}=\frac12\left[\begin{smallmatrix}0&0&0&0\\0&1&0&i\\
  0&0&0&0\\0&i&0&-1\end{smallmatrix}\right]
  ,\qquad\qquad&&P_{0-}=\frac12\left[\begin{smallmatrix}
  0&0&0&0\\0&1&0&-i\\0&0&0&0\\0&-i&0&-1\end{smallmatrix}\right].
\end{alignat*}

 \item For all smooth functions $f : \Sp_4(\R) \rightarrow \C$,  $X \in \g$, define $(Xf)(g)=\frac{d}{dt}\big|_0f(\exp(tX))$. This action is extended $\C$-linearly to $\g_\C$. Further, it is extended to all elements $X \in \mathcal{U}(\mathfrak{g}_\C)$ in the usual manner.

 \end{enumerate}

\section{Lowest weight representations}\label{s:lowest}
In this section we study the lowest weight representations of the Hermitian symmetric pair $(G,K)$, where $G=\Sp_4(\R)$ and $K$ is its maximal compact subgroup. We will determine composition series and $K$-types for each parabolic Verma module. Of course, lowest weight representations have been extensively studied in the literature, in the more general context of semisimple Lie groups. Much of our exposition will consist in making the general theorems explicit in our low-rank case.
\subsection{Set-up and basic facts}\label{setupsec}
The subgroup $K$ of $\Sp_4(\R)$ consisting of all elements of the form $\left[\begin{smallmatrix}A&B\\-B&A\end{smallmatrix}\right]$ is a maximal compact subgroup. It is isomorphic to $U(2)$ via the map $\left[\begin{smallmatrix}A&B\\-B&A\end{smallmatrix}\right]\mapsto A+iB$.

Let $\mathfrak{g}=\mathfrak{sp}_4(\R)$ be the Lie algebra of $\Sp_4(\R)$, which we think of as a $10$-dimensional space of $4\times4$ matrices. Let $\mathfrak{k}$ be the Lie algebra of $K$; it is a four-dimensional subspace of $\mathfrak{g}$. Let $\mathfrak{g}_\C$ (resp.\ $\mathfrak{k}_\C$) be the complexification of $\mathfrak{g}$ (resp.\ $\mathfrak{k}$). A Cartan subalgebra $\mathfrak{h}_\C$ of $\mathfrak{k}_\C$ (and of $\mathfrak{g}_\C$) is spanned by the two elements $Z$ and $Z'$. If $\lambda$ is in the dual space $\mathfrak{h}_\C^*$, we identify $\lambda$ with the element $(\lambda(Z),\lambda(Z'))$ of $\C^2$. The root system of $\mathfrak{g}_\C$ is $\Phi=\{(\pm2,0),(0,\pm2),(\pm1,\pm1),(\pm1,\mp1)\}$. These vectors lie in the subspace $E:=\R^2$ of $\C^2$, which we think of as a Euclidean plane. The
analytically integral elements of $\mathfrak{h}_\C^*$ are those that identify with points of $\Z^2$. These are exactly the points of the weight lattice $\Lambda$. The following diagram indicates the weight lattice, as well as the roots and the elements of the Lie algebra spanning the corresponding root spaces.
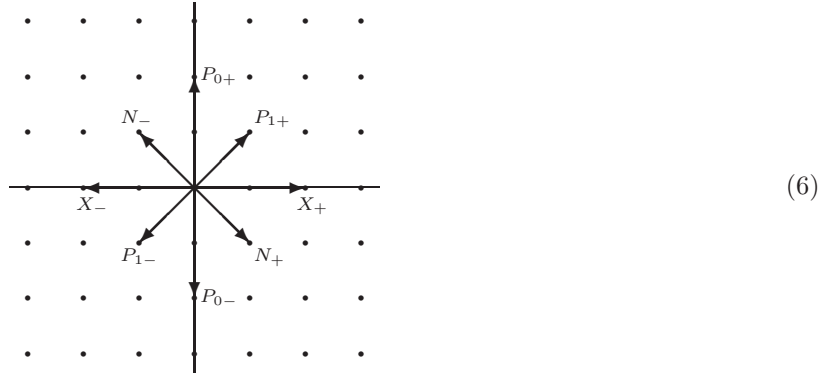
\begin{equation}\label{rootdiagrameq}
\begin{minipage}{47ex}
\setlength{\unitlength}{0.70pt}
\begin{picture}(200, 200)(-100, -100)
\put(-100,0){\line(1,0){200}}
\put(0,100){\line(0,-1){200}}

\thicklines
\put(0,0){\vector(1,0){60}}
\put(0,0){\vector(-1,0){60}}
\put(0,0){\vector(0,1){60}}
\put(0,0){\vector(0,-1){60}}
\put(0,0){\vector(1,1){30}}
\put(0,0){\vector(-1,-1){30}}
\put(0,0){\vector(1,-1){30}}
\put(0,0){\vector(-1,1){30}}

\put(55,-12){$\scriptstyle X_+$}
\put(-64,-12){$\scriptstyle X_-$}
\put(32,-40){$\scriptstyle N_+$}
\put(-40,35){$\scriptstyle N_-$}
\put(3,58){$\scriptstyle P_{0+}$}
\put(3,-62){$\scriptstyle P_{0-}$}
\put(32,35){$\scriptstyle P_{1+}$}
\put(-40,-40){$\scriptstyle P_{1-}$}

\put(-90,90){\circle*{3}}
\put(-60,90){\circle*{3}}
\put(-30,90){\circle*{3}}
\put(0,90){\circle*{3}}
\put(30,90){\circle*{3}}
\put(60,90){\circle*{3}}
\put(90,90){\circle*{3}}

\put(-90,60){\circle*{3}}
\put(-60,60){\circle*{3}}
\put(-30,60){\circle*{3}}
\put(0,60){\circle*{3}}
\put(30,60){\circle*{3}}
\put(60,60){\circle*{3}}
\put(90,60){\circle*{3}}

\put(-90,30){\circle*{3}}
\put(-60,30){\circle*{3}}
\put(-30,30){\circle*{3}}
\put(0,30){\circle*{3}}
\put(30,30){\circle*{3}}
\put(60,30){\circle*{3}}
\put(90,30){\circle*{3}}

\put(-90,0){\circle*{3}}
\put(-60,0){\circle*{3}}
\put(-30,0){\circle*{3}}
\put(0,0){\circle*{3}}
\put(30,0){\circle*{3}}
\put(60,0){\circle*{3}}
\put(90,0){\circle*{3}}

\put(-90,-30){\circle*{3}}
\put(-60,-30){\circle*{3}}
\put(-30,-30){\circle*{3}}
\put(0,-30){\circle*{3}}
\put(30,-30){\circle*{3}}
\put(60,-30){\circle*{3}}
\put(90,-30){\circle*{3}}

\put(-90,-60){\circle*{3}}
\put(-60,-60){\circle*{3}}
\put(-30,-60){\circle*{3}}
\put(0,-60){\circle*{3}}
\put(30,-60){\circle*{3}}
\put(60,-60){\circle*{3}}
\put(90,-60){\circle*{3}}

\put(-90,-90){\circle*{3}}
\put(-60,-90){\circle*{3}}
\put(-30,-90){\circle*{3}}
\put(0,-90){\circle*{3}}
\put(30,-90){\circle*{3}}
\put(60,-90){\circle*{3}}
\put(90,-90){\circle*{3}}

\end{picture}
\end{minipage}
\end{equation}
Here, $(1,-1)$ and $(-1,1)$ are the compact roots, with the corresponding root spaces being spanned by $N_+$ and $N_-$. We declare the set
$$
 \Phi^+=\{(-2,0),(-1,-1),(0,-2),(1,-1)\}
$$
to be a positive system of roots. We define an ordering on $\Lambda$ by
\begin{equation}\label{dominanceordereq2}
 \mu\preccurlyeq\lambda\quad\Longleftrightarrow\quad\lambda\in\mu+\Upsilon,
\end{equation}
where $\Upsilon$ is the set of all $\Z_{\geq0}$-linear combinations of elements of $\Phi^+$. Hence, under this ordering, $(0,-2)$ is maximal among the non-compact positive roots.

Let $\mathcal{Z}$ be the center of the universal enveloping algebra $\mathcal{U}(\mathfrak{g}_\C)$. A particular element in $\mathcal{Z}$ is the Casimir element
\begin{align}\label{casimireq1}
 \Omega_2&=\frac12Z^2+\frac12Z'^2-\frac12(N_+N_-+N_-N_+)+X_+X_-+X_-X_+\nonumber\\
  &\hspace{10ex}+\frac12(P_{1+}P_{1-}+P_{1-}P_{1+})+P_{0+}P_{0-}+P_{0-}P_{0+}.
\end{align}
Using the commutation relations, an alternative form is
\begin{align}\label{casimireq2}
 \Omega_2&=\frac12Z^2+\frac12Z'^2-Z-2Z'-N_-N_++2X_+X_-+P_{1+}P_{1-}+2P_{0+}P_{0-}.
\end{align}
The characters of $\mathcal{Z}$ are indexed by elements of $\mathfrak{h}_\C^*$ modulo Weyl group action; see Sects.\ 1.7--1.10 of \cite{Humphreys2008}. Let $\chi_\lambda$ be the character of $\mathcal{Z}$ corresponding to $\lambda\in\mathfrak{h}_\C^*$. We normalize this correspondence such that $\chi_\varrho$ is the trivial character (i.e., the central character of the trivial representation of $\mathcal{U}(\mathfrak{g}_\C)$); here, $\varrho=(-1,-2)$ is half the sum of the positive roots. Note that Humphrey's $\chi_\lambda$ is our $\chi_{\lambda+\varrho}$.

If $\mathfrak{k}_\C$ acts on a space $V$, and $v\in V$ satisfies $Zv=kv$ and $Z'v=\ell v$ for $k,\ell\in\C$, then we say that $v$ has \emph{weight} $(k,\ell)$. If the weight lies in $E$, we indicate it as a point in this Euclidean plane. Let $V$ be a finite-dimensional $\mathfrak{k}_\C$-module. Then this representation of $\mathfrak{k}_\C$ can be integrated to a representation of $K$ if and only if all occurring weights are analytically integral. The isomorphism classes of irreducible such $\mathfrak{k}_\C$-modules, or the corresponding irreducible representations of $K$, are called \emph{$K$-types}.

Let $V$ be a $K$-type. A non-zero vector $v\in V$ is called a \emph{highest weight vector} if $N_+v=0$. Such a vector $v$ is unique up to scalars. Let $(k,\ell)$ be its weight. Then the weights occurring in $V$ are $(k-j,\ell+j)$ for $j=0,1,\ldots,k-\ell$. In particular, the dimension of $V$ is $k-\ell+1$. If we associate with each $K$-type its highest weight, then we obtain a bijection between $K$-types and analytically integral elements $(k,\ell)$ with $k\geq\ell$.

\begin{definition}We let $\Lambdap$ denote the subset of $\Lambda$ consisting of pairs of integers $(k, \ell)$ with $k \geq \ell$. If $\lambda \in \Lambdap$, we denote by $\rho_\lambda$ the corresponding $K$-type.\end{definition}

Let $\mathfrak{p}_\pm=\langle X_\pm,P_{1\pm},P_{0\pm}\rangle$. Then $\mathfrak{p}_+$ and $\mathfrak{p}_-$ are commutative subalgebras of $\mathfrak{g}_\C$, and we have $[\mathfrak{k}_\C,\mathfrak{p}_\pm]\subset\mathfrak{p}_\pm$. Let $\rho_\lambda$ be a $K$-type. Let $F(\lambda)$ be any model for $\rho_\lambda$. We consider $F(\lambda)$ a module for $\mathfrak{k}_\C+\mathfrak{p}_-$ by letting $\mathfrak{p}_-$ act trivially. Let
\begin{equation}\label{Mlambdadefeq}
 N(\lambda):=\mathcal{U}(\mathfrak{g}_\C)\otimes_{\mathfrak{k}_\C+\mathfrak{p}_-}F(\lambda).
\end{equation}
Then $N(\lambda)$ is a $\mathfrak{g}_\C$-module in the obvious way. It also is a $(\mathfrak{g},K)$-module, with $K$-action given by $g.(X\otimes v)={\rm Ad}(g)(X)\otimes\rho_\lambda(g)v$ for $g\in K$, $X\in\mathcal{U}(\mathfrak{g}_\C)$ and $v\in F(\lambda)$. The modules $N(\lambda)$ are often called \emph{highest weight modules} in the literature. However, when we think of the $K$-type $\rho_\lambda$ as the weight of a modular form, it will be more natural to think of the $N(\lambda)$ as \emph{lowest} weight modules.

As vector spaces, we have
\begin{equation}\label{tildeEpsppluseq}
 N(\lambda)=\mathcal{U}(\mathfrak{p}_+)\otimes_\C F(\lambda).
\end{equation}
Since $\mathcal{U}(\mathfrak{p}_+)$ is simply a polynomial algebra in $X_+,P_{1+},P_{0+}$, it follows that $N(\lambda)$ is spanned by the vectors
\begin{equation}\label{tildeEpsspaneq}
 X_+^\alpha\,P_{1+}^\beta\,P_{0+}^\gamma\,N_-^\delta\,w_0,\qquad\alpha,\beta,\gamma,\delta\geq0,\;\delta\leq k-\ell,
\end{equation}
where $\lambda=(k,\ell)$, and these vectors are linearly independent. Here, $w_0$ is a highest weight vector in $F(\lambda)$ (identified with the element $1\otimes w_0$ in the tensor product \eqref{Mlambdadefeq}). Alternatively, $N(\lambda)$ is spanned by the vectors
\begin{equation}\label{tildeEpsspaneq2}
 N_-^\delta\,X_+^\alpha\,P_{1+}^\beta\,P_{0+}^\gamma\,w_0,\qquad\alpha,\beta,\gamma,\delta\geq0,
\end{equation}
but these are not linearly independent.

It will be convenient to work in a parabolic version of category $\mathcal{O}$; see Sect.~9 of \cite{Humphreys2008}. Let $\mathfrak{n}=\langle X_-,P_{1-},P_{0-}\rangle$; this is the same as $\mathfrak{p}_-$, but we will use the symbol $\mathfrak{n}$ henceforth. Let $M$ be a $\mathfrak{g}_\C$-module. We say $M$ lies in category $\mathcal{O}^{\mathfrak{p}}$ if it satisfies the following conditions:
\begin{enumerate}
 \item[($\mathcal{O}^{\mathfrak{p}}$1)] $M$ is a finitely generated $\mathcal{U}(\mathfrak{g}_\C)$-module.
 \item[($\mathcal{O}^{\mathfrak{p}}$2)] $M$ is the direct sum of $K$-types.
 \item[($\mathcal{O}^{\mathfrak{p}}$3)] $M$ is locally $\mathfrak{n}$-finite, meaning: For each $v\in M$ the subspace $\mathcal{U}(\mathfrak{n})v$ is finite-dimensional.
\end{enumerate}

Recall that, by definition, all the weights occurring in a $K$-type are analytically integral. It follows that all the weights occurring in any module in category $\mathcal{O}^{\mathfrak{p}}$ are integral.

Evidently, the modules $N(\lambda)$ defined in \eqref{Mlambdadefeq} satisfy these conditions. In fact, they are nothing but the \emph{parabolic Verma modules} defined in Sect.~9.4 of \cite{Humphreys2008}. From the theory developed there, we have the following basic properties of the modules $N(\lambda)$.

\begin{enumerate}
 \item Each weight of $N(\lambda)$ occurs with finite multiplicity. These multiplicities can be determined from \eqref{tildeEpsspaneq}.
 \item $N(\lambda)$ contains the $K$-type $\rho_\lambda$ with multiplicity one.
 \item The module $N(\lambda)$ has the following universal property: Let $M$ be a $(\mathfrak{g},K)$-module which contains a vector $v$ such that:
  \begin{itemize}
   \item $M=\mathcal{U}(\mathfrak{g}_\C)v$;
   \item $v$ has weight $\lambda$;
   \item $v$ is annihilated by $\langle X_-, P_{1-},P_{0-},N_+\rangle$.
  \end{itemize}
  Then there is a surjection $N(\lambda)\to M$ mapping a highest weight vector in $N(\lambda)$ to $v$.
 \item $N(\lambda)$ admits a unique irreducible submodule, and a unique irreducible quotient $L(\lambda)$. In particular, $N(\lambda)$ is indecomposable.
 \item $N(\lambda)$ has finite length. Each factor in a composition series is of the form $L(\mu)$ for some $\mu\preccurlyeq\lambda$.
 \item $N(\lambda)$ admits a central character, given by $\chi_{\lambda+\varrho}$.  Here, as before, $\varrho=(-1,-2)$ is half the sum of the positive roots.
 \item $L(\lambda)$ is finite-dimensional if and only if $\lambda=(k,\ell)$ with $0\geq k\geq\ell$.
\end{enumerate}
The modules $M$ in $\mathcal{O}^{\mathfrak{p}}$ enjoy properties analogous to those in category $\mathcal{O}$. In particular:
\begin{itemize}
 \item $M$ has finite length, and admits a filtration
  \begin{equation}\label{categoryOpeq1}
   0=V_0\subset V_1\subset\ldots\subset V_n\subset M,
  \end{equation}
   with $V_i/V_{i-1}\cong L(\lambda)$ for some $\lambda\in\Lambdap$.
 \item $M$ can be written as a finite direct sum of indecomposable modules.
 \item If $M$ is an indecomposable module, then there exists a character $\chi$ of $\mathcal{Z}$ such that $M=M(\chi)$. Here,
  \begin{equation}\label{Mchidefeq}
   M(\chi)=\{v\in M\:|\:(z-\chi(z))^nv=0\text{ for some $n$ depending on $z$}\}.
  \end{equation}

\end{itemize}

The following result, which is deeper, follows from standard classification theorems. Its last part will imply that cusp forms must have positive weight (see Proposition \ref{AA0nfindecomp2prop}).

\begin{proposition}\label{Mlambdaunitaryprop}
 Let $\lambda=(k,\ell)\in\Lambdap.$
 \begin{enumerate}
  \item $L(\lambda)$ is square-integrable if and only if $\ell\geq3$.
  \item $L(\lambda)$ is tempered if and only if $\ell\geq2$.
  \item $L(\lambda)$ is unitarizable if and only if $\ell\geq1$ or $(k,\ell)=(0,0)$.
 \end{enumerate}
\end{proposition}
\begin{proof}
(1) follows from the classification of discrete series representations; see Theorem 12.21 of \cite{Knapp1986}. (The $L(\lambda)$ with $\ell\geq3$ are precisely the holomorphic discrete series representations.)

(2) follows from the classification of tempered representations; see Theorem 8.5.3 of \cite{Knapp1986}. (The $L(\lambda)$ with $\ell=2$ are precisely the limits of holomorphic discrete series representations.)

For a more explicit description of these classifications in the case of $\Sp_4(\R)$, see
\cite{Muic2009}.

(3) follows from the classification of unitary highest weight modules; see \cite{Jakobsen1983}, \cite{EnrightHoweWallach1983} or \cite{EnrightJoseph1990}. We omit the details.
\end{proof}

\begin{lemma}\label{locnfinLlambdalemma}
 The only irreducible, locally $\mathfrak{n}$-finite $(\mathfrak{g},K)$-modules are the $L(\lambda)$ for $\lambda\in\Lambdap.$
\end{lemma}
\begin{proof}
Let $R$ be a locally $\mathfrak{n}$-finite $(\mathfrak{g},K)$-module. Then  $R$ lies in category $\mathcal{O}^\p$. By~\eqref{categoryOpeq1}, $R$ has a finite composition series with the quotients being $L(\lambda)$'s. So if $R$ is irreducible, it must be an $L(\lambda)$.
\end{proof}

\begin{lemma}\label{casimirNlambdalemma}
 Let $\lambda=(k,\ell)\in\Lambdap$. The Casimir operator $\Omega_2$ defined in \eqref{casimireq1} acts on $N(\lambda)$, and hence on $L(\lambda)$, by the scalar $\frac12(k(k-2)+\ell(\ell-4))$.
\end{lemma}
\begin{proof}
Since $\Omega_2$ lies in the center of $\mathcal{U}(\mathfrak{g}_\C)$, it is enough to prove that $\Omega_2w_0=\frac12(k(k-2)+\ell(\ell-4))w_0$, where $w_0$ is a vector of weight $(k,\ell)$. This follows from \eqref{casimireq2}.
\end{proof}

\subsection{Reducibilities and \texorpdfstring{$K$}{}-types}
In this section we will determine composition series for each of the modules $N(\lambda)$, and determine the $K$-types of each $N(\lambda)$ and $L(\lambda)$.
\begin{proposition}\label{Mlambdairredprop}
 Let $\lambda=(k,\ell)\in\Lambdap$. Then $N(\lambda)$ is irreducible if and only if one of the following conditions is satisfied:
 \begin{enumerate}
  \item $\ell\geq2$.
  \item $k=1$.
  \item $k+\ell=3$.
 \end{enumerate}
 Hence $N(\lambda)$ is irreducible if and only if $\lambda$ corresponds to one of the blackened points in the following diagram:
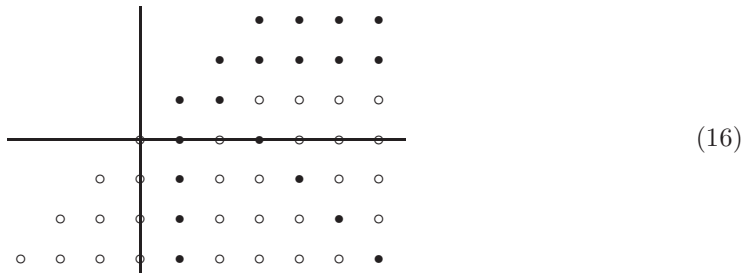
\setlength{\unitlength}{0.5px}
\begin{equation}\label{Mlambdairredproppic}
\begin{minipage}{60ex}
\begin{center}
\begin{picture}(290, 200)(-100, -100)
\put(-100,0){\line(1,0){300}}
\put(0,100){\line(0,-1){200}}

\put(90,90){\circle*{6}}
\put(120,90){\circle*{6}}
\put(150,90){\circle*{6}}
\put(180,90){\circle*{6}}

\put(60,60){\circle*{6}}
\put(90,60){\circle*{6}}
\put(120,60){\circle*{6}}
\put(150,60){\circle*{6}}
\put(180,60){\circle*{6}}

\put(30,30){\circle*{6}}
\put(60,30){\circle*{6}}
\put(90,30){\circle{6}}
\put(120,30){\circle{6}}
\put(150,30){\circle{6}}
\put(180,30){\circle{6}}

\put(0,0){\circle{6}}
\put(30,0){\circle*{6}}
\put(60,0){\circle{6}}
\put(90,0){\circle*{6}}
\put(120,0){\circle{6}}
\put(150,0){\circle{6}}
\put(180,0){\circle{6}}

\put(-30,-30){\circle{6}}
\put(0,-30){\circle{6}}
\put(30,-30){\circle*{6}}
\put(60,-30){\circle{6}}
\put(90,-30){\circle{6}}
\put(120,-30){\circle*{6}}
\put(150,-30){\circle{6}}
\put(180,-30){\circle{6}}

\put(-60,-60){\circle{6}}
\put(-30,-60){\circle{6}}
\put(0,-60){\circle{6}}
\put(30,-60){\circle*{6}}
\put(60,-60){\circle{6}}
\put(90,-60){\circle{6}}
\put(120,-60){\circle{6}}
\put(150,-60){\circle*{6}}
\put(180,-60){\circle{6}}

\put(-90,-90){\circle{6}}
\put(-60,-90){\circle{6}}
\put(-30,-90){\circle{6}}
\put(0,-90){\circle{6}}
\put(30,-90){\circle*{6}}
\put(60,-90){\circle{6}}
\put(90,-90){\circle{6}}
\put(120,-90){\circle{6}}
\put(150,-90){\circle{6}}
\put(180,-90){\circle*{6}}
\end{picture}
\end{center}
\end{minipage}
\end{equation}

\end{proposition}
\begin{proof}
Most cases can be handled by Theorem 9.12 in \cite{Humphreys2008}. The condition (*) in this theorem translates into $\ell\geq2$. Hence, by part a) of the theorem, $N(\lambda)$ is irreducible if $\ell\geq2$, and by part b) of the theorem, $N(\lambda)$ is reducible if $\ell\leq1$ and $\lambda+\varrho$ is regular (does not lie on a wall).

Hence consider $\lambda$ with $\ell\leq1$ and $\lambda+\varrho$ singular. Then either $\lambda=(1,\ell)$ with $\ell\leq1$ or $\lambda=(x+1,-x+2)$ with $x\geq1$. In the second case it is clear that no $L(\lambda')$ with $\lambda'\neq\lambda$ has the same central character as $N(\lambda)$; thus $N(\lambda)$ is irreducible. In the case that $\lambda=(1,\ell)$ with $\ell\leq1$ we may use Theorem 9.13 in \cite{Humphreys2008} (Jantzen's simplicity criterion) to see that $N(\lambda)$ is irreducible, as follows. The set of roots $\Psi_\lambda^+$ defined on page 198 of \cite{Humphreys2008} is easily calculated to be $\{(-1,-1),(0,-2)\}$. Hence, by Corollary 9.13 a) of \cite{Humphreys2008}, $N(\lambda)$ is irreducible if and only if
\begin{equation}\label{Mlambdairredpropeq1}
 \theta(s_{(-1,-1)}\cdot\lambda)+\theta(s_{(0,-2)}\cdot\lambda)=0.
\end{equation}
Here, the dot action of the Weyl group is defined in Sect.~1.8 of \cite{Humphreys2008}, and
\begin{equation}\label{Mlambdairredpropeq2}
 \theta(\mu)=\text{ch}(M(\mu))-\text{ch}(M(s_{(1,-1)}\cdot\mu))
\end{equation}
for any $\mu\in\Lambda^+$. It is straightforward to calculate that the characters involved on the left hand side of \eqref{Mlambdairredpropeq1} cancel each other out, so that \eqref{Mlambdairredpropeq1} is satisfied.
\end{proof}

We see from \eqref{Mlambdairredproppic} that the $\lambda=(k,\ell)$, $k\geq\ell$, for which $N(\lambda)$ is reducible fall into one of three regions:
\begin{itemize}
 \item Region A: $k\leq0$; these are the dominant integral weights.
 \item Region B: $k\geq2$ and $k+\ell\leq2$.
 \item Region C: $\ell\leq1$ and $k+\ell\geq4$.
\end{itemize}
In addition, we will consider
\begin{itemize}
 \item Region D: $\ell\geq3$.
\end{itemize}
Note that the disjoint union of Regions A -- D comprises precisely the regular integral weights with $k\geq\ell$.

The \emph{dot action} of an element $w$ of the Weyl group $W$ on $\lambda\in\mathfrak{h}_\C^*$ is defined by $w\cdot\lambda=w(\lambda+\varrho)-\varrho$, where on the right side we have the usual action of $W$ via reflections, and where $\varrho=(-1,-2)$ is half the sum of the positive roots. Let $s_1\in W$ be the reflection corresponding to the short simple root, and let $s_2\in W$ be the reflection corresponding to the long simple root. Explicitly, $s_1(x,y)=(y,x)$ and $s_2(x,y)=(-x,y)$. Under the dot action, we have
\begin{equation}\label{w1w2w3ABCDeq}
 s_2\cdot\text{A}=\text{B},\qquad s_2s_1\cdot\text{A}=\text{C},\qquad s_2s_1s_2\cdot\text{A}=\text{D},
\end{equation}
 where we wrote ``A'' for ``Region A'', etc. Consequently, $s_2s_1s_2\cdot\text{B}=\text{C}$ and $s_1s_2s_1\cdot\text{C}=\text{D}$.
\begin{proposition}\label{Mlambdaexactsequencesprop}
 Let $\lambda=(k,\ell)\in \Lambdap$.
 \begin{enumerate}
  \item Assume that $\lambda$ is in Region A. Then there is an exact sequence
   $$
    0\longrightarrow L(s_2\cdot\lambda)\longrightarrow N(\lambda)\longrightarrow L(\lambda)\longrightarrow0.
   $$
   The weight $s_2\cdot\lambda=(-k+2,\ell)$ is in Region B.
  \item Assume that $\lambda$ is in Region B. Then there is an exact sequence
   $$
    0\longrightarrow L(s_2s_1s_2\cdot\lambda)\longrightarrow N(\lambda)\longrightarrow L(\lambda)\longrightarrow0.
   $$
   The weight $s_2s_1s_2\cdot\lambda=(-\ell+3,-k+3)$ is in Region C.
  \item Assume that $\lambda$ is in Region C. Then there is an exact sequence
   $$
    0\longrightarrow L(s_1s_2s_1\cdot\lambda)\longrightarrow N(\lambda)\longrightarrow L(\lambda)\longrightarrow0,
   $$
   The weight $s_1s_2s_1\cdot\lambda=(k,-\ell+4)$ is in Region D.
 \end{enumerate}
\end{proposition}
\begin{proof}
In this proof we will make use of the fact that a composition series for any $N(\lambda)$ is multiplicity free, i.e., each $L(\mu)$ can occur at most once as a subquotient in such a series. This fact is generally true for Hermitian symmetric pairs $(\mathfrak{g},\mathfrak{p})$ and other pairs for which $\mathfrak{p}$ is maximal parabolic; see \cite{BoeCollingwood1986} or \cite{BoeCollingwood1990}.

We first prove (3). Thus, assume that $\lambda$ is in Region C. By general properties, each factor in a composition series of $N(\lambda)$ is of the form $L(\mu)$ for some $\mu\preccurlyeq\lambda$. Also, $N(\lambda)$ and $L(\mu)$ have the same central character, which is equivalent to $\lambda$ and $\mu$ being in the same $W$-orbit under the dot action. The only $\mu$ satisfying these properties, other than $\lambda$ itself, is $s_1s_2s_1\cdot\lambda=(k,-\ell+4)$. Since $N(\lambda)$ is reducible by Proposition \ref{Mlambdairredprop}, the module $L(s_1s_2s_1\cdot\lambda)$ occurs at least once in a composition series for $N(\lambda)$. By multiplicity one, $L(s_1s_2s_1\cdot\lambda)$ occurs exactly once. The assertion follows.

To prove (1) and (2), assume that $\lambda$ is in Region A. By Theorem 9.16 of \cite{Humphreys2008}, there is an exact sequence
\begin{align}\label{Mlambdaexactsequencespropeq1}
  &0\longrightarrow N(s_2s_1s_2\cdot\lambda)\longrightarrow N(s_2s_1\cdot\lambda)\nonumber\\
   &\hspace{10ex}\longrightarrow N(s_2\cdot\lambda)\longrightarrow N(\lambda)\longrightarrow L(\lambda)\longrightarrow0.
\end{align}
Note that $N(s_2s_1s_2\cdot\lambda)=L(s_2s_1s_2\cdot\lambda)$ by Proposition \ref{Mlambdairredprop}. By the already proven part (3), we get an exact sequence
\begin{equation}\label{Mlambdaexactsequencespropeq2}
  0\longrightarrow L(s_2s_1\cdot\lambda)\longrightarrow N(s_2\cdot\lambda)\longrightarrow N(\lambda)\longrightarrow L(\lambda)\longrightarrow0.
\end{equation}
It follows that $N(\lambda)$ and $N(s_2\cdot\lambda)$ have the same length. By central character considerations and multiplicity one, the length of $N(s_2\cdot\lambda)$ can be at most $3$. Hence the common length of $N(\lambda)$ and $N(s_2\cdot\lambda)$ is $2$ or $3$, and our proof, of both (1) and (2), will be complete if we can show this length is $2$.

By Proposition 9.14 of \cite{Humphreys2008}, the socle of $N(\lambda)$ is simple. It follows that the length of $N(\lambda)$ coincides with its Loewy length. We are thus reduced to showing that the Loewy length of $N(\lambda)$ is $2$.

For this we will employ Theorem 4.3 of \cite{Irving1985}. In the notation of this paper we have $^SW_\lambda=\{1,s_2,s_2s_1,s_2s_1s_2\}$. From \eqref{Mlambdaexactsequencespropeq1} and \eqref{Mlambdaexactsequencespropeq2} we conclude that the set $^SX_\lambda$ determining the \emph{socular weights} contains at least $s_2s_1$ and $s_2s_1s_2$. It follows from \eqref{Mlambdaexactsequencespropeq2} that $N(\lambda)$ does not contain $L(s_2s_1\cdot\lambda)$ in its composition series. Hence $(s_2s_1)^\vee=s_2$, where $w^\vee$ for $w\in\,\!^SX_\lambda$ is defined on p.~734 of \cite{Irving1985}. A computation of the elements $\overline{w}$, defined on p.~743 of \cite{Irving1985} for $w\in\,^SW_\lambda$, yields
$$
 \overline{1}=1,\qquad\overline{s_2}=s_2,\qquad\overline{s_2s_1}=s_1,\qquad\overline{s_2s_1s_2}=s_2.
$$
Thus the number $t$ appearing in Theorem 4.3 of \cite{Irving1985}, defined as the maximal length of any $\overline{w}$, is $1$. The hypothesis of this theorem is satisfied (set $x=s_2s_1$). The theorem implies that the Loewy length of any $N(w\cdot\lambda)$ for $w\in\,\! ^SW_\lambda$ is at most $2$. In particular, the Loewy length of $N(\lambda)$ is $2$, concluding our proof.
\end{proof}

We will next determine the $K$-types in $N(\lambda)$. Let $V$ be any admissible $(\mathfrak{g},K)$-module. For a weight $\lambda\in\Lambda$, let $V_\lambda$ be the corresponding weight space. We denote by
\begin{equation}\label{Epsweightseq1}
 m_\lambda(V)=\dim V_\lambda
\end{equation}
the multiplicity of the weight $\lambda$ in $V$. Let
\begin{equation}\label{Epsweightseq2}
 {\rm mult}_\lambda(V)=\text{ the multiplicity of the $K$-type $\rho_\lambda$ in V}.
\end{equation}
It follows from the weight structure of the $K$-types that
\begin{equation}\label{Epsweightseq3}
 {\rm mult}_\lambda(V)=m_\lambda(V)-m_{\lambda+(1,-1)}(V).
\end{equation}
Let $Q(\lambda)$ be the number of ways to write $\lambda\in\Lambda$ as a $\Z_{\geq0}$ linear combination of $(2,0)$, $(1,1)$ and $(0,2)$. It is easy to see that, for $\lambda=(x,y)$ with integers $x,y$,
\begin{equation}\label{Qfunctioneq2}
 Q(x,y)=\begin{cases}
                  \displaystyle\Big\lfloor\frac{\min(x,y)+2}2\Big\rfloor&\text{if $x,y\geq0$ and $x\equiv y$ mod $2$},\\[2ex]
                  0&\text{otherwise}.
                 \end{cases}
\end{equation}
\begin{lemma}\label{MlambdaKtypeslemma}
 Let $\lambda=(k,\ell)\in \Lambdap$. Let $x,y$ be integers with $x\geq y$. Then
   $$
    {\rm mult}_{(x,y)}(N(\lambda))=0\quad\text{if $x<k$, or $y<\ell$, or $x-y\not\equiv k-\ell$ mod $2$}.
   $$
   If $x\geq k$ and $y\geq\ell$ and $x-y\equiv k-\ell$ mod $2$, then
   \begin{align*}
    &{\rm mult}_{(x,y)}(N(\lambda))=\begin{cases}
      \displaystyle\Big\lfloor\frac{\min(x-k,y-\ell)+2}2\Big\rfloor&\text{if }y\leq k,\\[2ex]
      \displaystyle\Big\lfloor\frac{\min(x-k,y-\ell)}2\Big\rfloor-\Big\lfloor\frac{y-k-1}2\Big\rfloor&\text{if }y>k.
     \end{cases}
   \end{align*}
\end{lemma}
\begin{proof}
It follows from \eqref{tildeEpsspaneq} that
\begin{equation}\label{Epsweightseq4}
 m_{(x,y)}(N(\lambda))=\sum_{n=0}^{k-\ell}Q(x-k+n,y-\ell-n).
\end{equation}
The result now follows from \eqref{Epsweightseq3} and some simplification.
\end{proof}

Proposition \ref{Mlambdaexactsequencesprop} combined with Lemma \ref{MlambdaKtypeslemma} allows us to calculate the multiplicities of the $K$-types of any $L(\lambda)$. Note also that if $\lambda$ is not in Region A, B or C, then $L(\lambda)=N(\lambda)$ by Proposition \ref{Mlambdairredprop}, so that Lemma \ref{MlambdaKtypeslemma} can be used directly to calculate the multiplicities. We record a few special cases in the following result; the details of the elementary proofs are omitted.

\begin{proposition}\label{EpsKtypesprop}
 Let $x\geq y$ be integers.
 \begin{enumerate}
  \item Assume that $\lambda=(\ell,\ell)$ with an integer $\ell\geq1$. Then
   $$
    {\rm mult}_{(x,y)}(L(\lambda))=0\quad\text{if $x<\ell$, or $y<\ell$, or $x\not\equiv y$ mod $2$}.
   $$
   If $x\geq\ell$ and $y\geq\ell$ and $x\equiv y$ mod $2$, then
   $$
    {\rm mult}_{(x,y)}(L(\lambda))=
     \begin{cases}
      1&\text{if }y\equiv\ell\;{\rm mod}\;2,\\
      0&\text{if }y\not\equiv\ell\;{\rm mod}\;2.
     \end{cases}
   $$
  \item Assume that $\lambda=(k,1)$ with an integer $k\geq2$. Then
   $$
    {\rm mult}_{(x,y)}(L(\lambda))=0\quad\text{if $x<k$, or $y<1$, or $x-y\not\equiv k-1$ mod $2$}.
   $$
   If $x\geq k$ and $y\geq 1$ and $x-y\equiv k-1$ mod $2$, then
   $$
    {\rm mult}_{(x,y)}(L(\lambda))
     =\begin{cases}
       1&\text{if }y\leq x-k+1,\\
       0&\text{if }y>x-k+1.
      \end{cases}
   $$
  \item Assume that $\lambda=(k,\ell)$ is in Region C. Then ${\rm mult}_{(x,y)}(L(\lambda))=0$ if $y\geq x-k-\ell+4$. Hence, all the $K$-types of $L(\lambda)$ are strictly below the diagonal line running through the point $(k,-\ell+4)$.
  \item Assume that $\lambda=(2,0)$. Then
   $$
    {\rm mult}_{(x,y)}(L(\lambda))
     =\begin{cases}
       1&\text{if $x\geq2$, $y\geq0$ and $x\equiv y\equiv0$ mod $2$},\\
       0&\text{otherwise}.
      \end{cases}
   $$
 \end{enumerate}
\end{proposition}

Finally, we consider the location of the boundary $K$-types in the modules $N(\lambda)$ for any $\lambda=(k,\ell)$ with $k\geq\ell$. By Lemma \ref{MlambdaKtypeslemma}, all the boundary $K$-types occur with multiplicity one. There are no $K$-types $\rho_{(x,y)}$ for $x<k$ or $y<\ell$. For $x=k$ or $y=\ell$ the $K$-types occur in steps of $2$. The top boundary is provided by the line $y=x$ if $k\equiv\ell$ mod $2$, or the line $y=x-1$ if $k\not\equiv\ell$ mod $2$. In the first case, the $K$-types on this line occur in steps of $2$, in the second case in steps of $1$. The following diagrams illustrate these two cases.

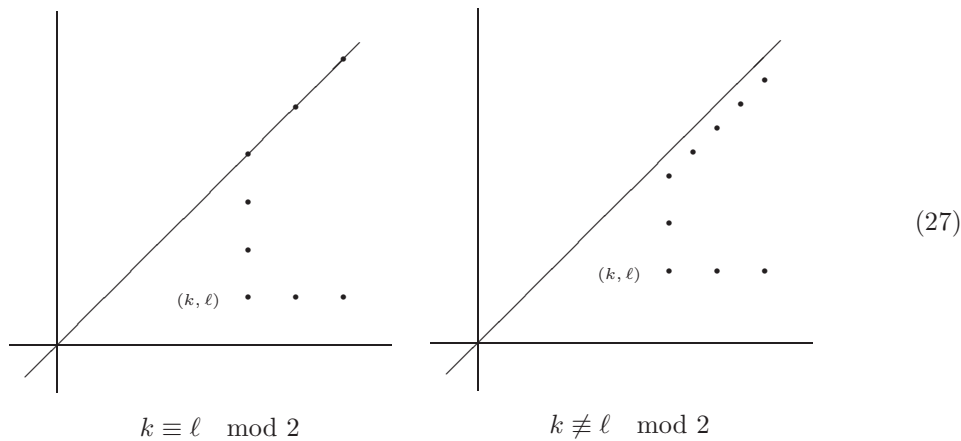
\begin{equation}\label{boundaryKtypesdiagram}
\begin{minipage}{37ex}
\setlength{\unitlength}{0.6px}
\begin{picture}(240, 240)(-30,-30)
\put(-30,0){\line(1,0){240}}
\put(0,-30){\line(0,1){240}}
\put(-20,-20){\line(1,1){210}}

\put(180,180){\circle*{3}}

\put(150,150){\circle*{3}}

\put(120,120){\circle*{3}}

\put(120,90){\circle*{3}}

\put(120,60){\circle*{3}}

\put(120,30){\circle*{3}}
\put(150,30){\circle*{3}}
\put(180,30){\circle*{3}}

\put(75,25){\tiny$(k,\ell)$}
\end{picture}
$$
 k\equiv\ell\mod2
$$
\end{minipage}
\begin{minipage}{35ex}
\setlength{\unitlength}{0.6px}
\begin{picture}(240, 240)(-30,-30)
\put(-30,0){\line(1,0){240}}
\put(0,-30){\line(0,1){240}}
\put(-20,-20){\line(1,1){210}}

\put(180,165){\circle*{3}}

\put(165,150){\circle*{3}}

\put(150,135){\circle*{3}}

\put(135,120){\circle*{3}}

\put(120,105){\circle*{3}}

\put(120,75){\circle*{3}}

\put(120,45){\circle*{3}}
\put(150,45){\circle*{3}}
\put(180,45){\circle*{3}}

\put(75,40){\tiny$(k,\ell)$}
\end{picture}
$$
 k\not\equiv\ell\mod2
$$
\end{minipage}
\end{equation}
\subsection{Navigating the highest weight vectors}
Let $V$ be a $(\mathfrak{g},K)$-module. In this section we will investigate a collection of elements of $\mathcal{U}(\mathfrak{g}_\C)$ that preserve the property of being a highest weight vector in some $K$-type. In other words, these elements $X$ will have the property that $N_+Xv=0$ if $N_+v=0$. Evidently, elements that commute with $N_+$, like $X_+$ and $P_{0-}$, have this property.

\begin{table}[h]
\caption{Some elements of $\mathcal{U}(\mathfrak{g}_\C)$ that take highest weight vectors to highest weight vectors. The last column shows the resulting weight after applying an operator to a vector of weight $(\ell+m,\ell)$.}\label{Nplusoperatorstable}
$$\renewcommand{\arraystretch}{1.5}\renewcommand{\arraycolsep}{0.3cm}
 \begin{array}{cccc}
  \toprule
   \text{name}&\text{definition}&\text{new weight}\\
  \toprule
   P_{0-}&&(\ell+m,\ell-2)\\
  \midrule
   L&m(m-1)X_--(m-1)P_{1-}N_-+P_{0-}N_-^2&(\ell+m-2,\ell)\\
  \midrule
   U&m(m-1)P_{0+}+(m-1)P_{1+}N_-+X_+N_-^2&(\ell+m,\ell+2)\\
  \midrule
   X_+&&(\ell+m+2,\ell)\\
  \midrule
   E_{+}&(m+2)P_{1+}+2N_-X_+&(\ell+m+1,\ell+1)\\
  \midrule
   E_{-}&(m+2)P_{1-}-2N_-P_{0-}&(\ell+m-1,\ell-1)\\
  \midrule
   D_+&P_{1+}^2-4X_+P_{0+}&(\ell+m+2,\ell+2)\\
  \midrule
   D_-&P_{1-}^2-4X_-P_{0-}&(\ell+m-2,\ell-2)\\
  \bottomrule
 \end{array}
$$
\end{table}

More specifically, we consider a vector $v \in V$ of weight $(\ell+m,\ell)$ for $m \ge 0$. The new elements of $\mathcal{U}(\mathfrak{g}_\C)$  that we introduce are called $U$, $L$, $E_{+}$, $E_{-}$, $D_{+}$, and $D_{-}$. Their definitions appear in Table \ref{Nplusoperatorstable}. These operators take $v$ to another vector in $V$ of the weight indicated in the ``new weight" column. Note that the operators  $U$, $L$, $E_{+}$ and $E_{-}$ depend on $m$. However, for brevity, our notation will not reflect this dependence. The formulas for the operators $U$ and $L$ are given only for $m\ge 2$; we adopt the convention that $U=L=0$ if $m<2$.

\begin{remark}As was pointed out by the referee, the element $Z-Z'$ of $\mathcal{U}(\mathfrak{g}_\C)$ acts on a vector of weight $(\ell+m,\ell)$ by multiplication by $m$; therefore the operators in Table \ref{Nplusoperatorstable} can be defined as single elements of $\mathcal{U}(\mathfrak{g}_\C)$ without needing to involve the integer $m$. This leads to the alternate definitions below, which are equivalent to the ones given in Table \ref{Nplusoperatorstable}:
\begin{align*}L &= X_{-}(Z-Z')(Z-Z'-1) - P_{1-}N_-(Z-Z'-1) + P_{0-}N_-^2,\\
U &= P_{0+}(Z-Z')(Z-Z'-1)+P_{1+}N_-(Z-Z'-1)+X_+N_-^2,\\
E_{+} &= P_{1+}(Z-Z'+2)+2N_-X_+,\\
E_{-} &= P_{1-}(Z-Z'+2)-2N_-P_{0-}.
\end{align*}
\end{remark}

\begin{lemma}\label{Nplusoperatorslemma}
 Let $\ell$ be an integer, and $m$ a non-negative integer. Let $v$ be a vector of weight $(\ell+m,\ell)$ in some $(\mathfrak{g},K)$-module $V$. Let $X\in\mathcal{U}(\mathfrak{g}_\C)$ be one of the elements in Table \ref{Nplusoperatorstable}. Then $N_+Xv=0$ if $N_+v=0$. The weight of $Xv$ is indicated in the last column of Table \ref{Nplusoperatorstable}. For the $U$ and $L$ operators we assume $m\geq2$.
\end{lemma}
\begin{proof}
 In order to show that $N_+Xv=0$ if $N_+v=0$, it suffices to show that $N_+ X \in \mathcal{U}(\mathfrak{g}_\C) N_+$. This is easily verified using the commutation relations. The weight  of $Xv$ as indicated in the last column of Table \ref{Nplusoperatorstable} is also immediate from the commutation relations.
\end{proof}

As we already mentioned, $[N_+,X_+]=[N_+,P_{0-}]=0$. The two-step diagonal operators $D_\pm$ have in fact the property that $[N_+,D_\pm]=[N_-,D_\pm]=0$. The other operators in Table \ref{Nplusoperatorstable} do not universally commute with $N_+$. Using the commutation relations, one may further verify that
\begin{align}
 \label{Empluscommuteeq1}X_+E_{+}&=E_{+}X_+,\\
 \label{Empluscommuteeq2}UE_{+}&=E_{+}U,\\
 \label{Empluscommuteeq3}D_+E_{+}&=E_{+}D_+,\\
 \label{Empluscommuteeq4}UD_+&=D_+U,\\
 \label{Empluscommuteeq5}X_+U-UX_+&=(m+1)D_+.
\end{align}
We remind the reader that the particular element in $\mathcal{U}(\mathfrak{g}_\C)$ that each operator appearing in the above equations corresponds to, depends on the weight of the vector that the operator will act on. For instance, consider both sides of~\eqref{Empluscommuteeq1} acting on a vector of weight $(\ell+m, \ell)$. Then the $E_{+}$ on the left side is given by the formula in Table \ref{Nplusoperatorstable} while the $E_{+}$ on the right side is obtained by the substitution $m \mapsto m+2$ in the same formula.

Now consider a weight $\lambda=(\ell+m,\ell)$ with $\ell\in\Z$ and $m\geq0$. By Lemma \ref{MlambdaKtypeslemma}, if a $K$-type $\rho_{(x,y)}$ occurs in $N(\lambda)$, then $x\geq\ell+m$ and $y\geq\ell$. We may therefore hope to generate all highest weight vectors in the $K$-types of $N(\lambda)$ by applying appropriate powers of the operators $X_+$, $D_+$, $U$ and $E_+$ to the lowest weight vector $w_0$ of $N(\lambda)$. We will see below that this is indeed the case.

As a first step in this direction, consider the $K$-types $\rho_{(x,y)}$ with $x=\ell+m$; these are the ones that are straight above the minimal weight. By Lemma \ref{MlambdaKtypeslemma}, these are exactly the $K$-types $(\ell+m,\ell+2i)$, $i\in\{0,1,\ldots,\lfloor\frac m2\rfloor\}$, and each of these occurs with multiplicity $1$ in $N(\lambda)$. Let $w_0$ be a lowest weight vector (i.e., a highest weight vector in the minimal $K$-type of $N(\lambda)$); thus, $w_0$ has weight $(\ell+m,\ell)$, and $N_+w_0=0$. For $i\in\{0,1,\ldots,\lfloor\frac m2\rfloor\}$, let\footnote{Once again, we remind the reader that the operator $U^i$ in \eqref{P0minusUmlemmaeq0} is really shorthand for $U_{m+2-2i}\,\ldots\,U_{m-2}\,U_m$; i.e., the integer $m$ appearing in the definition of $U$ changes at each step.}
\begin{equation}\label{P0minusUmlemmaeq0}
 w_i=U^iw_0.
\end{equation}
Then $w_i$ has weight $(\ell+m,\ell+2i)$, and $N_+w_i=0$. If $w_i\neq0$, then it is a highest weight vector in the $K$-type $\rho_{(\ell+m,\ell+2i)}$ of $N(\lambda)$.
\begin{lemma}\label{P0minusUmlemma}
 With the above notations,
 $$
  P_{0-}w_{i+1}=-(i+1)(\ell+i-1)(m-2i)(m-2i-1)\,w_i.
 $$
 for $i\in\{0,1,\ldots,\lfloor\frac m2\rfloor-1\}$. In particular, if $\ell\geq2$, then $w_i\neq0$ for all $i\in\{0,1,\ldots,\lfloor\frac m2\rfloor\}$.
\end{lemma}

\begin{lemma}\label{Dminuslemma}
 Suppose $\lambda = (\ell+m, \ell)$ with $m\ge 0$, $m$ even, and $\ell\ge 1$. If $\ell=1$, assume further that  $m=0$. Let $w_0$ be a non-zero vector of weight $(\ell+m, \ell)$ in $N(\lambda)$ such that $N_{+}w_0 = 0$. Then, for all $\beta \ge 0$, $P_{0-}^{m/2}D_{-}^\beta D_{+}^\beta U^{m/2} w_0$ is a non-zero multiple of $w_0$.
\end{lemma}
The above two lemmas can be proved by routine calculation.

Before stating the next result, it will be convenient to introduce the concept of \emph{$N_-$-layers}. Let $\lambda=(\ell+m,\ell)\in\Lambda$ with $\ell\in\Z$ and $m\geq0$. Given a non-negative integer $\delta$, the $\delta$-th $N_-$-layer of $N(\lambda)$, denoted by $N(\lambda)^\delta$, is defined as the subspace spanned by all vectors of the form
\begin{equation}\label{tildeEpsspaneq4}
 X_+^\alpha\,P_{1+}^\beta\,P_{0+}^\gamma\,N_-^\delta\,w_0,\qquad\alpha,\beta,\gamma\geq0.
\end{equation}
Here, as before, $w_0$ is a fixed non-zero vector of weight $\lambda$. Note that $N(\lambda)^\delta=0$ for $\delta>m$. By \eqref{tildeEpsspaneq}, we have $N(\lambda)=N(\lambda)^0\oplus\ldots\oplus N(\lambda)^m$. We also introduce the notation $N(\lambda)^{\leq\delta}=N(\lambda)^0\oplus\ldots\oplus N(\lambda)^\delta$. Observe that, since $N_-$ normalizes $\p_+=\langle P_{0+},P_{1+},X_+\rangle$, in any expression involving these four operators we may always move the $N_-$'s to the right. In fact,
\begin{equation}\label{Nminuslayerseq}
 N_-Yw_0\in YN_-w_0+N(\lambda)^\delta\qquad\text{for }Y\in\mathcal{U}(\mathfrak{p}_+)N_-^\delta.
\end{equation}
It follows that the operator $N_-$ maps $N(\lambda)^\delta$ to $N(\lambda)^\delta\oplus N(\lambda)^{\delta+1}$. In particular, $N_-$ induces an endomorphism of the top layer $N(\lambda)^m$.

\begin{lemma}\label{plusoperatorsinjectivelemma}
 Let $\lambda=(\ell+m,\ell)\in\Lambdap$.
 \begin{enumerate}
  \item Let $f\in\C[X,Y,Z]$ be a non-zero polynomial. Then the element $f(X_+,P_{1+},P_{0+})$ of $\mathcal{U}(\mathfrak{g}_\C)$ acts injectively on $N(\lambda)$, and it preserves $N_-$-layers.
  \item The restriction of ${E}_+$ to $N(\lambda)^{\leq(m-1)}$ is injective.
 \end{enumerate}
\end{lemma}
\begin{proof}
(1) is immediate from \eqref{tildeEpsppluseq}. (2) follows easily from \eqref{Nminuslayerseq} and the defining formula $E_{+}=(m+2)P_{1+}+2N_-X_+$.
\end{proof}
\begin{lemma}\label{wiDplusXpluslemma}
 Let $\lambda=(\ell+m,\ell)\in\Lambda$ with $\ell\geq2$ and $m\geq0$. Let the vectors $w_i\in N(\lambda)$ be defined as in \eqref{P0minusUmlemmaeq0}. Then the vectors
 \begin{equation}\label{wiDplusXpluslemmaeq1}
  X_+^\alpha D_+^\beta w_i,\qquad\alpha,\beta\geq0,\:i\in\Big\{0,1,\ldots,\Big\lfloor\frac m2\Big\rfloor\Big\},
 \end{equation}
 are linearly independent.
\end{lemma}
\begin{proof}
First note that the $w_i$ are non-zero by Lemma \ref{P0minusUmlemma}. By Lemma \ref{plusoperatorsinjectivelemma} (1), all the vectors \eqref{wiDplusXpluslemmaeq1} are non-zero. We see from the defining formula for the $U$ operator in Table \ref{Nplusoperatorstable} that $X_+^\alpha D_+^\beta w_i$ lies in $N(\lambda)^{\leq2i}$, but not in $N(\lambda)^{\leq(2i-1)}$. It follows that any linear combination between the vectors \eqref{wiDplusXpluslemmaeq1} can only involve a single $i$. But for fixed $i$ the vectors \eqref{wiDplusXpluslemmaeq1} have distinct weights as $\alpha$ and $\beta$ vary. Our assertion follows.
\end{proof}

Recall from Lemma \ref{MlambdaKtypeslemma} that if a $K$-type $\rho_{(x,y)}$ occurs in $N(\lambda)$, where $\lambda=(\ell+m,\ell)$, then $x-y\equiv m$ mod $2$. We say that such a $K$-type is of \emph{parity $0$} if $x\equiv\ell+m$ mod $2$ and $y\equiv\ell$ mod $2$. Otherwise, if $x\not\equiv\ell+m$ mod $2$ and $y\not\equiv\ell$ mod $2$, we say the $K$-type is of \emph{parity $1$}. We apply the same terminology to the highest weight vectors of such $K$-types. Clearly, the operators $X_+$, $P_{0-}$, $ {U}$, $ {L}$ and $D_\pm$ preserve the parity, while $ {E}_\pm$ change the parity. Let $N(\lambda)_{\text{par}(0)}$ (resp.\ $N(\lambda)_{\text{par}(1)}$) be the subspace of $N(\lambda)$ spanned by highest weight vectors of parity $0$ (resp.\ parity $1$). We now state the main result of this section.
\begin{proposition}\label{navigateKtypesprop}
 Let $\lambda=(\ell+m,\ell)\in\Lambdap$ with $\ell\geq2$ and $m\geq0$.
 \begin{enumerate}
  \item $N(\lambda)_{\text{\rm par}(0)}$ is precisely the space spanned by the vectors \eqref{wiDplusXpluslemmaeq1}.
  \item If $m$ is odd, then the map $ {E}_+:\:N(\lambda)_{\text{\rm par}(0)}\to N(\lambda)_{\text{\rm par}(1)}$ is an isomorphism.
  \item If $m$ is even, then the map $ {E}_+:\:N(\lambda)_{\text{\rm par}(0)}\to N(\lambda)_{\text{\rm par}(1)}$ is surjective, and its kernel is spanned by the vectors \eqref{wiDplusXpluslemmaeq1} with $i=m/2$.
 \end{enumerate}
\end{proposition}
\begin{proof}
(1) Clearly, the highest weight vectors \eqref{wiDplusXpluslemmaeq1} all have parity $0$. By easy combinatorics we can determine the number of vectors \eqref{wiDplusXpluslemmaeq1} of a fixed weight $(x,y)$. Comparing with the formula from Lemma \ref{MlambdaKtypeslemma}, we see that this number coincides with ${\rm mult}_{(x,y)}(N(\lambda))$. This proves (1) in view of the linear independence of the vectors \eqref{wiDplusXpluslemmaeq1}.

(2) If $m$ is odd, then the vectors \eqref{wiDplusXpluslemmaeq1} are all contained in $N(\lambda)^{\leq(m-1)}$. Hence $ {E}_+:\:N(\lambda)_{\text{par}(0)}\to N(\lambda)_{\text{par}(1)}$ is injective by part (1) and Lemma \ref{plusoperatorsinjectivelemma} (2). To prove surjectivity, it is enough to show that ${\rm mult}_{(x,y)}(N(\lambda))={\rm mult}_{(x-1,y-1)}(N(\lambda))$ for all $(x,y)$ of parity $1$. This follows from the formula in Lemma \ref{MlambdaKtypeslemma}.

(3) Assume that $m$ is even. The vector $w_{m/2}$ has weight $(\ell+m,\ell+m)$. By Lemma \ref{MlambdaKtypeslemma}, the $K$-type $\rho_{(\ell+m+1,\ell+m+1)}$ is not contained in $N(\lambda)$; see also \eqref{boundaryKtypesdiagram}. Hence $ {E}_+w_{m/2}=0$. By \eqref{Empluscommuteeq1} - \eqref{Empluscommuteeq3} , $ {E}_+$ annihilates all vectors \eqref{wiDplusXpluslemmaeq1} with $i=m/2$. The vectors \eqref{wiDplusXpluslemmaeq1} with $i<m/2$ are all contained in $N(\lambda)^{\leq(m-1)}$. Therefore, the assertion about the kernel of $ {E}_+$ follows from part (1) and Lemma \ref{plusoperatorsinjectivelemma} (2).

To prove the surjectivity assertion, first note that, by Lemma \ref{MlambdaKtypeslemma},
$$
 {\rm mult}_{(x,y)}(N(\lambda))=
  \begin{cases}
   {\rm mult}_{(x-1,y-1)}(N(\lambda))&\text{if }y\leq\ell+m,\\
   {\rm mult}_{(x-1,y-1)}(N(\lambda))-1&\text{if }y>\ell+m,
  \end{cases}
$$
for all $K$-types $\rho_{(x,y)}$ of parity $1$. The $K$-type $\rho_{(x-1,y-1)}$ of parity $0$ receives a contribution from a vector \eqref{wiDplusXpluslemmaeq1} with $i=m/2$ if and only if $y>\ell+m$. The surjectivity therefore follows by what we already proved about the kernel of $ {E}_+$.
\end{proof}

\begin{remark}
Let $N(\lambda)_{\rm hw}=N(\lambda)_{\text{\rm par}(0)}\oplus N(\lambda)_{\text{\rm par}(1)}$ be the subspace of highest weight vectors, and let $\mathcal{I}$ be the subalgebra of $\mathcal{U}(\mathfrak{p}_+)$ generated by $X_+$ and $D_+$. Then Proposition 2.14 implies that $N(\lambda)_{\rm hw}$ is a free $\mathcal{I}$-module of rank $m+1$, the dimension of the minimal $K$-type. For holomorphic discrete series representations, i.e., for $\ell\geq3$, this also follows from the main result of \cite{HoweKraft1998}.
\end{remark}

\subsubsection*{The case of lowest weight $(1+m,1)$}
In Proposition \ref{navigateKtypesprop} we assumed $\ell\geq2$ since otherwise some of the vectors $w_i$ might be zero; see Lemma \ref{P0minusUmlemma}. However, for later applications we also require the following analogous result for the $L(\lambda)$ with $\lambda=(1+m,1)$.
\begin{proposition}\label{navigateKtypesell1prop}
 Let $\lambda=(1+m,1)$ with $m\geq0$. Let $w_0$ be a non-zero vector of weight $(1+m,1)$ in $L(\lambda)$.
 \begin{enumerate}
  \item $L(\lambda)_{\text{\rm par}(0)}$ is precisely the space spanned by the vectors
   \begin{equation}\label{navigateKtypesell1propeq1}
    X_+^\alpha D_+^\beta w_0,\qquad \alpha,\beta\geq0.
   \end{equation}
  \item If $m\geq1$, then the map $ {E}_+:\:L(\lambda)_{\text{\rm par}(0)}\to L(\lambda)_{\text{\rm par}(1)}$ is an isomorphism. If $m=0$, then $L(\lambda)_{\text{\rm par}(1)}=0$.
 \end{enumerate}
\end{proposition}
\begin{proof}
The proof is similar to the previous case, and is omitted.
\end{proof}
\section{Differential operators}\label{s:diff}
\subsection{Functions on the group and functions on \texorpdfstring{$\H_2$}{}}\label{fctsgrpH2sec}
Recall that $K\cong U(2)$ via $\left[\begin{smallmatrix}A&B\\-B&A\end{smallmatrix}\right]\mapsto A+iB$. On the Lie algebra level, this map induces an isomorphism $\mathfrak{k}\cong\mathfrak{u}(2)$ given by the same formula. Extending this map $\C$-linearly, we get an isomorphism $\mathfrak{k}_\C\cong\mathfrak{gl}_2(\C)$. Under this isomorphism,
\begin{equation}\label{ku2eq}
 Z\longmapsto\mat{1}{0}{0}{0},\quad
 Z'\longmapsto\mat{0}{0}{0}{1},\quad
 N_+\longmapsto\mat{0}{1}{0}{0},\quad
 N_-\longmapsto\mat{0}{0}{-1}{0}.
\end{equation}
Let $\ell$ be an integer, and $m$ a non-negative integer. Let $W_m \simeq \sym^m(\C^2)$ be the space of all complex homogeneous polynomials of total degree $m$ in the two variables $S$ and $T$. For any $g \in \GL_2(\C)$, and $P(S,T) \in W_m$, define $\eta_{\ell,m}(g)P(S,T) = \det(g)^\ell P((S,T)g)$. Then $(\eta_{\ell,m}, W_m)$ gives a concrete realization of the irreducible representation $\det^\ell\sym^m$  of $\GL_2(\C)$. We will denote the derived representation of $\mathfrak{gl}_2(\C)$ by the same symbol $\eta_{\ell,m}$. Easy calculations show that, under the identification \eqref{ku2eq},
\begin{align}
 \label{etalmeq1}\eta_{\ell,m}(Z)S^{m-j}T^j&=(\ell+m-j)S^{m-j}T^j,\\
 \label{etalmeq2}\eta_{\ell,m}(Z')S^{m-j}T^j&=(\ell+j)S^{m-j}T^j,\\
 \label{etalmeq3}\eta_{\ell,m}(N_+)S^{m-j}T^j&=jS^{m-j+1}T^{j-1},\\
 \label{etalmeq4}\eta_{\ell,m}(N_-)S^{m-j}T^j&=-(m-j)S^{m-j-1}T^{j+1}.
\end{align}
In particular, $\eta_{\ell,m}(N_+)S^m=0$ and $\eta_{\ell,m}(N_-)T^m=0$. Since the vector $S^m$ is a highest weight vector of weight $(\ell+m,\ell)$, we see that
\begin{equation}\label{rhohatrhoeq1}
 \text{The restriction of $\eta_{\ell,m}$ to $U(2)$ is $\rho_{(\ell+m,\ell)}$}.
\end{equation}

For a smooth function $\Phi$ on $\Sp_4(\R)$ of weight $(\ell+m,\ell)$ satisfying $N_+ \Phi=0$, we define a function $\vec\Phi$ taking values in the polynomial ring $\C[S,T]$ by
\begin{equation}\label{vecPhidefeq}
 \vec\Phi(g)=\sum_{j=0}^m\frac{(-1)^j}{j!}(N_-^j\Phi)(g)S^{m-j}T^j,\qquad g\in\Sp_4(\R).
\end{equation}
Evidently, $\vec\Phi$ takes values in the space $W_m\subset\C[S,T]$ of the representation $\eta_{\ell,m}$. Hence, an expression like $\eta_{\ell,m}(h)(\vec\Phi(g))$ makes sense, for any $h\in\GL_2(\C)$.

In the following lemma, for clarity of notation, we let $\iota$ be the transposition map on $2\times2$ complex matrices. We may interpret $\iota$ as an anti-involution of $\GL_2(\C)$. The derived map, also given by transposition and also denoted by $\iota$, is an anti-involution of $\mathfrak{gl}_2(\C)$. It extends to an anti-involution of the algebra $\mathcal{U}(\mathfrak{gl}_2(\C))$. When we write $\iota(h)$ for $h\in K$, we mean $\iota$ applied to the element of $U(2)$ corresponding to $h\in K$ via the map $\left[\begin{smallmatrix}A&B\\-B&A\end{smallmatrix}\right]\mapsto A+iB$.
\begin{lemma}\label{vecPhitransformlemma}
 Let $\ell$ be any integer, and $m$ a non-negative integer. Let $\Phi$ be a $K$-finite function on $\Sp_4(\R)$ of weight $(\ell+m,\ell)$ satisfying $N_+\Phi=0$ (right translation action). Let $\vec\Phi$ be the polynomial-valued function defined in \eqref{vecPhidefeq}. Then
 \begin{equation}\label{vecPhitransformlemmaeq0}
  \vec\Phi(gh)=\eta_{\ell,m}(\iota(h))(\vec\Phi(g)),\qquad\text{for }h\in K
 \end{equation}
 and $g\in\Sp_4(\R)$. On the Lie algebra level,
 \begin{equation}\label{vecPhitransformlemmaeq1}
  (X\vec\Phi)(g)=\eta_{\ell,m}(\iota(X))(\vec\Phi(g))
 \end{equation}
 for $X\in\mathcal{U}(\mathfrak{k}_\C)$ and $g\in\Sp_4(\R)$. More generally,
 \begin{equation}\label{vecPhitransformlemmaeq2}
  (YX\vec\Phi)(g)=\eta_{\ell,m}(\iota(X))((Y\vec\Phi)(g))
 \end{equation}
 for $X\in\mathcal{U}(\mathfrak{k}_\C)$, $Y\in\mathcal{U}(\mathfrak{g}_\C)$ and $g\in\Sp_4(\R)$.
\end{lemma}
\begin{proof}
Fixing $g\in\Sp_4(\R)$, we first claim that \eqref{vecPhitransformlemmaeq1} holds for $X\in\mathfrak{k}_\C$. In fact, this assertion is easily verified using the formulas \eqref{etalmeq1} -- \eqref{etalmeq4}. For $X=N_+$ the identity
$$
 N_+N_-^j=N_-^jN_++jN_-^{j-1}(Z'-Z)+j(j-1)N_-^{j-1}
$$
is helpful.

Replacing $g$ by $g\exp(tY)$ and taking $\frac{d}{dt}\big|_0$ on both sides, one proves that \eqref{vecPhitransformlemmaeq1} also holds for elements of degree $2$ in $\mathcal{U}(\mathfrak{k}_\C)$. Continuing in this manner, we see that \eqref{vecPhitransformlemmaeq1} holds for any element $X\in\mathcal{U}(\mathfrak{k}_\C)$. Now using that $\exp((d\eta)(X))=\eta(\exp(X))$ for any representation $\eta$ and $X\in\mathfrak{k}$, one can derive the identity \eqref{vecPhitransformlemmaeq0}.

To prove \eqref{vecPhitransformlemmaeq2}, replace $g$ by $g\exp(tY)$ in \eqref{vecPhitransformlemmaeq1} for some $Y\in\mathfrak{g}$. Taking $\frac{d}{dt}\big|_0$ on both sides, we see that \eqref{vecPhitransformlemmaeq2} holds for $Y\in\mathfrak{g}$, and then also for $Y\in\mathfrak{g}_\C$. Continuing in this manner, we conclude that \eqref{vecPhitransformlemmaeq2} holds for $Y\in\mathcal{U}(\mathfrak{g}_\C)$ of any degree.
\end{proof}
Evidently, the function $\Phi$ in Lemma \ref{vecPhitransformlemma} can be recovered as the $S^m$-compo\-nent of $\vec\Phi$. It is easy to see that the map $\Phi\mapsto\vec\Phi$ establishes an isomorphism between the space of $K$-finite functions of weight $(\ell+m,\ell)$ satisfying $N_+\Phi=0$, and the space of smooth functions $\vec\Phi:\:\Sp_4(\R)\to W_m$ satisfying \eqref{vecPhitransformlemmaeq0}.

For later use, we make the following observation. Recall from Sect.~\ref{setupsec} that we have $\mathfrak{n}=\langle X_-,P_{1-},P_{0-}\rangle$, and that this commutative Lie algebra is normalized by $\mathfrak{k}_\C$. For a smooth function $\Phi$ of weight $(\ell+m,\ell)$ satisfying $N_+\Phi=0$, we then have
\begin{equation}\label{PhivecPhiholeq}
 \mathfrak{n}\Phi=0\qquad\Longleftrightarrow\qquad\mathfrak{n}\vec\Phi=0.
\end{equation}
(on both sides we mean the right translation action of $\mathfrak{n}$ on smooth functions on the group). This follows from the definition \eqref{vecPhidefeq}, and the fact that $N_-$ normalizes $\mathfrak{n}$.
\subsubsection*{Descending to the Siegel upper half space}
From the vector-valued function $\vec\Phi$ we can construct a vector-valued function on $\HH_2$, as follows. For $g\in\Sp_4(\R)$ and $Z\in\HH_2$, let
\begin{equation}\label{Jdefeq}
 J(g,Z)=CZ+D,\qquad g=\mat{A}{B}{C}{D}.
\end{equation}
Then $J(g_1g_2,Z)=J(g_1,g_2Z)J(g_2,Z)$. Since $\iota(h)=\bar h^{-1}$ for $h\in U(2)$, the transformation property \eqref{vecPhitransformlemmaeq0} can be rewritten as $\vec\Phi(gh)=\eta_{\ell,m}(J(h,I))^{-1}\vec\Phi(g)$ for $h\in K$; recall that $I=\left[\begin{smallmatrix}i&\\&i\end{smallmatrix}\right]$. It follows that the $W_m$-valued function $g\mapsto\eta_{\ell,m}(J(g,I))\vec\Phi(g)$ is right $K$-invariant. Hence, this function descends to a function $F$ on $\HH_2\cong\Sp_4(\R)/K$. Explicitly, we define $F$ by
\begin{equation}\label{FvecPhieq}
 F(Z)=\eta_{\ell,m}(J(g,I))\vec\Phi(g),
\end{equation}
where $g$ is any element of $\Sp_4(\R)$ satisfying $gI=Z$. Conversely, if $F$ is a smooth $W_m$-valued function on $\HH_2$, then we can define a smooth function $\vec\Phi$ on $\Sp_4(\R)$ by the formula  $\vec\Phi(g)=\eta_{\ell,m}(J(g,I))^{-1}F(gI)$. Clearly, $\vec\Phi$ satisfies the transformation property \eqref{vecPhitransformlemmaeq0}. Combining the maps $\Phi\mapsto\vec\Phi$ and $\vec\Phi\mapsto F$, we obtain the following result.

\begin{lemma}\label{FPhilemma}
 Let $\ell$ be any integer, and $m$ a non-negative integer. Let $\mathcal{V}_{\ell,m}$ be the space of smooth $K$-finite functions $\Phi:\:\Sp_4(\R)\to\C$ of weight $(\ell+m,\ell)$ satisfying $N_+\Phi=0$. Then $\mathcal{V}_{\ell,m}$ is isomorphic to the space of smooth functions $F:\:\HH_2\to W_m$. If $\Phi\in\mathcal{V}_{\ell,m}$, then the corresponding function $F$ is given by \eqref{FvecPhieq}, where $\vec\Phi$ is defined in \eqref{vecPhidefeq}.
\end{lemma}

Given any function $F:\:\HH_2\to W_m$, we will write $F$ in the form
$$
 F(Z)=\sum_{j=0}^mF_j(Z)S^{m-j}T^j,
$$
and call the complex-valued functions $F_j$ the \emph{component functions} of $F$. The component $F_0$ is obtained from $F$ by setting $(S,T)=(1,0)$. The component $F_1$ is obtained by taking $\frac{\partial}{\partial T}$ and then setting $(S,T)=(1,0)$. In general,
\begin{equation}\label{Fjformulaeq}
 F_j(Z)=\frac1{j!}\,\frac{\partial^j}{\partial T^j}F(Z)\big|_{(S,T)=(1,0)}.
\end{equation}

Next, we introduce coordinates on $\HH_2$, as follows. Let us write an element $Z\in\mathbb{H}_2$ as
\begin{equation}\label{H2coordinateseq}
 Z=\mat{\tau}{z}{z}{\tau'},\qquad\tau=x+iy,\qquad z=u+iv,\qquad\tau'=x'+iy',
\end{equation}
where $x,y,u,v,x',y'$ are real numbers, $y,y'>0$, and $yy'-v^2>0$.
We set
\begin{equation}\label{bZdefeq}
 b_Z=\begin{bmatrix}1&&x&u\\&1&u&x'\\&&1\\&&&1\end{bmatrix}\begin{bmatrix}1&v/y'\\&1\\&&1\\&&-v/y'&1\end{bmatrix}\begin{bmatrix}a\\&b\\&&a^{-1}\\&&&b^{-1}\end{bmatrix}
\end{equation}
with
\begin{equation}\label{abdefeq}
 a=\sqrt{y-\frac{v^2}{y'}}\quad\text{and}\quad b=\sqrt{y'}.
\end{equation}
Then $b_Z$ is an element of the Borel subgroup of $\Sp_4(\R)$, and $b_ZI=Z$. Every element of $\Sp_4(\R)$ can be written as $b_Zh$ for a uniquely determined $Z\in\HH_2$ and a uniquely determined $h\in K$.

If $F$, $\Phi$, $\vec\Phi$ are as above, then the following relation is immediate from~\eqref{FvecPhieq}. \begin{equation}\label{Frhotilde2eq}
 F(Z)=\eta_{l,m}(J(b_Z, I))\vec\Phi(b_Z).
\end{equation}

\subsection{The action of the root vectors}
Let $\Phi$, $\vec\Phi$ and $F$ be as in Lemma \ref{FPhilemma}. In this section we will calculate $(X\vec\Phi)(b_Z)$, where $X$ is any of the root vectors $X_\pm,P_{1\pm},P_{0\pm},N_\pm$, and where $b_Z$ is the element defined in \eqref{bZdefeq}. The result will be expressed in terms of differential operators applied to the function $F$. As a consequence, we will prove that $F$ is holomorphic if and only if $\mathfrak{n}\Phi=0$.

For $Z\in\HH_2$, let $D_Z=J(b_Z,I)$. Then $D_Z$ is simply the lower right $2\times2$-block of $b_Z$, explicitly,
\begin{equation}\label{DZdefeq}
 D_Z=\mat{1}{}{-v/y'}{1}\mat{a^{-1}}{}{}{b^{-1}},\qquad a=\sqrt{y-\frac{v^2}{y'}},\; b=\sqrt{y'}.
\end{equation}
\begin{proposition}\label{rootvectorsactionprop}
 Let $(\eta,W)$ be a finite-dimensional holomorphic representation of $\GL_2(\C)$.
 Let $F$ be a $W$-valued smooth function on $\HH_2$, and let $\vec\Phi$ be the corresponding $W$-valued function on $\Sp_4(\R)$, i.e.,
 $$
  \vec\Phi(g)=\eta(J(g,I))^{-1}F(gI).
 $$
 Let $b_Z$ be as in \eqref{bZdefeq}, and $D_Z$ as in \eqref{DZdefeq}. Then the following formulas hold. \begin{align}
  \eta(D_Z)(N_+\vec\Phi)(b_Z)
    &=\eta(D_Z)\eta(\mat{0}{0}{1}{0})\eta(D_Z)^{-1}F(Z).\label{opfourierlemmaeq1}\\
  \eta(D_Z)(N_-\vec\Phi)(b_Z)
    &=-\eta(D_Z)\eta(\mat{0}{1}{0}{0})\eta(D_Z)^{-1}F(Z).\label{opfourierlemmaeq2}\\
  \eta(D_Z)(P_{0+}\vec\Phi)(b_Z)
    &=\eta(D_Z)\eta(\mat{0}{0}{0}{1})\eta(D_Z)^{-1}F(Z).\nonumber\\
     &\hspace{10ex}+\frac {2i}{y'}\Big(v^2\frac{\partial F}{\partial \tau}+vy'\frac{\partial F}{\partial z}+y'{}^2\frac{\partial F}{\partial \tau'}\Big)(Z).\label{opfourierlemmaeq3}\\
  \eta(D_Z)(P_{0-}\vec\Phi)(b_Z)
    &=-\frac {2i}{y'}\Big(v^2\frac{\partial F}{\partial \bar\tau}+vy'\frac{\partial F}{\partial \bar{z}}+y'{}^2\frac{\partial F}{\partial \bar\tau'}\Big)(Z).\label{opfourierlemmaeq4}\\
  \eta(D_Z)(P_{1+}\vec\Phi)(b_Z)
    &=\eta(D_Z)\eta(\mat{0}{1}{1}{0})\eta(D_Z)^{-1}F(Z)\nonumber\\
     &\hspace{10ex}+\frac{2i}{y'}\sqrt{\Delta}\Big(2v\frac{\partial F}{\partial\tau}+y'\frac{\partial F}{\partial z}\Big)(Z).\label{opfourierlemmaeq5}\\
  \eta(D_Z)(P_{1-}\vec\Phi)(b_Z)
    &=-\frac{2i}{y'}\sqrt{\Delta}\Big(2v\frac{\partial F}{\partial\bar\tau}+y'\frac{\partial F}{\partial \bar{z}}\Big)(Z).\label{opfourierlemmaeq6}\\
  \eta(D_Z)(X_+\vec\Phi)(b_Z)
   &=\eta(D_Z)\eta(\mat{1}{0}{0}{0})\eta(D_Z)^{-1}F(Z)+\frac{2i}{y'}\Delta\frac{\partial F}{\partial\tau}(Z).\label{opfourierlemmaeq7}\\
  \eta(D_Z)(X_-\vec\Phi)(b_Z)
    &=-\frac{2i}{y'}\Delta\frac{\partial F}{\partial\bar\tau}(Z).\label{opfourierlemmaeq8}
 \end{align}
 Here, we used the abbreviation $\Delta=yy'-v^2$.
\end{proposition}
\begin{proof}To prove these formulas, one has to first compute the action of a basis of root vectors in the uncomplexified Lie algebra. This is relatively straightforward using the definitions, though somewhat tedious. Once that is done, the action of the root vectors above lying in the complexified Lie algebra follows by linearity. We omit the details.
\end{proof}
\begin{corollary}\label{PhiFholomorphiclemma}
 Let $\ell$ be any integer, and $m$ a non-negative integer. Let $\Phi:\:\Sp_4(\R)\to\C$ be a $K$-finite function of weight $(\ell+m,\ell)$ satisfying $N_+\Phi=0$. Let $F:\:\HH_2\to W_m$ be the function corresponding to $\Phi$ according to Lemma \ref{FPhilemma}. Then $F$ is holomorphic if and only if $\mathfrak{n}\Phi=0$.
\end{corollary}
\begin{proof}
It follows from \eqref{opfourierlemmaeq4}, \eqref{opfourierlemmaeq6} and \eqref{opfourierlemmaeq8} that $F$ is holomorphic if and only if $\mathfrak{n}\vec\Phi=0$. Now use \eqref{PhivecPhiholeq}.
\end{proof}
\subsection{The differential operators in classical language}
Let $\ell$ be any integer, and $m$ a non-negative integer. Let $\Phi$ be a $K$-finite complex-valued function on $\Sp_4(\R)$ of weight $(\ell+m,\ell)$ satisfying $N_+\Phi=0$. Let $F:\:\HH_2\to W_m$ be the function corresponding to $\Phi$ according to Lemma \ref{FPhilemma}. Let $X$ be one of the operators defined in Table \ref{Nplusoperatorstable}, and set $\Psi=X\Phi$. Then $\Psi$ is a $K$-finite function satisfying $N_+\Psi=0$, of weight indicated in Table \ref{Nplusoperatorstable}. Hence, according to Lemma \ref{FPhilemma}, there exists a vector-valued function $G$ corresponding to $\Psi$. In this section, we write down $G$ in terms of $F$ for all elements $X$ defined in Table \ref{Nplusoperatorstable}. The proofs rely on our formulas for the action of the root vectors (Proposition~\ref{rootvectorsactionprop}). The calculations involved, while long and detailed, are essentially routine, and therefore omitted.
\subsubsection*{Going down}
We start with $X=P_{0-}$. Hence, let $\Psi=P_{0-}\Phi$. Then $\Psi$ has weight $(\ell+m,\ell-2)$ and satisfies $N_+\Psi=0$. Let $G:\:\HH_2\to W_{m+2}$ be the function corresponding to $\Psi$ according to Lemma \ref{FPhilemma}. The following diagram illustrates the situation.
\begin{equation}\label{P0minusPhidiagram}
\begin{CD}
 \text{(weight $(\ell+m,\ell)$)}\qquad&&&\Phi@>>>\vec\Phi@>>>F&&&\qquad\text{(values in $W_m$)}\\
 &&&@V{P_{0-}}VV\\
 \text{(weight $(\ell+m,\ell-2)$)}\qquad&&&\Psi@>>>\vec\Psi@>>>G&&&\qquad\text{(values in $W_{m+2}$)}
\end{CD}
\end{equation}
Let $F_0,\ldots,F_m$ be the component functions of $F$, and let $G_0,\ldots,G_{m+2}$ be the component functions of $G$; see \eqref{Fjformulaeq}. We define three differential operators on $\HH_2$,

\begin{align}
 \label{del0eq}\bar\partial_0&=2i\Big(v^2 \frac{\partial}{\partial \bar \tau} + vy' \frac{\partial}{\partial \bar z}  + y'^2 \frac{\partial}{\partial \bar\tau'}\Big),\\
 \label{del1eq}\bar\partial_1&=-2i\Big(2vy \frac{\partial}{\partial \bar \tau} + (yy' + v^2) \frac{\partial}{\partial \bar z}  + 2vy' \frac{\partial}{\partial \bar\tau'}\Big),\\
 \label{del2eq}\bar\partial_2&=2i\Big(y^2 \frac{\partial}{\partial \bar \tau} + vy \frac{\partial}{\partial \bar z}  + v^2 \frac{\partial}{\partial \bar\tau'}\Big).
\end{align}
The following result expresses the $G_j$ in terms of the $F_i$.
\begin{proposition}\label{P0minusprop}
 With the above notations,
 \begin{equation}\label{P0minuspropeq1}
   G_j = -\big(\bar\partial_2F_{j-2} + \bar\partial_1F_{j-1} + \bar\partial_0F_j\big)
 \end{equation}
 for $j=0,\ldots,m+2$. (We understand $F_i=0$ for $i<0$ or $i>m$.)
\end{proposition}
\subsubsection*{Going left}
Next we write down the effect of the operator $L$, whose defining formula is given in Table \ref{Nplusoperatorstable}. In order for $L$ to be defined, we assume $m\geq2$. Let $\Psi=L\Phi$. Then $\Psi$ has weight $(\ell+m-2,\ell)$, and $N_+\Psi=0$. Let $G:\:\HH_2\to W_{m-2}$ be the function corresponding to $\Psi$. Let $F_0,\ldots,F_m$ be the component functions of $F$, and let $G_0,\ldots,G_{m-2}$ be the component functions of $G$; see \eqref{Fjformulaeq}.

\begin{proposition}\label{Lmprop}
 With the above notations,
 \begin{align}\label{Lmpropeq1}
  G_j&=-(m-j)(m-j-1)\bar\partial_2F_j\nonumber\\
   &\qquad+(m-j-1)(j+1)\bar\partial_1F_{j+1}\nonumber\\
   &\qquad-(j+2)(j+1)\bar\partial_0F_{j+2}.
 \end{align}
 for $j=0,\ldots,m-2$.
\end{proposition}

\subsubsection*{Going up}
Next, we consider $U$, whose defining formula is given in Table \ref{Nplusoperatorstable}. We will assume $m\geq2$, so that $U$ is well-defined. Let $\Psi=U\Phi$. Then $\Psi$ has weight $(\ell+m,\ell+2)$, and $N_+\Psi=0$. Let $G:\:\HH_2\to W_{m-2}$ be the function corresponding to $\Psi$.
Let $F_0,\ldots,F_m$ be the component functions of $F$, and let $G_0,\ldots,G_{m-2}$ be the component functions of $G$; see \eqref{Fjformulaeq}. The following result expresses the $G_j$ in terms of the $F_i$.
\begin{proposition}\label{Umprop}
 With the above notations,
 \begin{align}\label{Umpropeq1}
  G_j=(m-j)(m-j-1)&\Big((\ell-1)\frac y\Delta+2i\frac{\partial}{\partial\tau'}\Big)F_j\nonumber\\
   +(m-j-1)(j+1)&\Big((\ell-1)\frac{2v}\Delta-2i\frac{\partial}{\partial z}\Big)F_{j+1}\nonumber\\
   +(j+2)(j+1)&\Big((\ell-1)\frac{y'}\Delta+2i\frac{\partial}{\partial\tau}\Big)F_{j+2}
 \end{align}
 for $j=0,\ldots,m-2$.
\end{proposition}

\subsubsection*{Going right}
Next we calculate the effect of the operator $X_+$. Let $\Psi=X_+\Phi$. Then $\Psi$ has weight $(\ell+m+2,\ell)$, and $N_+\Psi=0$. Let $G:\:\HH_2\to W_{m+2}$ be the function corresponding to $\Psi$. Let $F_0,\ldots,F_m$ be the component functions of $F$, and let $G_0,\ldots,G_{m+2}$ be the component functions of $G$; see \eqref{Fjformulaeq}.
\begin{proposition}\label{Xplusprop}
 With the above notations,
 \begin{align}\label{Xpluspropeq1}
  G_j=\;&\Big((\ell+m)\frac{y}{\Delta}+2i\frac{\partial}
  {\partial\tau'}\Big)F_{j-2}\nonumber\\
   -&\Big((\ell+m)\frac{2v}{\Delta}-2i\frac{\partial}{\partial z}\Big)F_{j-1}\nonumber  \\ +&\Big((\ell+m)\frac{y'}{\Delta}+2i\frac{\partial}{\partial\tau}\Big)F_j   \end{align}
 for $j=0,\ldots,m+2$.
\end{proposition}

\begin{remark}
The operator $X_+$ is the same as the operator $\delta_{\ell+m}$ occurring in \cite{BochererSatohYamazaki1992}.
\end{remark}

\subsubsection*{The degree $1$ diagonal operators}
 Let $\Psi^\pm=E_{\pm}\Phi$. Then $\Psi^\pm$ has weight $(\ell+m\pm1,\ell\pm1)$, and $N_+\Psi^\pm=0$. Let  $G^\pm:\:\HH_2\to W_m$ be the function corresponding to $\Psi^\pm$. Let $F_0,\ldots,F_m$ be the component functions of $F$, and let $G^\pm_0,\ldots,G^\pm_m$ be the component functions of $G^\pm$; see \eqref{Fjformulaeq}. The following result expresses the $G^\pm_j$ in terms of the $F_i$.

\begin{proposition}\label{Dmminusprop}
 With the above notations,
 \begin{align}
  G^+_j&=(m-j+1)\Big((2\ell+m-2)\frac{y}{\Delta}+4i\frac{\partial}{\partial \tau'}\Big)F_{j-1}\nonumber\\
  &\qquad+(m-2j)\Big(-(2\ell+m-2)\frac{v}{\Delta}+2i\frac{\partial}{\partial z}\Big)F_j\nonumber\\
  &\qquad-(j+1)\Big((2\ell+m-2)\frac{y'}{\Delta}+4i\frac{\partial}{\partial\tau}\Big)F_{j+1},\label{Dmminuspropeq1b}\\
  G^-_j&=2(m+1-j)\bar\partial_2F_{j-1}+(m-2j)\bar\partial_1F_j-2(j+1)\bar\partial_0F_{j+1},\label{Dmminuspropeq1a}
 \end{align}
 for $j=0,\ldots,m$. (We understand $F_i=0$ for $i<0$ or $i>m$.) The differential operators $\bar\partial_0,\bar\partial_1,\bar\partial_2$ are the ones defined in \eqref{del0eq} -- \eqref{del2eq}.
\end{proposition}

\subsubsection*{The degree 2 diagonal operators}
Let $\Psi^\pm=D_\pm\Phi$. Then $\Psi$ has weight $(\ell+m\pm2,\ell\pm2)$, and $N_+\Psi^\pm=0$. Let  $G^\pm:\:\HH_2\to W_m$ be the function corresponding to $\Psi^\pm$. Let $F_0,\ldots,F_m$ be the component functions of $F$, and let $G^\pm_0,\ldots,G^\pm_m$ be the component functions of $G^\pm$; see \eqref{Fjformulaeq}. For scalar-valued or vector-valued functions on $\HH_2$, we define the differential operator
\begin{equation}
 \label{del3eq}\bar\partial_3=2i\Big(y \frac{\partial}{\partial \bar \tau}+v\frac{\partial}{\partial \bar z}+y'\frac{\partial}{\partial \bar\tau'}\Big).
\end{equation}
The following result expresses the $G^\pm_j$ in terms of the $F_i$.
\begin{proposition}\label{Dplusminusprop}
 With the above notations,
 \begin{align}
 \label{Dplusminuspropeq1} G^+_j&= (m-j+1)(m-j+2)\frac{y^2}{\Delta^2} F_{j-2}\nonumber \\&+ \bigg( 4i(m-j+1)\big(\frac{y}{\Delta}  \frac{\partial}{\partial z} + \frac{2v}{\Delta}  \frac{\partial}{\partial \tau'}\big)  - 2 (m-2j+1)(m-j+1) \frac{vy}{\Delta^2} \bigg)F_{j-1} \nonumber\\
 & + \bigg[(4j^2+4l^2-4jm+m(m-3)+l(4m-2))\frac{v^2}{\Delta^2} \nonumber\\
 & \quad + (j^2-jm-(2l-1)(l+m))\frac{2yy'}{\Delta^2} +16 \frac{\partial^2}{\partial \tau \partial \tau'} -4 \frac{\partial^2}{\partial z^2}\nonumber \\ & \quad -4i\Big((2j+2l-1) \frac{y}{\Delta}  \frac{\partial}{\partial \tau} + (m+2l-1) \frac{v}{\Delta}  \frac{\partial}{\partial z} + (2m-2j+2l-1) \frac{y'}{\Delta}  \frac{\partial}{\partial \tau'}\Big) \bigg]F_j \nonumber\\ & +\bigg( 4i(j+1) \big(\frac{2v}{\Delta}  \frac{\partial}{\partial \tau} + \frac{y'}{\Delta}  \frac{\partial}{\partial z}\big)  + 2(m-2j-1)(j+1)\frac{vy'}{\Delta^2}  \bigg)F_{j+1} \nonumber\\
 &  + (j+1)(j+2) \frac{y'^2}{\Delta^2} F_{j+2},\\
  \label{Dplusminuspropeq2}G^-_j&=\Big(4\Delta^2\Big(4\frac{\partial}{\partial\bar\tau}\frac{\partial}{\partial \bar\tau'}-\frac{\partial^2}{\partial \bar z^2}\Big)-2\Delta\bar\partial_3\Big) F_j,
 \end{align}
 for $j=0,\ldots,m$.
\end{proposition}

\begin{remark} The formula for $D_+$ in the special case that $m=0$ (the scalar valued case) is given by
\begin{equation}\label{D+scalar-eqn}
D_+F = \Big(-\frac{2l(2l-1)}{\Delta} - \frac{4i(2l-1)}{\Delta}\big(y\frac{\partial}{\partial\tau}+v\frac{\partial}{\partial z}+y'\frac{\partial}{\partial\tau'}\big)+4\big(4\frac{\partial^2}{\partial \tau \partial \tau'}-\frac{\partial^2}{\partial z^2}\big)\Big)F.
\end{equation}
In this case, the operator $D_+$ was originally defined by Maass in his book~\cite{maassbook}.
\end{remark}
\subsection{Nearly holomorphic functions}\label{s:nearholofn}
Let $p$ be a non-negative integer. We will write elements $Z \in \HH_2$ as $Z=X+iY$ with real $X$ and $Y$. By definition of $\HH_2$, the real symmetric matrix $Y$ is positive definite. We let $N^p(\HH_2)$ denote the space of all polynomials of degree $\le p$ in the entries of $Y^{-1}$ with holomorphic functions on $\HH_2$ as coefficients. The space
$$
 N(\HH_2)=\bigcup_{p\geq0}N^p(\HH_2)
$$
is the space of \emph{nearly holomorphic functions} on $\H_2$. Evidently, $N(\H_2)$ is a ring, and we have $N^p(\HH_2)N^q(\HH_2)\subset N^{p+q}(\HH_2)$. For convenience, we let $N^p(\HH_2)=0$ for negative $p$. If $f\in N(\HH_2)$ lies in $N^p(\HH_2)$ but not in $N^{p-1}(\HH_2)$, we say that $f$ has \emph{nearly holomorphic degree $p$}. Note that $N^0(\HH_2)$ is the space of holomorphic functions on $\HH_2$.

As before, we will use the coordinates \eqref{H2coordinateseq} on $\HH_2$, and set $\Delta=yy'-v^2$. The entries of $Y^{-1}$ are then $y/\Delta$, $v/\Delta$ and $y'/\Delta$. Since $\frac{y}{\Delta}\frac{y'}{\Delta}-\frac{v^2}{\Delta^2}=\frac1{\Delta}$, the function $\frac1{\Delta}$ is a nearly holomorphic function. For a typical nearly holomorphic monomial we will use the notation
\begin{equation}\label{nearholomonomialeq}
 \big[\alpha,\beta,\gamma\big]:=\Big(\frac{y}{\Delta}\Big)^\alpha\Big(\frac{v}{\Delta}\Big)^\beta\Big(\frac{y'}{\Delta}\Big)^\gamma;
\end{equation}
here, $\alpha,\beta,\gamma$ are non-negative integers.

We may ask how the various differential operators we defined in previous sections behave with respect to nearly holomorphic functions. It is easy to see that the basic partial derivatives
$$
    \frac{\partial}{\partial\tau},\quad\frac{\partial}{\partial z},\quad\frac{\partial}{\partial\tau'},\quad\frac{\partial}{\partial\bar\tau},\quad\frac{\partial}{\partial\bar z},\quad\frac{\partial}{\partial\bar\tau'}
$$
map $N^p(\H_2)$ to $N^{p+1}(\H_2)$. The following lemma gives the action of differential operators, including those defined in \eqref{del0eq} -- \eqref{del2eq} and \eqref{del3eq}, on a nearly holomorphic monomial. In particular, the lemma shows that the operators $\bar\partial_0,\bar\partial_1,\bar\partial_2$ act as ``nearly holomorphic derivatives''.
\begin{lemma}\label{nearholodiffoplemma}
 The following formulas hold for all non-negative integers $\alpha,\beta,\gamma$.
   \begin{align}
    \label{nearholodiffoplemmaeq1}\bar\partial_0\big[\alpha,\beta,\gamma\big]&=\alpha\big[\alpha-1,\beta,\gamma\big],\\
    \label{nearholodiffoplemmaeq2}\bar\partial_1\big[\alpha,\beta,\gamma\big]&=\beta\big[\alpha,\beta-1,\gamma\big],\\
    \label{nearholodiffoplemmaeq3}\bar\partial_2\big[\alpha,\beta,\gamma\big]&=\gamma\big[\alpha,\beta,\gamma-1\big],\\
    \label{nearholodiffoplemmaeq5}\bar\partial_3\big[\alpha,\beta,\gamma\big]&=(\alpha+\beta+\gamma)\big[\alpha,\beta,\gamma\big],\\
    \label{nearholodiffoplemmaeq7}D_{-}\big[\alpha,\beta,\gamma\big]&=\beta(\beta-1)\big[\alpha,\beta-2,\gamma\big]-4\alpha\gamma\big[\alpha-1,\beta,\gamma-1\big].
   \end{align}
\end{lemma}
\begin{proof}
Everything follows from direct calculations.
\end{proof}
As a consequence, we note that the operators $\bar\partial_0,\bar\partial_1,\bar\partial_2$ commute on $N(\HH_2)$ (they do not commute on all of $C^\infty(\HH_2)$).
\begin{lemma}\label{nearholouniquelemma}
 Assume that $F=\sum_{\alpha,\beta,\gamma\geq0}[\alpha,\beta,\gamma]F_{\alpha,\beta,\gamma}$ is a nearly holomorphic function, where the $F_{\alpha,\beta,\gamma}$ are holomorphic. Then $F$ is zero if and only if all $F_{\alpha,\beta,\gamma}$ are zero.
\end{lemma}
\begin{proof}
This can be proved by induction on the nearly holomorphic degree, using the formulas \eqref{nearholodiffoplemmaeq1} -- \eqref{nearholodiffoplemmaeq3}.
\end{proof}

\subsubsection*{Operators on vector-valued functions}
Let $\ell$ be any integer, and $m$ a non-negative integer. Let $C^\infty_{\ell,m}(\HH_2)$ be the space of smooth functions $F:\:\HH_2\to W_m$. Note that this space does not actually depend on $\ell$; nevertheless, it will be useful to carry this subindex along (the significance of this subindex will be seen in the next chapter, when we will restrict to the subspace of $C^\infty_{\ell,m}(\HH_2)$ consisting of forms $F$ which transform via $\eta_{\ell,m}$ with respect to some congruence subgroup).

For each of the operators $X$ appearing in Table~\ref{Nplusoperatorstable} we will define a linear map $X:\:C^\infty_{\ell,m}(\HH_2)\to C^\infty_{\ell_1,m_1}(\HH_2)$, where $(\ell_1+m_1,\ell_1)$ is the ``new weight'' given in Table \ref{Nplusoperatorstable}. Some of the operators $X$ will depend on $\ell$, (or $m$, or both) but, as before, our notation will not reflect this dependence.

If $m<2$, we set $U=L=0$. In all other cases, the definitions will be in terms of the component functions $F_0,\ldots,F_m$ of $F$ given by $
 F(Z)=\sum_{j=0}^mF_j(Z)S^{m-j}T^j,
$ and are as follows.
\begin{align}
  \label{operatorsonfunctions1}(P_{0-}F)_j&=\text{right hand side of \eqref{P0minuspropeq1}},\\
  \label{operatorsonfunctions2}(LF)_j&=\text{right hand side of \eqref{Lmpropeq1}},\\
  \label{operatorsonfunctions3}(UF)_j&=\text{right hand side of \eqref{Umpropeq1}},\\
  \label{operatorsonfunctions4}(X_+F)_j&=\text{right hand side of \eqref{Xpluspropeq1}},\\
  \label{operatorsonfunctions5}(E_{+}F)_j&=\text{right hand side of \eqref{Dmminuspropeq1b}},\\
  \label{operatorsonfunctions6}(E_{-}F)_j&=\text{right hand side of \eqref{Dmminuspropeq1a}},\\
  \label{operatorsonfunctions7}(D_+F)_j&=\text{right hand side of \eqref{Dplusminuspropeq1}},\\
  \label{operatorsonfunctions8}(D_-F)_j&=\text{right hand side of \eqref{Dplusminuspropeq2}}.
\end{align}
These formulas hold for \emph{all} $j\in\Z$, but the expressions on the right hand sides are automatically zero if $j<0$ or $j>m_1$.

For a non-negative integer $p$, let $N^p_{\ell,m}(\HH_2)$ be the subspace of $C^\infty_{\ell,m}(\HH_2)$ consisting of those $F$ for which all component functions $F_j$ are in $N^p(\HH_2)$. Hence, these are nearly holomorphic $W_m$-valued functions. The space $N^0_{\ell,m}(\HH_2)$ consists of the holomorphic $W_m$-valued functions.
\begin{table}[!h]
\caption{Let $X$ be one of the operators given in the first column. Let $F\in N^p_{\ell,m}(\HH_2)$. Then $XF\in N^{p_1}_{\ell_1,m_1}(\HH_2)$, with $\ell_1,m_1,p_1$ given in the last three columns of the table. The second column indicates the direction from the old weight $(\ell+m,\ell)$ to the new weight $(\ell_1+m_1,\ell_1)$, assuming $F$ corresponds to the $K$-finite function $\Phi:\:\Sp_4(\R)\to\C$ of weight $(\ell+m,\ell)$. If $m<2$, then by definition, $U=L=0$.
} \label{Nplusoperatorsnearholotable}
$$\renewcommand{\arraystretch}{1.5}\renewcommand{\arraycolsep}{0.4cm}
 \begin{array}{ccccc}
  \toprule
   \text{operator}&\text{direction}&\text{new $\ell$}&\text{new $m$}&\text{new $p$}\\
  \toprule
   P_{0-}&\downarrow&\ell-2&m+2&p-1\\
  \midrule
   L&\leftarrow&\ell&m-2&p-1\\
  \midrule
   U&\uparrow&\ell+2&m-2&p+1\\
  \midrule
   X_+&\rightarrow&\ell&m+2&p+1\\
  \midrule
   E_{+}&\nearrow&\ell+1&m&p+1\\
  \midrule
   E_{-}&\swarrow&\ell-1&m&p-1\\
  \midrule
   D_+&\nearrow&\ell+2&m&p+2\\
  \midrule
   D_-&\swarrow&\ell-2&m&p-2\\
  \bottomrule
 \end{array}
$$
\end{table}
\begin{proposition}\label{Nplusoperatorsnearholoprop}
 Let $\ell$ be any integer, and $m$ a non-negative integer. Let $X$ be one of the operators in Table \ref{Nplusoperatorstable}. Let $F\in C^\infty_{\ell,m}(\HH_2)$.
 \begin{enumerate}
  \item Assume that $F$ corresponds, via Lemma \ref{FPhilemma}, to the $K$-finite function $\Phi$ on $\Sp_4(\R)$ of weight $(\ell+m,\ell)$ satisfying $N_+\Phi=0$. Then $XF$ corresponds to $X\Phi$. In other words, the diagram
   \begin{equation}\label{Nplusoperatorsnearholopropeq1}
    \begin{CD}
     \mathcal{V}_{\ell,m}@>\sim>>C^\infty_{\ell,m}(\HH_2)\\
     @V{X}VV @VV{X}V\\
     \mathcal{V}_{\ell_1,m_1}@>\sim>>C^\infty_{\ell_1,m_1}(\HH_2)
    \end{CD}
   \end{equation}
   is commutative. Here, $\ell_1,m_1$ are given in Table \ref{Nplusoperatorsnearholotable}, and the horizontal isomorphisms are those from Lemma \ref{FPhilemma}.
  \item If $F\in N^p_{\ell,m}(\HH_2)$, then $XF\in N^{p_1}_{\ell_1,m_1}(\HH_2)$, where $\ell_1,m_1,p_1$ are given in the last three columns of Table \ref{Nplusoperatorsnearholotable}.
 \end{enumerate}
\end{proposition}
\begin{proof}
(1) simply summarizes the content of Propositions \ref{P0minusprop}, \ref{Lmprop}, \ref{Umprop}, \ref{Xplusprop}, \ref{Dmminusprop} and \ref{Dplusminusprop}.

(2) follows from the formulas \eqref{operatorsonfunctions1} -- \eqref{operatorsonfunctions8}, together with Lemma \ref{nearholodiffoplemma}.
\end{proof}

We see from (2) of this result that if we walk in the direction of one of the roots in $\mathfrak{n}$, then the nearly holomorphic degree decreases, while if we walk in the direction of one of the roots in $\p_+$, then the nearly holomorphic degree (potentially) increases. In the next section, we will use the following holomorphy criterion to prove that spaces of nearly holomorphic modular forms are finite-dimensional.
\begin{lemma}\label{holomorphyconditionlemma}
 Let $\ell$ be any integer, and $m$ a non-negative integer. Let $F\in C_{\ell,m}^\infty(\HH_2)$. Let $p\in\{0,1\}$.
 \begin{enumerate}
  \item If $m=0$, then the following are equivalent:
   \begin{enumerate}
    \item $F\in N^p(\HH_2)$.
    \item $P_{0-}F\in N^{p-1}(\HH_2)$.
   \end{enumerate}
   In particular, $F$ is holomorphic if and only if $P_{0-}F=0$.
  \item If $m=1$, then the following are equivalent:
   \begin{enumerate}
    \item $F\in N^p(\HH_2)$.
    \item $P_{0-}F,E_{-}F\in N^{p-1}(\HH_2)$.
   \end{enumerate}
   In particular, $F$ is holomorphic if and only if $P_{0-}F=E_{-}F=0$.
  \item If $m\geq2$, then the following are equivalent:
   \begin{enumerate}
    \item $F\in N^p(\HH_2)$.
    \item $P_{0-}F,E_{-}F,LF\in N^{p-1}(\HH_2)$.
   \end{enumerate}
   In particular, $F$ is holomorphic if and only if $P_{0-}F=E_{-}F=LF=0$.
 \end{enumerate}
\end{lemma}
\begin{proof}
These properties can be verified from the explicit formulas for the operators written above. The details are omitted.
\end{proof}
\section{The structure theorems}\label{s:structure}
\subsection{Modular forms}
Recall that, for a positive integer $N$, the \emph{principal congruence subgroup} $\Gamma(N)$ consists of all elements of $\Sp_4(\Z)$ that are congruent to the identity matrix modulo $N$. A \emph{congruence subgroup} of $\Sp_4(\Q)$  is a subgroup that, for some $N$, contains $\Gamma(N)$ with finite index. The reason that we do not restrict ourselves to subgroups of $\Sp_4(\Z)$ is that we would like to include groups like the paramodular group.

Let $\ell$ be an integer, and $m$ a non-negative integer. Recall from Sect.\ \ref{fctsgrpH2sec} that $\eta_{\ell,m}$ denotes the $(m+1)$-dimensional representation $\det^\ell{\rm sym}^m$ of $\GL_2(\C)$. As before, let $C^\infty_{\ell,m}(\HH_2)$ be the space of smooth $W_m$-valued functions on $\HH_2$. We define a right action of $\Sp_4(\R)$ on $C^\infty_{\ell,m}(\HH_2)$ by
\begin{equation}\label{slashoperatoreq}
 (F\big|_{\ell,m}g)(Z)=\eta_{\ell,m}(J(g,Z))^{-1}F(gZ)\qquad\text{for }g\in\Sp_4(\R),\;Z\in\HH_2.
\end{equation}
In the following we fix a congruence subgroup $\Gamma$ of $\Sp_4(\Q)$. Let $C^\infty_{\ell,m}(\Gamma)$ be the space of smooth functions $F:\:\HH_2\to W_m$ satisfying
\begin{equation}\label{modformeq1}
 F\big|_{\ell,m}\gamma=F\qquad\text{for all }\gamma\in\Gamma.
\end{equation}
It is easy to see that $F\in C^\infty_{\ell,m}(\HH_2)$ has this transformation property if and only if the function $\Phi\in\mathcal{V}_{\ell,m}$ corresponding to $F$ via Lemma \ref{FPhilemma} satisfies $\Phi(\gamma g)=\Phi(g)$ for all $g\in\Sp_4(\R)$ and $\gamma\in\Gamma$. Let $\mathcal{V}_{\ell,m}(\Gamma)$ be the subspace of $\mathcal{V}_{\ell,m}$ consisting of $\Phi$ with this transformation property. If $X$ is one of the operators in Table \ref{Nplusoperatorstable}, then it follows from Proposition \ref{Nplusoperatorsnearholoprop} that there is a commutative diagram
\begin{equation}\label{Nplusoperatorsnearholopropeq2}
    \begin{CD}
     \mathcal{V}_{\ell,m}(\Gamma)@>\sim>>C^\infty_{\ell,m}(\Gamma)\\
     @V{X}VV @VV{X}V\\
     \mathcal{V}_{\ell_1,m_1}(\Gamma)@>\sim>>C^\infty_{\ell_1,m_1}(\Gamma)
    \end{CD}
\end{equation}
Here, $\ell_1,m_1$ are the integers given in Table \ref{Nplusoperatorsnearholotable}. (One could verify directly that if $F$ satisfies \eqref{modformeq1}, then $XF$ satisfies $(XF)\big|_{\ell_1,m_1}\gamma=F$ for all $\gamma\in\Gamma$, but the use of the diagrams is much easier.)

More generally, one has the following basic commutation relation.

\begin{lemma}\label{slashliecommute}
 Let $\mathcal{X}$ be the free monoid consisting of all (finite) strings of the symbols in the left column of Table~\ref{Nplusoperatorstable}. Suppose that $X$ is an element of $\mathcal{X}$ and let $(\ell_1, m_1)$ be the integers (uniquely determined by $\ell$, $m$ and $X$) such that $X$ takes $C_{\ell, m}^\infty(\Gamma)$ to $C_{\ell_1, m_1}^\infty(\Gamma)$. Let $\gamma \in \Sp_4(\R)$. Then, for all $F \in C_{\ell, m}^\infty(\H_2)$, we have $(XF)|_{\ell_1, m_1}\gamma = X(F|_{\ell,m}\gamma)$.
\end{lemma}
\begin{proof}
Let $\Phi$ be the function corresponding to $F$  via Lemma \ref{FPhilemma}. Then it follows from Proposition \ref{Nplusoperatorsnearholoprop} that $X\Phi$ corresponds to $XF$. On the other hand, the operation $|_{\ell_1, m_1}\gamma$ corresponds to left multiplication of the argument by $\gamma$. Define a function $\Phi_1$ on $\Sp_4(\R)$ by $\Phi_1(g) = \Phi(\gamma g)$. Now the proof follows from the obvious identity $(X\Phi)(\gamma g) = (X\Phi_1)(g)$.\end{proof}
\subsubsection*{Fourier expansions}
Now let $F\in C^\infty_{\ell,m}(\Gamma)\cap N^p_{\ell,m}(\HH_2)$. Hence, $F$ is nearly holomorphic and satisfies \eqref{modformeq1}. Let $F_0,\ldots,F_m$ be the component functions of $F$, as defined in \eqref{Fjformulaeq}. Suppose $F_j$ is written as $F_j=\sum_{\alpha,\beta,\gamma}[\alpha,\beta,\gamma]F_{j,\alpha,\beta,\gamma}$ with holomorphic functions $F_{j,\alpha,\beta,\gamma}$; see \eqref{nearholomonomialeq} for notation. Since $F$ is invariant under the translations $\tau\mapsto\tau+N$, $z\mapsto z+N$ and $\tau'\mapsto\tau'+N$ for some positive integer $N$, the same is true for $F_j$ and each $F_{j,\alpha,\beta,\gamma}$; observe here Lemma \ref{nearholouniquelemma}. Thus $F_{j,\alpha,\beta,\gamma}$ admits a Fourier expansion
\begin{equation}\label{Fourierexpansioneq1}
 F_{j,\alpha,\beta,\gamma}(Z)=\sum_{Q}a_{j,\alpha,\beta,\gamma}(Q)e^{2\pi i\,{\rm Tr}(QZ)},
\end{equation}
where $Q$ runs over matrices $\big[\begin{smallmatrix}a&b/2\\b/2&c\end{smallmatrix}\big]$ with $a,b,c\in\frac1N\Z$. It follows that $F_j$ admits a Fourier expansion
\begin{equation}\label{Fourierexpansioneq2}
 F_j(Z)=\sum_{Q}a_j(Q)e^{2\pi i\,{\rm Tr}(QZ)},\quad a_j(Q):=\sum_{\alpha,\beta,\gamma}a_{j,\alpha,\beta,\gamma}(Q)\big[\alpha,\beta,\gamma\big],
\end{equation}
and that $F$ admits a Fourier expansion
\begin{equation}\label{Fourierexpansioneq3}
 F(Z)=\sum_{Q \in M_2^{\sym}(\Q)}a(Q)e^{2\pi i\,{\rm Tr}(QZ)},
\end{equation}
where
\begin{equation}\label{Fourierexpansioneq4}
 a(Q)=\sum_{j=0}^m\:\sum_{\alpha,\beta,\gamma}a_{j,\alpha,\beta,\gamma}(Q)\big[\alpha,\beta,\gamma\big]S^{m-j}T^j.
\end{equation}
Thus, the Fourier coefficients of $F$ are polynomial functions in the entries of $Y^{-1}$ taking values in $W_m$. For fixed $Q$, the complex-valued functions $a_j(Q)$ in \eqref{Fourierexpansioneq2} are nothing but the component functions of $a(Q)$. If $X$ is one of the operators defined in \eqref{operatorsonfunctions1}, \eqref{operatorsonfunctions2}, \eqref{operatorsonfunctions6} or \eqref{operatorsonfunctions8}, and if $F$ has Fourier expansion \eqref{Fourierexpansioneq3}, then $XF$ has Fourier expansion
\begin{equation}\label{Fourierexpansioneq5}
 (XF)(Z)=\sum_{Q}(Xa(Q))e^{2\pi i\,{\rm Tr}(QZ)}.
\end{equation}
This follows directly from the definitions and the fact that $e^{2\pi i\,{\rm Tr}(QZ)}$ is holomorphic for all matrices $Q$. If $X$ is one of the operators defined in \eqref{operatorsonfunctions3}, \eqref{operatorsonfunctions4}, \eqref{operatorsonfunctions5} or \eqref{operatorsonfunctions7}, then the Fourier expansion of $XF$ is more complicated. However, it is easy to see that
\begin{equation}\label{Fourierexpansioneq6}
 (XF)(Z)=\sum_{Q}b(Q)e^{2\pi i\,{\rm Tr}(QZ)},\qquad\text{with $b(Q)=0$ if $a(Q)=0$}.
\end{equation}
Hence, none of the eight operators introduces any ``new'' Fourier coefficients.
\subsubsection*{Nearly holomorphic modular forms}
Let $\ell$ be an integer, and $m,p$ be non-negative integers. For a congruence subgroup $\Gamma$, let $N^p_{\ell,m}(\Gamma)$ be the space of all functions $F:\:\HH_2\to W_m$ with the following properties.
\begin{enumerate}
 \item $F\in N^p_{\ell,m}(\HH_2)$.
 \item $F$ satisfies the transformation property \eqref{modformeq1}.
 \item $F$ satisfies the ``no poles at cusps" condition. This means: For any $g\in\Sp_4(\Q)$ the function $F\big|_{\ell,m}g$ admits a Fourier expansion of the form \eqref{Fourierexpansioneq3} such that $a(Q)=0$ unless $Q$ is positive semidefinite.
\end{enumerate}
Let $N_{\ell,m}(\Gamma)=\bigcup_{p\geq0}N^p_{\ell,m}(\Gamma)$. We refer to $N_{\ell,m}(\Gamma)$ as the space of \emph{nearly holomorphic Siegel modular forms} of weight $\det^{\ell}\sym^m$ with respect to $\Gamma$. We sometimes write $M_{\ell,m}(\Gamma)$ for $N^0_{\ell,m}(\Gamma)$; this is the usual space of holomorphic vector-valued Siegel modular forms taking values in $\eta_{\ell,m}$.

An element $F\in N_{\ell,m}(\Gamma)$ is called a \emph{cusp form} if it vanishes at all cusps. By definition, this means: For any $g\in\Sp_4(\Q)$ the function $F\big|_{\ell,m}g$ admits a Fourier expansion of the form \eqref{Fourierexpansioneq3}, for some $N$, such that $a(Q)=0$ unless $Q$ is positive definite. We write $N_{\ell,m}(\Gamma)^\circ$ for the subspace of cusp forms. Let $N^p_{\ell,m}(\Gamma)^\circ=N_{\ell,m}(\Gamma)^\circ\cap N^p_{\ell,m}(\Gamma)$. We sometimes write $S_{\ell,m}(\Gamma)$ for $N^0_{\ell,m}(\Gamma)^\circ$; this is the usual space of holomorphic vector-valued Siegel cusp forms taking values in $\eta_{\ell,m}$.

\begin{lemma}\label{Nlmpfindimlemma}
 The spaces $N^p_{\ell,m}(\Gamma)$ and $N^p_{\ell,m}(\Gamma)^\circ$ are finite-dimensional.
\end{lemma}
\begin{proof}
Obviously, we only need to prove this for $N^p_{\ell,m}(\Gamma)$. It is well known, and can be proved using Harish-Chandra's general finiteness result stated as Theorem 1.7 in \cite{BorelJacquet1979}, that the statement is true for $p=0$, i.e., for holomorphic modular forms. Assume that $p>0$. If $m=0$, then, by (1) of Lemma \ref{holomorphyconditionlemma}, the map $F\mapsto P_{0-}F$ gives rise to an exact sequence
$$
 0\longrightarrow M_{\ell,m}(\Gamma)\longrightarrow N^p_{\ell,m}(\Gamma)\longrightarrow N^{p-1}_{\ell-2,m+2}(\Gamma).
$$
If $m=1$, then, by (2) of Lemma \ref{holomorphyconditionlemma}, the map $F\mapsto(P_{0-}F,E_{-}F)$ gives rise to an exact sequence
$$
 0\longrightarrow M_{\ell,m}(\Gamma)\longrightarrow N^p_{\ell,m}(\Gamma)\longrightarrow N^{p-1}_{\ell-2,m+2}(\Gamma)\oplus N^{p-1}_{\ell-1,m}(\Gamma).
$$
If $m\geq2$, then, by (3) of Lemma \ref{holomorphyconditionlemma}, the map $F\mapsto(P_{0-}F,E_{-}F,LF)$ gives rise to an exact sequence
$$
 0\longrightarrow M_{\ell,m}(\Gamma)\longrightarrow N^p_{\ell,m}(\Gamma)\longrightarrow N^{p-1}_{\ell-2,m+2}(\Gamma)\oplus N^{p-1}_{\ell-1,m}(\Gamma)\oplus N^{p-1}_{\ell,m-2}(\Gamma).
$$
Hence our assertion follows by induction on $p$.
\end{proof}
\subsection{Automorphic forms}
Let $\Gamma$ be a congruence subgroup of $\Sp_4(\Q)$. We denote by $\AA(\Gamma)$ the space of automorphic forms on $\Sp_4(\R)$ with respect to $\Gamma$. Recall that an automorphic form is a smooth function on $\Sp_4(\R)$ that is left $\Gamma$-invariant, $\mathcal{Z}$-finite, $K$-finite and slowly increasing; here $\mathcal{Z}$ is the center of $\mathcal{U}(\mathfrak{g}_\C)$. Let $\AA(\Gamma)^\circ$ be the subspace of cuspidal automorphic forms. We refer to \cite{BorelJacquet1979} for precise definitions of these notions. The spaces $\AA(\Gamma)$ and $\AA(\Gamma)^\circ$ are $(\mathfrak{g},K)$-modules under right translation.

Let $dg$ be any Haar measure on $\Sp_4(\R)$. For $\Phi_1$ and $\Phi_2$ in $\AA(\Gamma)$, we define the integral
\begin{equation}\label{AA0innerproducteq}
 \langle\Phi_1,\Phi_2\rangle:=\frac1{{\rm vol}(\Gamma\backslash\Sp_4(\R))}\:\int\limits_{\Gamma\backslash\Sp_4(\R)}\Phi_1(g)\overline{\Phi_2(g)}\,dg
\end{equation}
whenever it is \emph{absolutely convergent}. This happens, for example, whenever at least one of $\Phi_1$ and $\Phi_2$ lies in $\AA(\Gamma)^\circ$. In particular, $\langle\,,\rangle$ defines an inner product on $\AA(\Gamma)^\circ$ invariant under right translations by $\Sp_4(\R)$. For an element $X \in \g$, we have $$\langle X\Phi_1, \Phi_2 \rangle + \langle \Phi_1, X\Phi_2 \rangle = 0.$$  By general principles (see \cite{BorelJacquet1979} and the references therein) $\AA(\Gamma)^\circ$ decomposes into an orthogonal direct sum of irreducible $(\mathfrak{g},K)$-modules, each occurring with finite multiplicity.

Let $\lambda=(k,\ell)$ be an element of the weight lattice $\Lambda$. We say that $\Phi\in\AA(\Gamma)$ has weight $\lambda$ if $Z\Phi=k\Phi$ and $Z'\Phi=\ell\Phi$ (right translation action). Let $\AA_\lambda(\Gamma)$ be the subspace of $\AA(\Gamma)$ consisting of elements of weight $\lambda$, and let $\AA_\lambda(\Gamma)^\circ$ be similarly defined.

Let $\mathfrak{n}\subset\mathfrak{g}$ be the span of the root vectors $X_-$, $P_{1-}$ and $P_{0-}$. Then $\mathcal{U}(\mathfrak{n})$ is the polynomial algebra in the three variables $X_-$, $P_{1-}$ and $P_{0-}$. An automorphic form $\Phi$ is called \emph{$\mathfrak{n}$-finite} if the space $\mathcal{U}(\mathfrak{n})\Phi$ is finite-dimensional. We denote by $\AA(\Gamma)_{\mathfrak{n}\text{-fin}}$ the space of $\mathfrak{n}$-finite automorphic forms, and by $\AA(\Gamma)^\circ_{\mathfrak{n}\text{-fin}}$ the subspace of cusp forms. The following properties are easy to verify:
\begin{itemize}
 \item $\AA(\Gamma)_{\mathfrak{n}\text{-fin}}$ is a $(\mathfrak{g},K)$-submodule of $\AA(\Gamma)$.
 \item $\AA(\Gamma)_{\mathfrak{n}\text{-fin}}$ is the direct sum of its $K$-types, i.e.: If $\Phi\in\AA(\Gamma)_{\mathfrak{n}\text{-fin}}$ and $\Phi=\Phi_1+\ldots+\Phi_m$, where $\Phi_i$ lies in the $\rho_i$-isotypical component of $\AA(\Gamma)$ for different $K$-types $\rho_i$, then $\Phi_i\in\AA(\Gamma)_{\mathfrak{n}\text{-fin}}$ for each $i$.

\end{itemize}
Analogous statements hold for cusp forms.
\begin{lemma}\label{admissibilitylemma}
 $\AA(\Gamma)_{\mathfrak{n}\text{-fin}}$ is an admissible $(\mathfrak{g},K)$-module.
\end{lemma}
\begin{proof}
Assume that a $K$-type $\rho_\lambda$ occurs infinitely often in $\AA(\Gamma)_{\mathfrak{n}\text{-fin}}$ for some $\lambda=(\ell+m,\ell)$. We may assume that $\lambda$ is maximal in the order \eqref{dominanceordereq2}. Let $W$ be the space of highest weight vectors in the $\rho_\lambda$-isotypical component; by assumption, $W$ is infinite-dimensional. By our maximality assumption, the kernel $W_1$ of $P_{0-}$ on $W$ is infinite-dimensional; note that $N_+$ commutes with $P_{0-}$. Similarly, the kernel $W_2$ of $P_{1-}$ on $W_1$ is infinite-dimensional. Finally, the kernel $W_3$ of $X_-$ on $W_2$ is infinite-dimensional. The vectors in $W_3$ correspond to holomorphic modular forms in $M_{\ell,m}(\Gamma)$. Since this space is finite-dimensional, we obtain a contradiction.
\end{proof}

\subsubsection*{Modular forms and automorphic forms}
We are going to prove that nearly holomorphic modular forms generate $\mathfrak{n}$-finite automorphic forms. The following lemma will be useful.
\begin{lemma}\label{nfinitenesslemma}
 Let $V$ be a $\mathfrak{g}_\C$-module, and $v_0\in V$ a vector with the following properties:
 \begin{itemize}
  \item $V=\mathcal{U}(\mathfrak{g}_\C)v_0$.
  \item $v_0$ has weight $(\ell+m,\ell)$ for some integer $\ell$ and non-negative integer $m$.
  \item $N_+v_0=0$.
  \item $N_-^rv_0=0$ for some $r>0$.
  \item $P_{0-}^sv_0=0$ for some $s>0$.
  \item $D_-^tv_0=0$ for some $t>0$.
 \end{itemize}
 Then $v_0$ is $\mathfrak{n}$-finite, and $V$ is an admissible $(\mathfrak{g},K)$-module in which each weight space is finite-dimensional.
\end{lemma}
\begin{proof}
Let $X=X_-$, $Y=P_{1-}$ and $Z=P_{0-}$, so that $\mathcal{U}(\mathfrak{n})$ is the polynomial ring $\C[X,Y,Z]$. In this ring, let $I$ be the ideal generated by $D_-^t=(Y^2-4XZ)^t$ and $Z^s$. By our hypothesis, every element of $I$ annihilates $v_0$.

In affine three-space, consider the vanishing set $N(I)$. Clearly, a point $(x,y,z)$ in $N(I)$ must have $y=z=0$. Since the polynomial $Y$ vanishes on all of $N(I)$, we have $Y^n\in I$ for some positive integer $n$ by Hilbert's Nullstellensatz.

By the PBW theorem, $\mathcal{U}(\mathfrak{g}_\C)$ is spanned by monomials of the form
$$
 \text{(monomial in $X_-,N_-,P_{0+},P_{1+},X_+,Z,Z'$)}\times P_{1-}^\alpha P_{0-}^\beta N_+^\gamma
$$
with $\alpha,\beta,\gamma\geq0$. Since $P_{1-}$, $P_{0-}$, $N_+$ are the only root vectors with a downwards component, and since $P_{1-}^nv_0=P_{0-}^sv_0=N_+v_0=0$, it follows that $V$ cannot have weights $(k,k')$ below a certain line $k'=k'_0$ for some $k'_0<\ell$.

Now consider the vectors $X_-^qv_0$ for positive integers $q$. Since $[N_-,X_-]=0$, all these vectors are annihilated by $N_-^r$. If $X_-^qv_0$ would be non-zero for very large $q$, then it would generate a $\mathfrak{k}_\C$-module containing weights below the line $k'=k'_0$; this is impossible. Hence there exists a $q$ such that $X_-^qv_0=0$.

Now, in $\C[X,Y,Z]$, consider the ideal $J$ generated by $X^q$ and $D_-^t=(Y^2-4XZ)^t$ and $Z^s$. Clearly, its vanishing set in affine three-space consists of only the point $(0,0,0)$. It follows that $\C[X,Y,Z]/J$ is finite-dimensional as a $\C$-vector space (see, e.g., Corollary 4 in Sect.~1.7 of \cite{Fulton1989}). Since the annihilator of $v_0$ contains $J$, it follows that $v_0$ is $\mathfrak{n}$-finite.

Since we know that $X_-^qv_0=0$, an argument analogous to the above shows that $V$ cannot have any weights $(k,k')$ to the left of a certain line $k=k_0$. Thus, $V$ contains only finite-dimensional $\mathfrak{k}_\C$-modules. It also follows that each weight space is finite-dimensional. In particular, $V$ is admissible.
\end{proof}

\begin{proposition}\label{Nplmautformprop}
 Let $\ell$ be an integer, and $m$ and $p$ be non-negative integers. Let $\Gamma$ be a congruence subgroup of $\Sp_4(\Q)$. Let $F\in N^p_{\ell,m}(\Gamma)$ be non-zero. Let $\Phi:\:\Sp_4(\R)\to\C$ be the function corresponding to $F$ via Lemma \ref{FPhilemma}. Then $\Phi\in\AA(\Gamma)_{\mathfrak{n}\text{-fin}}$. Furthermore, $\Phi\in\AA(\Gamma)^\circ_{\mathfrak{n}\text{-fin}}$ if and only if $F$ is a cusp form.
\end{proposition}
\begin{proof}
Evidently, $\Phi$ is smooth, left $\Gamma$-invariant, $K$-finite and has weight $(\ell+m,\ell)$. Next, we briefly sketch a proof that $\Phi$ is slowly increasing (for further details, the reader can look at \cite{pss-growth} where a generalisation of this to higher degree is proved). Using the fact that the Fourier expansion of each component function $F_j$ is supported on the positive semi-definite matrices, an argument similar to \cite[\S13]{maassbook} shows that the function $F$ is bounded in a Siegel set; use the fact that the functions $[\alpha,\beta,\gamma]$ are bounded in a Siegel set.
This shows that there is a constant $c$ such that $|(F |_{\ell,m}\gamma)(Z)|\leq c$ for all $\gamma$ in $\Sp_4(\Z)$ and all $Z$ in the Siegel set.

For any point $Z\in\HH_2$ there exists $\gamma\in\Sp_4(\Z)$ such that $Z':=\gamma Z$ lies in the Siegel set. One can prove that there exists a positive integer $n$, depending on $\ell$ and $m$, such that
$$
 \|\eta_{\ell,m}(J(\gamma,Z))^{-1}\|\leq({\rm tr}(Y)+{\rm tr}(Y^{-1}))^n
$$
for all $\gamma\in\Sp_4(\Z)$ and all $Z=X+iY\in\HH_2$.
This shows that $|F(Z)|\leq c({\rm tr}(Y)+{\rm tr}(Y^{-1}))^n$ for all $Z\in\HH_2$. It then follows easily that $\Phi$ is slowly increasing. 

Since, by Table \ref{Nplusoperatorsnearholotable}, the operators $D_-$ and $P_{0-}$ lower the nearly holomorphic degree, we have $P_{0-}^sF=D_-^tF=0$ for some $s,t>0$. By the diagram \eqref{Nplusoperatorsnearholopropeq1}, it follows that $P_{0-}^s\Phi=D_-^t\Phi=0$. Hence, we can apply Lemma \ref{nfinitenesslemma} and conclude that $\Phi$ is $\mathfrak{n}$-finite, and generates an admissible $(\mathfrak{g},K)$-module. Since each weight space is finite-dimensional by Lemma \ref{nfinitenesslemma}, it follows that $\Phi$ is $\mathcal{Z}$-finite. This proves $\Phi\in\AA(\Gamma)_{\mathfrak{n}\text{-fin}}$.

Finally, we sketch a proof that $\Phi\in\AA(\Gamma)^\circ_{\mathfrak{n}\text{-fin}}$  if and only if $F$ is a cusp form. Since we already know that $\Phi\in\AA(\Gamma)_{\mathfrak{n}\text{-fin}}$, it suffices to show that $\Phi$ is a cusp form  if and only if $F$ is a cusp form.
First, start with a cusp form $F \in N_{\ell, m}(\Gamma)^\circ$. For any $\gamma \in \Sp_4(\Q)$, define $F_\gamma := F|_{\ell,m} \gamma$. Then $F_\gamma \in N_{\ell, m}(\Gamma(N_\gamma))^\circ$ for some integer $N_\gamma$, and its  Fourier expansion is supported on positive definite matrices. For any $x \in \R$, define $n_x = \left[\begin{smallmatrix}1&0&0&0\\0&1&0&x\\0&0&1&0\\0&0&0&1
 \end{smallmatrix}\right]$ and put $N'(\R):= \{n_x : x \in \R\}$. Note that $N'(\R)$ is contained in the unipotent radicals of all the (standard-form) maximal parabolic subgroups  of $\Sp_4(\R)$.  In order to show that $\Phi$ is a cusp form, it suffices to show that for all $\gamma \in \Sp_4(\Q)$, $g \in \Sp_4(\R)$, we have  $\int_{N_\gamma\Z \bs \R} \Phi(\gamma n_x g) dx = 0$. This is implied by $\int_{N_\gamma\Z\bs \R} F_\gamma\left(Z + \left[\begin{smallmatrix}0&0\\0&x \end{smallmatrix}\right]\right) dx = 0$, which in turn follows from the fact that  Fourier coefficients of $F_\gamma$ vanish at all matrices of the form $\left[\begin{smallmatrix}a&b/2\\b/2&0 \end{smallmatrix}\right]$.

  Conversely, suppose that $F$ is \emph{not} a cusp form. We will show that $\Phi$ is not a cusp form. We can find $\gamma \in \Sp_4(\Q)$ such that $F_\gamma:= F|_{\ell,m} \gamma$ has a non-vanishing degenerate Fourier coefficient. By choosing $\gamma$ appropriately, we can assume that the Fourier coefficient of $F_\gamma$ is non-vanishing at a matrix of the form $Q_t:=\frac{1}{N_\gamma}\left[\begin{smallmatrix}0&0\\0&t
 \end{smallmatrix}\right]$ with $t \in \Z$. Fix such a matrix $Q_t$.
Write the Fourier expansion as \begin{equation}
 F_\gamma(Z)=\sum_{Q \in (1/N_\gamma)M_2^{\sym}(\Z)}a(Q)e^{2\pi i\,{\rm Tr}(QZ)},
\end{equation}
where
\begin{equation}
 a(Q)=\sum_{j=0}^m\:\sum_{\alpha,\beta,\gamma}a_{j,\alpha,\beta,\gamma}(Q)
 \Big(\frac{y}{yy'-v^2}\Big)^\alpha\Big(\frac{v}{yy'-v^2}\Big)^\beta\Big(\frac{y'}{yy'-v^2}\Big)^\gamma S^{m-j}T^j.
\end{equation}
Since $a(Q_t) \neq 0$, we can find positive real numbers $r_2$, $r_3$ such that the element $f(y)$ in $\C(y)[S, T]$ given by $$f(y) =\sum_{j=0}^m\:\sum_{\alpha,\beta,\gamma}a_{j,\alpha,\beta,\gamma}(Q_t)
 \Big(\frac{y}{yr_3 - r_2^2}\Big)^\alpha\Big(\frac{r_2}{yr_3 - r_2^2}\Big)^\beta\Big(\frac{r_3}{yr_3 - r_2^2}\Big)^\gamma S^{m-j}T^j $$ is not identically zero. Fix such $r_2, r_3$. For $y> r_2^2/r_3$, define the matrix $$M_y := \frac{1}{\sqrt{y+r_3+2 \sqrt{yr_3 - r_2^2}}} \mat{y+ \sqrt{yr_3 - r_2^2}}{r_2}{r_2}{r_3 + \sqrt{yr_3 - r_2^2}}$$ which is a square root of the matrix $ \mat{y}{r_2}{r_2}{r_3}$. Then  there is a real number $\alpha$ and some positive real $c$, both independent of $y$, such that \begin{equation}\label{e:keylowebd}|\eta_{\ell, m}(M_y)f(y)| > cy^\alpha\end{equation} holds  for all large enough $y$. Above, $| \ |$ is the norm on $\C[S,T]$ defined by $|\sum_{i,j}a_{i,j} S^i T^j | := \max_{i,j} |a_i|$, though any other norm on $\C[S,T]$ will work equally well.

Now let $\Phi_\gamma(g)= \Phi(\gamma g)$ be the automorphic form attached to $F_\gamma$. Consider the integral \begin{align*}&\int_{M_2^{\sym}(N_\gamma \Z \bs \R)}\vec\Phi_\gamma\left( \left[\begin{smallmatrix}I_2&X\\0&I_2 \end{smallmatrix}\right]\left[ \begin{smallmatrix}M_y&0\\0&M_y^{-1} \end{smallmatrix}\right]\right) e^{-2 \pi i \Tr\left(Q_t X\right) } dX\\&=\eta_{\ell,m}(M_y)\int_{M_2^{\sym}(N_\gamma \Z \bs \R)}F_\gamma\left( X + i  \left[\begin{smallmatrix}y&r_2\\r_2&r_3 \end{smallmatrix}\right]  \right) e^{-2 \pi i \Tr\left(Q_t X\right) } dX \\ &= \vol(M_2^{\sym}(N_\gamma \Z \bs \R)) \ e^{-\frac{2 \pi t r_3}{N_\gamma}} (\eta_{\ell,m}(M_y)  f(y)).\end{align*}

Now, using the triangle inequality on the above integral, together with \eqref{e:keylowebd},  we conclude that for all sufficiently large $y$, we can find $X \in M_2^\sym(\R)$ such that the matrix $$g = \left[\begin{smallmatrix}I_2&X\\0&I_2 \end{smallmatrix}\right]\left[ \begin{smallmatrix}M_y&0\\0&M_y^{-1} \end{smallmatrix}\right]$$ has the property $|\vec\Phi_\gamma(g)| > c e^{-\frac{2 \pi t r_3}{N_\gamma}} y^\alpha$. Letting $y \rightarrow \infty$, we conclude that  $|\vec\Phi_\gamma|$ is \emph{not} a rapidly decaying function on $\Sp_4(\R)$. Hence, $\Phi_\gamma$ is not a cusp form and therefore neither is $\Phi$. This completes the proof.
\end{proof}
\subsection{The structure theorem for cusp forms}
In this section we prove the structure theorem for cusp forms. It is based on the following decomposition of the space $\AA(\Gamma)^\circ_{\mathfrak{n}\text{-fin}}$ into irreducibles.

\begin{proposition}\label{AA0nfindecomp2prop}
 As $(\mathfrak{g},K)$-modules, we have
 $$
  \AA(\Gamma)^\circ_{\mathfrak{n}\text{-fin}}=\bigoplus_{\ell=1}^\infty\,\bigoplus_{m=0}^\infty\,n_{\ell,m}L(\ell+m,\ell),\qquad n_{\ell,m}=\dim S_{\ell,m}(\Gamma).
 $$
 The lowest weight vectors in the isotypical component $n_{\ell,m}L(\ell+m,\ell)$ correspond to elements of $S_{\ell,m}(\Gamma)$ via the isomorphism from Lemma \ref{FPhilemma}.
\end{proposition}
\begin{proof}
Since $\AA(\Gamma)^\circ_{\mathfrak{n}\text{-fin}}$ is a $(\mathfrak{g},K)$-submodule of $\AA(\Gamma)^\circ$, it decomposes into an orthogonal direct sum of irreducible $(\mathfrak{g},K)$-modules, each occurring with finite multiplicity. Recall from Lemma \ref{locnfinLlambdalemma} that the only irreducible, locally $\mathfrak{n}$-finite $(\mathfrak{g},K)$-modules are the $L(\lambda)$ for $\lambda\in\Lambda$. Since $\AA(\Gamma)^\circ_{\mathfrak{n}\text{-fin}}$ admits the inner product \eqref{AA0innerproducteq}, each $L(\lambda)$ occurring in the decomposition of $\AA(\Gamma)^\circ_{\mathfrak{n}\text{-fin}}$ is unitarizable. The trivial $(\mathfrak{g},K)$-module $L(0,0)$ cannot occur, since constant functions are not cuspidal. Proposition \ref{Mlambdaunitaryprop} (3) therefore implies that only $L(\ell+m,\ell)$ with $\ell\geq1$ can occur. The module $L(\ell+m,\ell)$ must occur with multiplicity $\dim S_{\ell,m}(\Gamma)$, since every lowest weight vector in its isotypical component gives rise to an element of
$S_{\ell,m}(\Gamma)$, and conversely.
\end{proof}
\begin{remark}
 By (2) of Proposition \ref{Mlambdaunitaryprop}, the modules $L(1+m,1)$ are non-tempered. Still, it is possible for these modules to occur in $\AA(\Gamma)^\circ_{\mathfrak{n}\text{-fin}}$ for certain $\Gamma$. After all, cusp forms of weight $1$ do exist; see \cite{Weissauer1992}. Globally, the modules $L(1+m,1)$ occur in CAP representations with respect to the Borel or Klingen parabolic subgroup, which were considered in \cite{Soudry1988}. Therefore, these modules have to be excluded from any correct formulation of the Ramanujan conjecture.
\end{remark}

Recall from Lemma \ref{slashliecommute} that $\mathcal{X}$ denotes the free monoid consisting of all strings of the symbols in the left column of Table~\ref{Nplusoperatorstable}. For integers $\ell,m,\ell',m'$, we define the following subsets of $\mathcal{X}$. If $\ell\geq\ell'\geq2$, $m \ge 0$, $m' \ge 0$, then let
\begin{align}\label{Xlmlmdefeq1}
  \mathcal{X}_{\ell',m'}^{\ell,m}=\:&\Big\{X_+^\alpha D_+^\beta U^\gamma\:\big|\:\alpha,\beta,\gamma\in\Z_{\geq0},\:\gamma\leq m'/2,\nonumber\\
   &\hspace{8ex}\ell'+m'+2\alpha+2\beta=\ell+m,\:\ell'+2\beta+2\gamma=\ell\Big\}\nonumber\\
   \cup\:&\Big\{E_+X_+^\alpha D_+^\beta U^\gamma\:\big|\:\alpha,\beta,\gamma\in\Z_{\geq0},\:\gamma<m'/2,\nonumber\\
   &\hspace{8ex}\ell'+m'+2\alpha+2\beta+1=\ell+m,\:\ell'+2\beta+2\gamma+1=\ell\Big\}.
\end{align}
If $\ell\geq\ell'=1$, $m \ge 0$, $m' \ge 0$, then let
\begin{equation}\label{Xlmlmdefeq2}
  \mathcal{X}_{\ell',m'}^{\ell,m}=\begin{cases}
       \emptyset&\text{if $m'>m$ or }m\not\equiv m'\bmod2,\\[2ex]
       \Big\{X_+^{\frac{m-m'}2}D_+^{\frac{\ell-1}2}\Big\}&\text{if }m'\leq m,\:m\equiv m'\bmod2,\text{ and $\ell$ is odd},\\[2ex]
       \Big\{E_+X_+^{\frac{m-m'}2}D_+^{\frac{\ell-2}2}\Big\}&\text{if }m'\leq m,\:m\equiv m'\bmod2,\text{ and $\ell$ is even}.
      \end{cases}
\end{equation}
In every other case we put $\mathcal{X}_{\ell',m'}^{\ell,m}= \emptyset$, except for $\mathcal{X}_{0,0}^{0,0}$ which we put equal to $\{1\}$.

With these notations we are now ready to prove one of our main results.
\begin{theorem}[Structure theorem for cusp forms]\label{cuspidalstructuretheorem}
 Let $\ell$ be an integer, and $m$ a non-negative integer. Then we have an orthogonal direct sum decomposition
 \begin{equation}\label{cuspidalstructuretheoremeq1}
  N_{\ell,m}(\Gamma)^\circ=\bigoplus_{\ell'=1}^\ell
  \bigoplus_{m'=0}^{\ell+m-\ell'}\;\sum_{X \in \mathcal{X}_{\ell', m'}^{\ell, m}} X(S_{\ell',m'}(\Gamma)).
 \end{equation}
\end{theorem}
\begin{proof}
Let $F\in N_{\ell,m}(\Gamma)^\circ$. Let $\Phi:\:\Sp_4(\R)\to\C$ be the function corresponding to $F$ via Lemma \ref{FPhilemma}. By Proposition \ref{Nplmautformprop}, we have $\Phi\in\AA(\Gamma)^\circ_{\mathfrak{n}\text{-fin}}$. According to Proposition \ref{AA0nfindecomp2prop}, we can write $\Phi=\sum_{j=1}^r\Phi_j$ with non-zero $\Phi_j$ of weight $(\ell+m,\ell)$ and lying in an irreducible submodule $L(\ell_j+m_j,\ell_j)$ of $\AA(\Gamma)^\circ_{\mathfrak{n}\text{-fin}}$. Since $N_+\Phi=0$, we have $N_+\Phi_j=0$ for all $j$. Considering the possible $K$-types of the $L(\lambda)$ given in Lemma \ref{MlambdaKtypeslemma}, we see $\ell_j\leq\ell$ and $\ell_j+m_j\leq\ell+m$ for all $j$.

Let $\Psi_j$ be a vector of weight $(\ell_j+m_j,\ell_j)$ in $L(\ell_j+m_j,\ell_j)$. By Propositions \ref{navigateKtypesprop} and \ref{navigateKtypesell1prop}, we can navigate from $\Psi_j$ to $\Phi_j$ using the operators $ {U}$, $X_+$, $D_+$ and $E_+$.
The functions $\Psi_j$ correspond to elements of $S_{\ell_j,m_j}(\Gamma)$. The commutativity of the diagram \eqref{Nplusoperatorsnearholopropeq2} allows us to rewrite the relations in terms of functions on $\HH_2$. This proves the theorem.
\end{proof}

\begin{corollary}\label{cuspidalstructuretheoremcor1}
 Let $\ell$ be an integer, and $m$ a non-negative integer. Then
 $$
  N_{\ell,m}(\Gamma)^\circ=N^p_{\ell,m}(\Gamma)^\circ\qquad\text{with }p=\ell-1+\Big\lfloor\frac m2\Big\rfloor.
 $$
\end{corollary}
\begin{proof}
Consider a typical term $X_+^\alpha D_+^\beta {U}^\gamma\,S_{\ell',m'}(\Gamma)$ appearing in the structure theorem. By Table \ref{Nplusoperatorsnearholotable}, such a term can produce nearly holomorphic degrees no larger than $\alpha+2\beta+\gamma$. By the conditions in the first set in \eqref{Xlmlmdefeq1},
$$
 \alpha+2\beta+\gamma=\ell-\ell'+\frac{m-m'}2\leq\ell-2+\frac m2.
$$
Similarly we can estimate the nearly holomorphic degree of all the terms in the structure theorem. The maximal number is $\ell-1+\frac m2$, proving our result.
\end{proof}

\begin{corollary}[Structure theorem for scalar-valued cusp forms]\label{cuspidalstructuretheoremcor2}
Let $\ell$ be an integer. Then we have an orthogonal direct sum decomposition
\begin{align*}
 N_{\ell,0}(\Gamma)^\circ&=\bigoplus_{\substack{\ell'=2\\\ell'\equiv\ell\bmod2}}^\ell\;
 \bigoplus_{\substack{m'=0\\m'\equiv0\bmod2}}^{\ell-\ell'}\;D_+^{(\ell-\ell'-m')/2}\, {U}^{m'/2}\,S_{\ell',m'}(\Gamma)\;\oplus\;N_{\ell,0}(\Gamma)^\circ_1,
\end{align*}
where
$$
  N_{\ell,0}(\Gamma)^\circ_1=
  \begin{cases}
    \displaystyle D_+^{(\ell-1)/2}\,S_{1,0}(\Gamma)&\text{if $\ell$ is odd},\\[1ex]
    0&\text{if $\ell$ is even}.
  \end{cases}
$$
\end{corollary}
\begin{proof}
The terms of the decomposition in Theorem \ref{cuspidalstructuretheorem} simplify for $m=0$. Note that all the $ {E}^+$ terms are zero by (3) of Proposition \ref{navigateKtypesprop} and (2) of Proposition \ref{navigateKtypesell1prop}.
\end{proof}

\begin{corollary}[Structure theorem for scalar-valued cusp forms of bounded nearly holomorphic degree]\label{cuspidalstructuretheoremcor3}
Let $\ell$ be an integer. Then, for each $p\ge 0$, we have an orthogonal direct sum decomposition
\begin{align*}
 N_{\ell,0}^p(\Gamma)^\circ&=\bigoplus_{\substack{\ell'=\max(2, \ell-2p)\\\ell'\equiv\ell\bmod2}}^\ell\;\bigoplus_{\substack{m'=\max(0, 2(\ell - \ell' -p))\\m'\equiv0\bmod2}}^{\ell-\ell'}\;D_+^{\frac{\ell-\ell'-m'}2}\, {U}^{\frac{m'}{2}}\,S_{\ell',m'}(\Gamma)\;\oplus\;N_{\ell,0}^p(\Gamma)^\circ_1,
\end{align*}
where
$$
  N_{\ell,0}^p(\Gamma)^\circ_1=
  \begin{cases}
    \displaystyle D_+^{(\ell-1)/2}\,S_{1,0}(\Gamma)&\text{if $\ell$ is odd and $p \ge \ell -1$},\\[1ex]
    0&\text{otherwise}.
  \end{cases}
$$
\end{corollary}
\begin{proof}
The fact that the right side is contained in the left side follows immediately from Table \ref{Nplusoperatorsnearholotable}. Next, let $F \in  N_{\ell,0}^p(\Gamma)^\circ$. By Corollary~\ref{cuspidalstructuretheoremcor2}, we can write
$$
 F = \sum_{\ell', m'}  D_+^{\frac{\ell-\ell'-m'}2}\, {U}^{\frac{m'}{2}} F_{\ell', m'} \ + \ D_{+}^{(\ell-1) /2} F_{1,0},
$$
where $F_{\ell', m'} \in  S_{\ell',m'}(\Gamma)$ and $F_{1,0} \in S_{1,0}(\Gamma)$ (with $F_{1,0} = 0$ if $\ell$ is even). To complete the proof, it suffices to show that each $\ell'$, $m'$ above with  $F_{\ell', m'} \neq 0$ satisfies $\ell-\ell' - \frac{m'}2 \le p$, and furthermore, that $F_{1,0} \neq 0$ implies $p \ge \ell-1$.

We show that $F_{\ell', m'} \neq 0$ implies $\ell-\ell' - \frac{m'}2 \le p$; the proof for the other inequality is similar. Suppose that $\ell-\ell' - \frac{m'}2 >p$. Then, using Table \ref{Nplusoperatorsnearholotable}, we see that  ${P^{m'/2}_{0-}}D_-^{(\ell-\ell'-m')/2}F = 0$. This implies that $${P^{m'/2}_{0-}}D_-^{(\ell-\ell'-m')/2} D_+^{(\ell-\ell'-m')/2}\, {U}^{m'/2}\,F_{\ell', m'} = 0.$$ But this contradicts Lemma~\ref{Dminuslemma}.
\end{proof}
\subsection{Petersson inner products}\label{peterssonsec}
Let $\ell$ be an integer, and $m,p$ be non-negative integers. We let $\langle\cdot,\cdot\rangle_m$ be the unique $U(2)$-invariant inner product on $W_m$ such that $\langle S^m , S^m \rangle_m = 1$. Let $\Gamma$ be a congruence subgroup of $\Sp_4(\Q)$. For $F,G \in N_{\ell, m}(\Gamma)$, we define the Petersson inner product $\langle F, G \rangle$ by
$$
 \langle F, G \rangle = \vl(\Gamma \bs \H_2)^{-1}\int\limits_{\Gamma \bs \H_2}\langle \eta_{\ell,m}\left( \Im(Z) \right)(F(Z)), G(Z)\rangle_m\,dZ,
$$
where $dZ$ is any invariant measure on $\H_2$,\emph{ provided the integral converges absolutely}. We denote this absolute convergence condition by $\langle F, G \rangle<\infty$. If the integral does not converge absolutely, we write $\langle F, G \rangle=\infty$.

\begin{remark}Let $F,G \in N_{\ell, m}(\Gamma)$ such that at least one of $F$ and $G$ lies in $N_{\ell, m}(\Gamma)^\circ$. Then $\langle F, G\rangle <\infty$. \end{remark}

\begin{lemma}\label{peterssonequallemma2}
 Let $\ell$ be an integer, and $m$ a non-negative integer. Let $F,G \in N_{\ell, m}^p(\Gamma)$ and let $\Phi_F$, $\Phi_G$ be the functions on $\Sp_4(\R)$ corresponding to $F$, $G$ respectively via Lemma~\ref{FPhilemma}. Suppose that $\langle F, G \rangle<\infty$. Then $\langle \Phi_F, \Phi_G \rangle <\infty$ and
 \begin{equation}\label{peterssonequallemma2eq1}
  \langle F, G \rangle = \langle \Phi_F, \Phi_G \rangle,
 \end{equation}
 where $\langle \Phi_F, \Phi_G \rangle$ is defined by \eqref{AA0innerproducteq}.
\end{lemma}
\begin{proof}
This follows from a standard computation as in~\cite[p.\ 195]{Asch}. We omit the details.
\end{proof}

We define the subspace $\E_{\ell, m}(\Gamma)$ to be the orthogonal complement of $N_{\ell, m}(\Gamma)^\circ$ in $N_{\ell, m}(\Gamma)$ with respect to the Petersson inner product.
\begin{lemma}\label{slashEorthogonallemma}
 Let $\ell,m$ be non-negative integers. Let $F\in\E_{\ell,m}(\Gamma)$, and let $\Phi\in\AA(\Gamma)_{\mathfrak{n}\text{-fin}}$ be the function corresponding to $F$ via Lemma \ref{FPhilemma}. Then $\Phi$ is orthogonal to $\AA(\Gamma)_{\mathfrak{n}\text{-fin}}^\circ$.
\end{lemma}
\begin{proof}
Let $\Psi\in\AA(\Gamma)_{\mathfrak{n}\text{-fin}}^\circ$; we have to show that $\langle\Phi,\Psi\rangle=0$. We may assume that $\Psi$ generates an irreducible module $L(\lambda)$ for some $\lambda$. Since, under the $K$-action, $\Phi$ generates the $K$-type $\rho_{(\ell+m,\ell)}$, we may assume that $\Psi$ does as well. Writing $\Psi$ as a sum of weight vectors, we may even assume that $\Psi$ has the same weight as $\Phi$, namely $(\ell+m,\ell)$. But then $\Psi$ corresponds to an element $G$ of $N_{\ell,m}(\Gamma)^\circ$. By hypothesis $\langle F,G\rangle=0$. Hence $\langle\Phi,\Psi\rangle=0$ by Lemma \ref{peterssonequallemma2}.
\end{proof}

\begin{lemma}\label{preserveeissiegel}
 Let $X$, $\mathcal{X}$ be as in Lemma~\ref{slashliecommute}. Then $X$ takes $N_{\ell, m}(\Gamma)^\circ$ to $ N_{\ell', m'}(\Gamma)^\circ$ and $\E_{\ell, m}(\Gamma)$ to $\E_{\ell', m'}(\Gamma)$.
\end{lemma}
\begin{proof}
The fact that $X$ takes $N_{\ell, m}(\Gamma)^\circ$ to $ N_{\ell', m'}(\Gamma)^\circ$  is an immediate consequence of the fact that $X$ does not introduce new Fourier coefficients (this is true for each operator in Table~\ref{Nplusoperatorstable} by~\eqref{Fourierexpansioneq6} and is therefore true for all elements of  $\mathcal{X}$).

To prove that $X$ takes $\E_{\ell, m}(\Gamma)$ to $ \E_{\ell', m'}(\Gamma)$, let $F\in\E_{\ell, m}(\Gamma)$, and let $\Phi\in\AA(\Gamma)_{\mathfrak{n}\text{-fin}}$ be the corresponding automorphic form. By Lemma \ref{slashEorthogonallemma}, $\Phi$ is orthogonal to $\AA(\Gamma)_{\mathfrak{n}\text{-fin}}^\circ$. Hence the entire $(\g, K)$-module $\U(\g_\C)\Phi$ is orthogonal to $\AA(\Gamma)_{\mathfrak{n}\text{-fin}}^\circ$. Since $XF$ corresponds to $X\Phi\in\U(\g_\C)\Phi$, our assertion follows.
\end{proof}
\begin{lemma}\label{peterssonequivlemma}
 Let $\ell$ be a positive integer, and $m$ a non-negative integer. Let $X$, $\mathcal{X}$ be as in Lemma~\ref{slashliecommute}. There exists a constant $c_{\ell, m, X}$ (depending only on $\ell$, $m$, $X$) such that for all $F \in S_{\ell, m}(\Gamma)$ we have $\langle XF,XF \rangle  = c_{\ell,m,X}\langle F ,F\rangle$.
\end{lemma}
\begin{proof}
Set $\lambda=(\ell+m,\ell)$, and consider the $(\mathfrak{g},K)$-module $L(\lambda)$. Let $v_0$ be a highest weight vector in the minimal $K$-type of $L(\lambda)$; of course,  $v_0$ is unique up to multiples. Since $L(\lambda)$ is unitary by Proposition \ref{Mlambdaunitaryprop}, we may endow it with a $\mathfrak{g}$-invariant inner product $\langle\cdot,\cdot\rangle$. By irreducibility, this inner product is unique up to multiples. Put $c_{\ell,m,X}=\langle Xv_0,Xv_0\rangle/\langle v_0,v_0\rangle$. Note that $c_{\ell,m,X}$ does not depend on the choice of model for $L(\lambda)$, the choice of $v_0$, or the normalization of inner products.

Now all we need to observe is that the automorphic form $\Phi\in\AA(\Gamma)_{\mathfrak{n}\text{-fin}}^\circ$ corresponding to $F$ generates a module isomorphic to $L(\lambda)$, that $\Phi$ is a lowest weight vector in this module, and Lemma \ref{peterssonequallemma2}.
\end{proof}

\begin{proposition}\label{peterssonequivprop}
 Let $\ell$ be a positive integer, and $m$ a non-negative integer. Let $X$, $\mathcal{X}$ be as in Lemma~\ref{slashliecommute}. Then, for all $F \in S_{\ell, m}(\Gamma)$ and $G \in M_{\ell, m}(\Gamma)$, we have $\langle X F,XG \rangle  = c_{\ell,m,X}\langle F , G\rangle$,
 where the constant $c_{\ell,m,X}$ is the same as in Lemma \ref{peterssonequivlemma}.
\end{proposition}
\begin{proof}
Because of Lemma~\ref{preserveeissiegel}, we may assume that $F$ and $G$ both belong to $S_{\ell,m}(\Gamma)$. Now the proposition follows by applying the previous lemma to $F+G$.
\end{proof}

%
%
%
\subsection{Initial decomposition in the general case}\label{initialdecompsec}
As before, we fix a congruence subgroup $\Gamma$, and consider the space $\AA(\Gamma)_{\mathfrak{n}\text{-fin}}$ of $\mathfrak{n}$-finite automorphic forms. In this and the next sections we investigate the algebraic structure of this $(\mathfrak{g},K)$-module. We know from Proposition \ref{AA0nfindecomp2prop} that the subspace of cusp forms is completely reducible. Since there is no inner product defined on all of $\AA(\Gamma)_{\mathfrak{n}\text{-fin}}$, this may no longer be true for the entire space. The following vanishing result for Siegel modular forms will imply some basic restrictions on the possible $K$-types occurring in $\AA(\Gamma)_{\mathfrak{n}\text{-fin}}$.

\begin{lemma}\label{weissauerlemma}
 Let $\ell,m\in\Z$ with $m\geq0$. Assume that $M_{\ell,m}(\Gamma)\neq0$. Then $\ell\geq1$ or $\ell=m=0$. The space $M_{0,0}(\Gamma)$ consists only of the constant functions.
\end{lemma}
\begin{proof}
The first statement follows from the vanishing theorem Satz 2 of \cite{Weissauer1983}. The second statement says that the only holomorphic modular forms of weight $0$ are the constant functions; this is well known.
\end{proof}
\begin{lemma}\label{nonegelllemma}
 The space $\AA(\Gamma)_{\mathfrak{n}\text{-fin}}$ does not contain any weights $(k,\ell)$ with negative $\ell$. It contains the weight $(0,0)$ with multiplicity one; the corresponding weight space consists precisely of the constant functions.
\end{lemma}
\begin{proof}
To prove the first statement, suppose that $\AA(\Gamma)_{\mathfrak{n}\text{-fin}}$ contains a non-zero vector $\Phi$ of weight $(k,\ell)$ with $\ell<0$. After applying $P_{0-}$, $P_{1-}$ and $X_-$ finitely many times to $\Phi$, we may assume that $\Phi$ is annihilated by all these operators. By Corollary \ref{PhiFholomorphiclemma}, $\Phi$ corresponds to a non-zero element $F$ of $M_{\ell,k-\ell}(\Gamma)$. But such $F$ do not exist by Lemma \ref{weissauerlemma}.

To prove the second statement, let $\Phi$ be a vector of weight $(0,0)$. By the first statement, $P_{1-}\Phi=P_{0-}\Phi=N_+\Phi=0$. Hence also $N_-\Phi=0$. Since $[N_-,P_{1-}]=2X_-$, then also $X_-\Phi=0$. Therefore $\Phi$ corresponds to an element of $M_{0,0}(\Gamma)$. By Lemma \ref{weissauerlemma}, $\Phi$ must be constant.
\end{proof}

\begin{lemma}\label{weight20lemma}
 The space $\AA(\Gamma)_{\mathfrak{n}\text{-fin}}$ does not contain the weight $(2,0)$.
\end{lemma}
\begin{proof}
Suppose that $\Phi\in\AA(\Gamma)_{\mathfrak{n}\text{-fin}}$ is a non-zero vector of weight $(2,0)$; we will obtain a contradiction. Since $\AA(\Gamma)_{\mathfrak{n}\text{-fin}}$ does not contain any weights $(k,\ell)$ with negative $\ell$, we have $E_-\Phi=P_{1-}\Phi=P_{0-}\Phi=N_+\Phi=0$. By Lemma \ref{weissauerlemma}, $\Phi$ cannot be annihilated by all of $\p_-$. Hence $X_-\Phi\neq0$. Since the formula for the $L$-operator in Table \ref{Nplusoperatorstable} can be rewritten as
$$
 L=m(m+1)X_--(m+1)N_-P_{1-}+N_-^2P_{0-},
$$
it follows that $L\Phi\neq0$. Since $L\Phi$ has weight $(0,0)$, it is a constant function by Lemma \ref{nonegelllemma}. We normalize such that $L\Phi=-6$; the reason for this normalization will become clear momentarily.

Let $F:\:\HH_2\to W_2$ be the function corresponding to $\Phi$. Let $F=F_0S^2+F_1ST+F_2T^2$, where $F_j$ are the component functions. By Proposition \ref{Nplusoperatorsnearholoprop} (1), the relations $E_-F=P_{0-}F=0$ and $LF=-6$ hold. Looking at the definitions \eqref{operatorsonfunctions1}, \eqref{operatorsonfunctions2}, \eqref{operatorsonfunctions6} of these differential operators, and solving the resulting linear system, we conclude
\begin{equation}\label{weight20lemmaeq1}
    F_j=\big[1,0,0\big]-2\big[0,1,0\big]+\big[0,0,1\big]+H_j
\end{equation}
for $j\in\{0,1,2\}$, where $H_j$ is holomorphic. (See \eqref{nearholomonomialeq} for notation.) Now consider the function on $\HH_2$ given by $G(Z):=F(Z)-2F(2Z)$. Then $G(Z)$ is a modular form with respect to a smaller congruence subgroup $\Gamma'$. It is easy to see that $G$ is non-zero. In view of \eqref{weight20lemmaeq1}, the nearly holomorphic parts of $F(Z)$ and $2F(2Z)$ cancel each other out, so that $G$ is holomorphic. Hence $G$ is a non-zero element of $M_{0,2}(\Gamma')$. By Lemma \ref{weissauerlemma}, this is impossible.
\end{proof}

For the next lemma, recall that $\AA_\lambda(\Gamma)_{\mathfrak{n}\text{-fin}}$ denotes the subspace of vectors of weight $\lambda\in\Lambda$.
\begin{lemma}\label{noncusplemma2}
 Let $\ell$ be an integer, and $m$ a non-negative integer.
 \begin{enumerate}
  \item $\AA_{(\ell+m,\ell)}(\Gamma)_{\mathfrak{n}\text{-fin}}=0$ if $\ell<0$.
  \item $\AA_{(0,0)}(\Gamma)_{\mathfrak{n}\text{-fin}}=\C$.
  \item $\AA_{(m,0)}(\Gamma)_{\mathfrak{n}\text{-fin}}=0$ for all $m>0$.
 \end{enumerate}
\end{lemma}
\begin{proof}
(1) and (2) were already noted in Lemma \ref{nonegelllemma}.

(3) By part (1) and Lemma \ref{weissauerlemma}, the operator $X_-$ induces injective maps
\begin{equation}\label{noncusplemma2eq1}
 \AA_{(m+2,0)}(\Gamma)_{\mathfrak{n}\text{-fin}}\longrightarrow\AA_{(m,0)}(\Gamma)_{\mathfrak{n}\text{-fin}}
\end{equation}
for each $m\geq0$. Clearly, $\AA_{(1,0)}(\Gamma)_{\mathfrak{n}\text{-fin}}=0$ by Lemma \ref{weissauerlemma}, and $\AA_{(2,0)}(\Gamma)_{\mathfrak{n}\text{-fin}}=0$ by Lemma \ref{weight20lemma}. Hence $\AA_{(m,0)}(\Gamma)_{\mathfrak{n}\text{-fin}}$ is zero for all $m>0$.
\end{proof}
For a character $\chi$ of $\mathcal{Z}$ (the center of $\mathcal{U}(\mathfrak{g}_\C)$) let $\AA(\Gamma,\chi)_{\mathfrak{n}\text{-fin}}$ be the subspace of $\AA(\Gamma)_{\mathfrak{n}\text{-fin}}$ consisting of vectors $\Phi$ with the property $(z-\chi(z))^n\Phi=0$ for all $z\in\mathcal{Z}$ and some $n$ depending on $z$.

\begin{lemma}\label{noncusplemma3}
 We have
 \begin{equation}\label{noncusplemma3eq1}
  \AA(\Gamma)_{\mathfrak{n}\text{-fin}}=\bigoplus_\chi\AA(\Gamma,\chi)_{\mathfrak{n}\text{-fin}}.
 \end{equation}
 Each space $\AA(\Gamma,\chi)_{\mathfrak{n}\text{-fin}}$ has finite length as a $(\mathfrak{g},K)$-module.
\end{lemma}
\begin{proof}
For a weight $\mu\in\Lambda$, let $\AA_{\succcurlyeq\mu}(\Gamma)_{\mathfrak{n}\text{-fin}}$ be the subspace of $\AA(\Gamma)_{\mathfrak{n}\text{-fin}}$ spanned by all vectors of weight $\lambda\succcurlyeq\mu$; see \eqref{dominanceordereq2} for the definition of the order. Since $\AA(\Gamma)_{\mathfrak{n}\text{-fin}}$ is admissible by Lemma \ref{admissibilitylemma}, and since there are no weights below a horizontal line by Lemma \ref{nonegelllemma}, the space $\AA_{\succcurlyeq\mu}(\Gamma)_{\mathfrak{n}\text{-fin}}$ is finite-dimensional. Therefore, the $(\mathfrak{g},K)$-module $\AA_{\langle\succcurlyeq\mu\rangle}(\Gamma)_{\mathfrak{n}\text{-fin}}$ generated by $\AA_{\succcurlyeq\mu}(\Gamma)_{\mathfrak{n}\text{-fin}}$ lies in category $\mathcal{O}^\p$. By general properties of this category, it admits a decomposition into $\chi$-isotypical components, as defined in \eqref{Mchidefeq}, each of which has finite length. If we move $\mu$ farther up and farther to the right, we will exhaust the whole space $\AA(\Gamma)_{\mathfrak{n}\text{-fin}}$.
The assertion follows.
\end{proof}

By Lemma \ref{noncusplemma3} (and Lemma \ref{locnfinLlambdalemma}), each $\AA(\Gamma,\chi)_{\mathfrak{n}\text{-fin}}$ has a finite length composition series whose irreducible quotients are of the form $L(\lambda)$ for some $\lambda\in\Lambdap$. Since $L(\lambda)$ has central character $\chi_{\lambda+\varrho}$, only those $\lambda$ with $\chi_{\lambda+\rho}=\chi$ can occur in $\AA(\Gamma,\chi)_{\mathfrak{n}\text{-fin}}$. For a given $\chi$, this allows for only finitely many $\lambda$. Lemma \ref{noncusplemma2} puts restrictions on the possible $L(\lambda)$'s that can occur; for example, $L(k,\ell)$ with $\ell<0$ can never occur in $\AA(\Gamma,\chi)_{\mathfrak{n}\text{-fin}}$. We will go through the list of $\chi$'s for which there exists at least one $L(\lambda)$ that is permitted by Lemma \ref{noncusplemma2}; evidently, only such $\chi$'s can occur in the decomposition \eqref{noncusplemma3eq1}:
\begin{itemize}
 \item The trivial character, i.e., $\chi=\chi_{\varrho}$, where $\varrho=(-1,-2)$. The irreducible modules $L(\lambda)$ that can occur as subquotients of $\AA(\Gamma,\chi_\varrho)_{\mathfrak{n}\text{-fin}}$ are $L(0,0)$ (the trivial representation), $L(3,1)$ and $L(3,3)$. (The module $L(2,0)$ also has central character $\chi_{\varrho}$, but is not permitted by (3) of Lemma \ref{noncusplemma2}). Following terminology in the literature, we call $\chi_\varrho$ the \emph{principal character}.
 \item The characters $\chi_{\lambda+\varrho}$ for $\lambda=(k,1)$ with $k\geq4$. The irreducible modules that can occur as subquotients of $\AA(\Gamma,\chi_{\lambda+\varrho})_{\mathfrak{n}\text{-fin}}$ are $L(k,1)$ and $L(k,3)$. Since the modules $L(k,1)$ are non-tempered by Proposition \ref{Mlambdaunitaryprop}, we will refer to these $\chi_{\lambda+\varrho}$ as \emph{non-tempered characters}.
 \item The character $\chi_{\lambda+\varrho}$ for $\lambda=(1,1)$. The irreducible modules that can occur as subquotients of $\AA(\Gamma,\chi_{\lambda+\varrho})_{\mathfrak{n}\text{-fin}}$ are $L(1,1)$ and $L(2,2)$.
 \item The character $\chi_{\lambda+\varrho}$ for $\lambda=(2,1)$. The only irreducible module that can occur as a subquotient of $\AA(\Gamma,\chi_{\lambda+\varrho})_{\mathfrak{n}\text{-fin}}$ is $L(2,1)$.
 \item The characters $\chi_{\lambda+\varrho}$ for $\lambda=(\ell+m,\ell)$ with ($\ell\geq4$, $m\geq0$), or ($\ell=2$, $m\geq1$). The only irreducible module that can occur as a subquotient of $\AA(\Gamma,\chi_{\lambda+\varrho})_{\mathfrak{n}\text{-fin}}$ is $L(\lambda)$. We will refer to these $\chi_{\lambda+\varrho}$ as the \emph{tempered characters}.
\end{itemize}

Our task in the following will be to determine the structure of each $\AA(\Gamma,\chi)_{\mathfrak{n}\text{-fin}}$ occurring in \eqref{noncusplemma3eq1}. We can quickly treat the case of tempered $\chi$. Since $L(\lambda)$ admits no non-trivial self-extensions by Proposition 3.1 (d) of \cite{Humphreys2008}, the component $\AA(\Gamma,\chi)_{\mathfrak{n}\text{-fin}}$ for tempered $\chi=\chi_{\lambda+\varrho}$ is a direct sum of copies of $L(\lambda)$. The lowest weight vector in such an $L(\lambda)$ corresponds to an element of $M_{\ell,m}(\Gamma)$, where $\lambda=(\ell+m,\ell)$. Thus,
\begin{equation}\label{temperedchareq}
 \AA(\Gamma,\chi)_{\mathfrak{n}\text{-fin}}=n_\lambda L(\lambda),\qquad n_\lambda=\dim M_{\ell,m}(\Gamma),
\end{equation}
for tempered $\chi=\chi_{\lambda+\varrho}$ with $\lambda=(\ell+m,\ell)$.

The third and fourth cases above can also be dealt with easily. We get \begin{equation}\label{21chareq}
 \AA(\Gamma,\chi)_{\mathfrak{n}\text{-fin}}=n_\lambda L(\lambda),\qquad n_\lambda=\dim M_{1,1}(\Gamma),
\end{equation}
for $\chi=\chi_{\lambda+\varrho}$ with $\lambda=(2,1)$, and
\begin{equation}\label{1122chareq}
 \AA(\Gamma,\chi)_{\mathfrak{n}\text{-fin}}=n_1L(1,1)\oplus n_2L(2,2),\qquad n_k=\dim M_{k,0}(\Gamma),
\end{equation}
for $\chi=\chi_{\lambda+\varrho}$ with $\lambda=(1,1)$.

As for the principal character, note that, by (2) of Lemma \ref{noncusplemma2}, the trivial module $L(0,0)$ occurs exactly once in $\AA(\Gamma)_{\mathfrak{n}\text{-fin}}$, and it occurs as a submodule. It is easy to see that $L(0,0)$ does not admit any non-trivial extensions with $L(3,1)$ or $L(3,3)$. It follows that
\begin{equation}\label{principalblockinitialeq}
 \AA(\Gamma,\chi_\varrho)_{\mathfrak{n}\text{-fin}}\cong L(0,0)\oplus V_3,
\end{equation}
where the module $V_3$ has a composition series with the only subquotients being $L(3,1)$ and $L(3,3)$. This module $V_3$ can be treated together with the non-tempered characters, which we will take up in the next section.
\subsection{The non-tempered characters}\label{nontemperedsec}
In this section we investigate the contribution to the space $\AA(\Gamma)_{\mathfrak{n}\text{-fin}}$ coming from non-tempered central characters, as defined in the previous section. Recall that these are the $\chi_{\lambda+\varrho}$ for $\lambda=(k,1)$ with $k\geq4$. The only irreducible $L(\lambda)$ that can occur as subquotients of such modules are $\lambda=(k,1)$ and $\lambda=(k,3)$.

In the following we will require the dual module $N(\lambda)^\vee$ of the module $N(\lambda)$. The basic properties of the duality functor on category $\mathcal{O}^{\mathfrak{p}}$ are given in Sects.~3.2 and 3.3 of \cite{Humphreys2008}. In particular, $N(\lambda)^\vee$ contains the same $K$-types as $N(\lambda)$, and admits $L(\lambda)$ as its unique irreducible submodule.

\begin{lemma}\label{extvanishinglemma}
 Let $k\geq3$ be an integer and $\lambda=(k,1)\in\Lambda$. Then
\begin{align}
 \label{extvanishinglemmaeq1}{\rm Ext}_{\mathcal{O}}(N(\lambda),N(\lambda))&=0,\\
 \label{extvanishinglemmaeq2}{\rm Ext}_{\mathcal{O}}(L(\lambda),L(\lambda))&=0,\\
 \label{extvanishinglemmaeq3}{\rm Ext}_{\mathcal{O}}(N(\lambda),L(\lambda))&=0,\\
 \label{extvanishinglemmaeq4}{\rm Ext}_{\mathcal{O}}(L(\lambda),N(\lambda))&=0,\\
 \label{extvanishinglemmaeq20}{\rm Ext}_{\mathcal{O}}(N(\lambda)^\vee,N(\lambda)^\vee)&=0,\\
 \label{extvanishinglemmaeq21}{\rm Ext}_{\mathcal{O}}(L(\lambda),N(\lambda)^\vee)&=0,\\
 \label{extvanishinglemmaeq22}{\rm Ext}_{\mathcal{O}}(N(\lambda)^\vee,L(\lambda))&=0.
\end{align}
\end{lemma}
\begin{proof}
Equations \eqref{extvanishinglemmaeq1} -- \eqref{extvanishinglemmaeq3} are general properties; see Proposition 3.1 of \cite{Humphreys2008}. The claim \eqref{extvanishinglemmaeq4} follows exactly as in the first part of the proof of Proposition 3.12 of \cite{Humphreys2008}. Equations \eqref{extvanishinglemmaeq20} -- \eqref{extvanishinglemmaeq22} follow from the previous ones and the properties of duality.\end{proof}

\begin{lemma}\label{k1k3moduleslemma}
 Let $k\geq3$ be an integer. Let $\lambda=(k,1)\in\Lambda$ and $\mu=(k,3)\in\Lambda$. Let $V$ be a module in category $\mathcal{O}^\p$ with the following properties:
 \begin{itemize}
  \item The only possible irreducible subquotients of $V$ are $L(\lambda)$ and $L(\mu)$.
  \item $V$ does not contain $N(\lambda)$ as a submodule.
 \end{itemize}
 Then
 \begin{equation}\label{k1k3moduleslemmaeq1}
  V=aL(\lambda)\,\oplus\,bL(\mu)\,\oplus\,cN(\lambda)^\vee.
 \end{equation}
 with non-negative integers $a,b,c$.
\end{lemma}
\begin{proof}
For $i=1$ or $i=3$, denote by $H_i$ the space of vectors $v\in V$ of weight $(k,i)$ that are annihilated by $N_+$. Then
\begin{equation}\label{k1k3moduleslemmaeq10}
 V=\mathcal{U}(\mathfrak{g}_\C)H_1+\mathcal{U}(\mathfrak{g}_\C)H_3.
\end{equation}
Indeed, the quotient $V/(\mathcal{U}(\mathfrak{g}_\C)H_1+\mathcal{U}(\mathfrak{g}_\C)H_2)$ does not contain the $K$-types $\rho_\lambda$ or $\rho_\mu$. Since its only possible irreducible subquotients are $L(\lambda)$ or $L(\mu)$, it must be zero.

Note that $P_{0-}$ induces a linear map $H_3\to H_1$. Let $H_3''$ be the kernel of this map. Let $H_3'$ be any vector space complement to $H_3''$ in $H_3$. Let $H_1'=P_{0-}(H_3')$, and let $H_1''$ be any vector space complement to $H_1'$ in $H_1$. Thus $H_3=H_3'\oplus H_3''$ and $H_1=H_1'\oplus H_1''$, and $P_{0-}$ induces an isomorphism $H_3'\to H_1'$. From \eqref{k1k3moduleslemmaeq10} we see that
\begin{equation}\label{k1k3moduleslemmaeq11}
 V=\mathcal{U}(\mathfrak{g}_\C)H_1''+\mathcal{U}(\mathfrak{g}_\C)H_3''+\mathcal{U}(\mathfrak{g}_\C)H_3'.
\end{equation}
Our hypothesis that $V$ does not contain $N(\lambda)$ as a submodule implies that every non-zero $v\in H_1$ generates a copy of $L(\lambda)$. From \eqref{extvanishinglemmaeq2} we thus conclude that $\mathcal{U}(\mathfrak{g}_\C)H_1''=aL(\lambda)$ (a direct sum) for some $a\geq0$. Similarly, $\mathcal{U}(\mathfrak{g}_\C)H_3''=bL(\mu)$ for some $b\geq0$. Our hypothesis, together with the PBW theorem, also implies that every non-zero $v\in H_3'$ generates a copy of $N(\lambda)^\vee$. Using \eqref{extvanishinglemmaeq20}, we get $\mathcal{U}(\mathfrak{g}_\C)H_3'=cN(\lambda)^\vee$ for some $c\geq0$. By \eqref{k1k3moduleslemmaeq11} we now have
\begin{equation}\label{k1k3moduleslemmaeq12}
 V=aL(\lambda)+bL(\mu)+cN(\lambda)^\vee.
\end{equation}
Lemma \ref{extvanishinglemma} implies that $aL(\lambda)+cN(\lambda)^\vee=aL(\lambda)\oplus cN(\lambda)^\vee$. Since
$$
 \Hom(L(\mu),aL(\lambda)\oplus cN(\lambda)^\vee)=0,
$$
the intersection of $bL(\mu)$ with $aL(\lambda)\oplus cN(\lambda)^\vee$ is zero. It follows that the sums in \eqref{k1k3moduleslemmaeq12} are direct.
\end{proof}

Let $\lambda=(k,1)$ and $\mu=(k,3)$ for some $k\geq3$. Let $\chi=\chi_{\lambda+\varrho}$. If $k\geq4$, then let $V_k=\AA(\Gamma,\chi)_{\mathfrak{n}\text{-fin}}$; hence, $V_k$ is the component appearing in the decomposition \eqref{noncusplemma3eq1} corresponding to the non-tempered character $\chi$. Let $V_3$ be the module appearing in \eqref{principalblockinitialeq}; hence, $V_3$ is ``almost'' $\AA(\Gamma,\chi_\varrho)_{\mathfrak{n}\text{-fin}}$, but without the trivial module.

\begin{lemma}\label{weissauerL2lemma}
 $V_k$ does not contain $N(\lambda)$ as a submodule.
\end{lemma}
\begin{proof}
Suppose that $V_k$ does contain $N(\lambda)$ as a submodule; we will obtain a contradiction.  Let $\Phi\in N(\lambda)$ be an automorphic form of weight $(k,1)$. Clearly, $\Phi$ generates $N(\lambda)$. Let $F\in M_{1,k-1}(\Gamma)$ be the holomorphic modular form corresponding to $\Phi$. By the Folgerung to Satz 3 of \cite{Weissauer1983}, the modular form $F$ is square-integrable. By Lemma~\ref{peterssonequallemma2}
the function $\Phi$ is square-integrable on $\Sp_4(\R)$. Since square-integrable automorphic forms constitute a $(\mathfrak{g},K)$-submodule of $\AA(\Gamma)_{\mathfrak{n}\text{-fin}}$, it follows that $N(\lambda)$ consists entirely of square-integrable forms, and hence admits an invariant inner product. In particular, we obtain the contradiction that $N(\lambda)$ is semisimple.
\end{proof}

For any $k\geq3$, the module $V_k$ admits only $L(\lambda)$ and $L(\mu)$ as irreducible subquotients. Therefore, by Lemma \ref{k1k3moduleslemma} and Lemma \ref{weissauerL2lemma},
\begin{equation}\label{Vkdecompeq2}
 V_k\cong aL(\lambda)\,\oplus\,bL(\mu)\,\oplus\,cN(\lambda)^\vee
\end{equation}
with certain multiplicities $a,b,c$. These multiplicities may be related to dimensions of spaces of modular forms, as follows. Any vector of weight $(k,1)$ in either $L(\lambda)$ or $N(\lambda)^\vee$ gives rise to an element of $M_{1,k-1}(\Gamma)$. Conversely, a non-zero element $F\in M_{1,k-1}(\Gamma)$ (or rather the function $\Phi$ on $\Sp_4(\R)$ corresponding to $F$) generates a copy of $L(\lambda)$ (which may lie inside an $N(\lambda)^\vee$). This explains the first of the following three equations,
\begin{align}
 \label{Vkmulteq1}a+c&=\dim M_{1,k-1}(\Gamma),\\
 \label{Vkmulteq2}b&=\dim M_{3,k-3}(\Gamma),\\
 \label{Vkmulteq3}b+c&=\dim M^*_{3,k-3}(\Gamma).
\end{align}
For the second equation, observe that any vector of weight $(k,3)$ in $L(\mu)$ gives rise to an element of $M_{3,k-3}(\Gamma)$. Conversely, a non-zero element $F\in M_{3,k-3}(\Gamma)$ generates a copy of $L(\mu)$.

The space appearing in \eqref{Vkmulteq3} is defined by
\begin{equation}\label{N1stardefeq}
 M^*_{3,k-3}(\Gamma)=\{F\in N_{3,k-3}(\Gamma)\:\big|\:LF=E_-F=0,\:P_{0-}F\text{ is holomorphic}\}.
\end{equation}
By (3) of Lemma \ref{holomorphyconditionlemma}, an alternative definition is
\begin{equation}\label{N1stardef2eq}
 M^*_{3,k-3}(\Gamma)=\{F\in N^1_{3,k-3}(\Gamma)\:\big|\:LF=E_-F=0\}.
\end{equation}
Evidently,
\begin{equation}\label{N1starinbetweeneq}
 M_{3,k-3}(\Gamma)\subset M^*_{3,k-3}(\Gamma)\subset N^1_{3,k-3}.
\end{equation}
We already noted that a vector of weight $(k,3)$ in $L(\mu)$ gives rise to an element of $M_{3,k-3}(\Gamma)$, and hence to an element of $M^*_{3,k-3}(\Gamma)$. We claim that a vector $\Phi$ of weight $(k,3)$ in $N(\lambda)^\vee$ also gives rise to an element of $M^*_{3,k-3}(\Gamma)$. Let $F$ be the smooth function on $\HH_2$ corresponding to $\Phi$. Then (3) of Lemma \ref{holomorphyconditionlemma} implies that $F$ is nearly holomorphic of degree $1$. Hence $F\in N^1_{3,k-3}(\Gamma)$. Clearly $F\in M^*_{3,k-3}(\Gamma)$, as claimed. Conversely, a non-zero $F\in M^*_{3,k-3}(\Gamma)$ generates either a copy of $L(\mu)$ or a copy of $N(\lambda)^\vee$. This proves \eqref{Vkmulteq3}.

Solving the linear system \eqref{Vkmulteq1} -- \eqref{Vkmulteq3}, we obtain the following result.
\begin{lemma}\label{Vkdecomplemma}
 For $k\geq3$, let $V_k$ be defined as above. Then we have the direct sum decomposition
 \begin{equation}\label{Vkdecomplemmaeq1}
  V_k\cong a_kL(\lambda)\,\oplus\,b_kL(\mu)\,\oplus\,c_kN(\lambda)^\vee,
 \end{equation}
 where
 \begin{align}
  \label{Vkdecomplemmaeq2}a_k&=\dim M_{1,k-1}(\Gamma)+\dim M_{3,k-3}(\Gamma)-\dim M^*_{3,k-3}(\Gamma),\\
  \label{Vkdecomplemmaeq3}b_k&=\dim M_{3,k-3}(\Gamma),\\
  \label{Vkdecomplemmaeq4}c_k&=\dim M^*_{3,k-3}(\Gamma)-\dim M_{3,k-3}(\Gamma).
 \end{align}
\end{lemma}
We note that the component $cN(\lambda)^\vee$ in \eqref{Vkdecomplemmaeq1} is not well-defined as a subspace of $V_k$; while the multiplicities of indecomposable modules are well-defined in category $\mathcal{O}^\p$, isotypical components are in general not. For example, if $\Phi$ has weight $(k,3)$ and generates an $N(\lambda)^\vee$, and if $\Psi$ has the same weight and generates an $L(\mu)$, then $\Phi+\Psi$ also generates an $N(\lambda)^\vee$. Hence, the vectors of weight $(k,3)$ generating the $N(\lambda)^\vee$ are only determined up to ``holomorphic'' vectors of the same weight.

In classical language, this means that we do not know of a canonical way to define a complement of $M_{3,k-3}(\Gamma)$ inside $M^*_{3,k-3}(\Gamma)$. We prefer not to choose any such complement, but work with the full space $M^*_{3,k-3}(\Gamma)$ instead. The modular forms in this space generate the component $b_kL(\mu)\,\oplus\,c_kN(\lambda)^\vee$, which is well-defined as a subspace of $V_k$.

Consider the map $P_{0-}$ from $M^*_{3,k-3}(\Gamma)$ to $M_{1,k-1}(\Gamma)$. Recall from \cite{Weissauer1983} that modular forms in the space $M_{1,k-1}(\Gamma)$ are square-integrable. Hence, we may consider the orthogonal complement $M^{**}_{1,k-1}(\Gamma)$ of $P_{0-}(M^*_{3,k-3}(\Gamma))$ inside $M_{1,k-1}(\Gamma)$. The various spaces are then connected by the exact sequence
\begin{equation}\label{MstarNstarexactseqeq}
 0\longrightarrow M_{3,k-3}(\Gamma)\longrightarrow M^*_{3,k-3}(\Gamma)\stackrel{P_{0-}}{\longrightarrow}M_{1,k-1}(\Gamma)\longrightarrow M^{**}_{1,k-1}(\Gamma)\longrightarrow0,
\end{equation}
in which the fourth map is orthogonal projection. The quantity $a_k$ in \eqref{Vkdecomplemmaeq2} is equal to $\dim M^{**}_{1,k-1}(\Gamma)$. Let $V_k^*$ be the subspace of $V_k$ generated by the elements of $M^*_{3,k-3}(\Gamma)$, and let $V_k^{**}$ be the subspace of $V_k$ generated by the elements of $M^{**}_{1,k-1}(\Gamma)$. Then
\begin{equation}\label{Vkdecomplemmaeq1b}
 V_k=V_k^*\:\oplus V_k^{**}.
\end{equation}
The subspaces $V_k^*$ and $V_k^{**}$ are canonically defined, and decompose according to
\begin{equation}\label{Vkdecomplemmaeq1c}
 V_k^*\cong b_kL(\mu)\,\oplus\,c_kN(\lambda)^\vee,\qquad V_k^{**}\cong a_kL(\lambda)
\end{equation}
as abstract modules.
\subsection{The structure theorem for all modular forms}
Recall that in Proposition \ref{AA0nfindecomp2prop} we obtained a decomposition of the space $\AA(\Gamma)^\circ_{\mathfrak{n}\text{-fin}}$ into irreducible $(\mathfrak{g},K)$-modules. The analogous statement for all $\mathfrak{n}$-finite modular forms is slightly more complicated.

\begin{proposition}\label{AAnfindecomp2prop}
 As $(\mathfrak{g},K)$-modules, we have
 \begin{align}\label{AAnfindecomp2propeq0}
  \AA(\Gamma)_{\mathfrak{n}\text{-fin}}&=\bigoplus_{\substack{\ell=2\\\ell\neq3}}^\infty\;\bigoplus_{m=0}^\infty n_{\ell,m}L(\ell+m,\ell)\nonumber\\
  &\quad\oplus\;\bigoplus_{k=3}^\infty V_k^*\;\oplus\:\bigoplus_{k=1}^\infty V_k^{**}\;\oplus\:L(0,0).
 \end{align}
 where $n_{\ell,m}=\dim M_{\ell,m}(\Gamma)$, the spaces $V_k^*,V_k^{**}$ for $k\geq3$ are as in \eqref{Vkdecomplemmaeq1c}, and $V_k^{**}=n_{1,k-1}L(k,1)$ for $k=1,2$.
\end{proposition}
\begin{proof}
This follows by combining Lemma \ref{noncusplemma3} with the results of the previous subsection.
\end{proof}

\begin{proposition}\label{holomorphicsimplemodule} Let $\ell$ be a positive integer, and $m$ a non-negative integer. Let $F \in M_{\ell, m}(\Gamma)$ and let $\Phi_F:\:\Sp_4(\R)\to\C$ be the function of weight $(\ell+m,\ell)$ corresponding to $F$ by Lemma~\ref{FPhilemma}. Then the submodule $\U(\g_\C)\Phi_F$ of $\AA(\Gamma)_{\mathfrak{n}\text{-fin}}$ is irreducible and isomorphic to $L(\ell+m, \ell).$
\end{proposition}
\begin{proof}By Property (3) of the modules $N(\lambda)$ in Section~\ref{setupsec}, we see that $\U(\g_\C)\Phi_F$ is isomorphic to a quotient of $N(\ell+m, \ell)$. If $N(\ell+m, \ell) = L(\ell+m, \ell)$ there is nothing to prove. Otherwise assume that $N(\ell+m, \ell) \neq L(\ell+m, \ell)$. It suffices to prove that $\U(\g_\C)\Phi_F$ is not isomorphic to  $N(\ell+m, \ell)$. But this follows from Proposition \ref{AAnfindecomp2prop}, as the module $N(\ell+m, \ell)$, when reducible, does not appear as a submodule of  $\AA(\Gamma)_{\mathfrak{n}\text{-fin}}$.
\end{proof}

Recall that the cuspidal structure theorem, Theorem \ref{cuspidalstructuretheorem}, was based on Proposition \ref{AA0nfindecomp2prop}, which is the cuspidal analogue of Proposition \ref{AAnfindecomp2prop}, \emph{and} Propositions \ref{navigateKtypesprop} and \ref{navigateKtypesell1prop}, which say that every highest weight vector in an $L(k,\ell)$ can be generated from the highest weight vector of its minimal $K$-type by applying $U$, $X_+$, $D_+$ and $E_+$ operators. We therefore require a result similar to Propositions \ref{navigateKtypesprop} and \ref{navigateKtypesell1prop} for the indecomposable modules $N(k,1)^\vee$ appearing in \eqref{AAnfindecomp2propeq0}. For these modules we define $N(k,1)^\vee_{\text{par}(0)}$ and $N(k,1)^\vee_{\text{par}(1)}$ just as we did in the paragraph before Proposition \ref{navigateKtypesprop} (set $\lambda=(k,1)$). Recall that $N(k,1)^\vee$ sits in an exact sequence
$$
 0\longrightarrow L(k,1)\longrightarrow N(k,1)^\vee\stackrel{\varphi}{\longrightarrow}L(k,3)\longrightarrow0.
$$
For the submodule $L(k,1)$ we have the spaces $L(k,1)_{\text{par}(0)}$ and $L(k,1)_{\text{par}(1)}$ of even and odd highest weight vectors, and clearly
$$
 L(k,1)_{\text{par}(0)}\subset N(k,1)^\vee_{\text{par}(0)}\qquad\text{and}\qquad L(k,1)_{\text{par}(1)}\subset N(k,1)^\vee_{\text{par}(1)}.
$$
The spaces $L(k,1)_{\text{par}(0)}$ and $L(k,1)_{\text{par}(1)}$ originate from $w_0$, the essentially unique vector of weight $(k,1)$, by applying $X_+$, $D_+$ and $E_+$ operators. Let $w_1$ be the essentially unique vector of weight $(k,3)$, so that $\varphi(w_1)$ is the highest weight vector in the minimal $K$-type of $L(k,3)$. Then, by Proposition \ref{navigateKtypesprop},
$$
 L(k,3)_{\text{par}(0)}=\bigoplus_{\substack{\alpha,\beta\geq0\\0\leq\gamma\leq\frac{k-3}2}}\C X_+^\alpha D_+^\beta U^\gamma\varphi(w_1)
, \quad
 L(k,3)_{\text{par}(1)}=\bigoplus_{\substack{\alpha,\beta\geq0\\0\leq\gamma<\frac{k-3}2}}\C E_+X_+^\alpha D_+^\beta U^\gamma\varphi(w_1).
$$
Now let
\begin{equation}\label{Nk1veeeq1}
 \tilde L(k,3)_{\text{par}(0)}=\bigoplus_{\substack{\alpha,\beta\geq0\\0\leq\gamma\leq\frac{k-3}2}}\C X_+^\alpha D_+^\beta U^\gamma w_1
\end{equation}
and
\begin{equation}\label{Nk1veeeq2}
 \tilde L(k,3)_{\text{par}(1)}=\bigoplus_{\substack{\alpha,\beta\geq0\\0\leq\gamma<\frac{k-3}2}}\C E_+X_+^\alpha D_+^\beta U^\gamma w_1.
\end{equation}
It is clear that $\varphi$ maps $\tilde L(k,3)_{\text{par}(i)}$ isomorphically onto $L(k,3)_{\text{par}(i)}$; in particular, the sums in \eqref{Nk1veeeq1} and \eqref{Nk1veeeq2} are really direct.
\begin{lemma}\label{Nk1veenavigatelemma}
 With the above notations, we have $N(k,1)^\vee_{\text{\rm par}(i)}=L(k,1)_{\text{\rm par}(i)}\oplus\tilde L(k,3)_{\text{\rm par}(i)}$ for $i=0,1$.
\end{lemma}
\begin{proof}
It is clear that the sum is direct, since $L(k,1)_{\text{\rm par}(i)}$ lies in the kernel of $\varphi$, while the restriction of $\varphi$ to $\tilde L(k,3)_{\text{\rm par}(i)}$ is an isomorphism. Let $v\in N(k,1)^\vee_{\text{\rm par}(i)}$. Then $\varphi(v)\in L(k,3)_{\text{\rm par}(i)}$. Let $\tilde v\in \tilde L(k,3)_{\text{par}(i)}$ be such that $\varphi(\tilde v)=\varphi(v)$. Then $v-\tilde v\in L(k,1)_{\text{\rm par}(i)}$. The assertion follows.
\end{proof}

\begin{theorem}[Structure theorem for all modular forms]\label{allmodformstructuretheorem}
 Let $\ell$ be a positive integer, and $m$ a non-negative integer.  Let the sets $\mathcal{X}_{\ell',m'}^{\ell,m}$ be defined as in \eqref{Xlmlmdefeq1} and \eqref{Xlmlmdefeq2}. Then we have a direct sum decomposition

 \begin{equation}\label{allmodformstructuretheoremeq1}
  N_{\ell,m}(\Gamma)=\bigoplus_{\ell'=1}^\ell
  \bigoplus_{m'=0}^{\ell+m-\ell'}\;\sum_{X \in \mathcal{X}_{\ell', m'}^{\ell, m}} X(M^*_{\ell',m'}(\Gamma)),
 \end{equation}
 where
 \begin{equation}\label{allmodformstructuretheoremeq2}
  M^*_{\ell',m'}(\Gamma)=\begin{cases}
                          M_{\ell',m'}(\Gamma)&\text{if }\ell'\neq3,\\
                          \text{as in \eqref{N1stardefeq}}&\text{if }\ell'=3.
                         \end{cases}
 \end{equation}
 The decomposition \eqref{allmodformstructuretheoremeq1} is orthogonal in the following sense: If
 \begin{equation}\label{allmodformstructuretheoremeq3}
  F_1\in\sum_{X \in \mathcal{X}_{\ell', m'}^{\ell, m}} X(S_{\ell',m'}(\Gamma)),\qquad
  F_2\in\sum_{X \in \mathcal{X}_{\ell'', m''}^{\ell, m}} X(M^*_{\ell'',m''}(\Gamma)),
 \end{equation}
 and if $(\ell',m')\neq(\ell'',m'')$, then $\langle F_1,F_2\rangle=0$.
\end{theorem}
\begin{proof}
The proof of \eqref{allmodformstructuretheoremeq1} is similar to that of Theorem \ref{cuspidalstructuretheorem}. Instead of Proposition \ref{AA0nfindecomp2prop} one uses Proposition \ref{AAnfindecomp2prop}. In addition to Propositions \ref{navigateKtypesprop} and \ref{navigateKtypesell1prop}, one also uses Lemma \ref{Nk1veenavigatelemma}. We omit the details.

To prove the orthogonality statement, write $F_2=\sum X_iF_i'+\sum X_j F_j''$, where $F'_i\in S_{\ell'',m''}(\Gamma)$, and the $F''_j\in M^*_{\ell'',m''}(\Gamma)$ are orthogonal to $S_{\ell'',m''}(\Gamma)$. Clearly, if $F'=\sum X_iF_i'$, then $\langle F_1,F_2\rangle=\langle F_1,F'\rangle$. We are thus reduced to cusp forms, for which the statement follows from the orthogonality of the decomposition in Theorem \ref{cuspidalstructuretheorem}.
\end{proof}
\begin{remark}\label{notcontainedremark}
 Not contained in Theorem \ref{allmodformstructuretheorem} is the case $\ell=0$. But recall from Lemma \ref{noncusplemma2} (or Proposition \ref{AAnfindecomp2prop}) that $N_{0,0}(\Gamma)=\C$, while $N_{0,m}(\Gamma)=0$ for $m>0$.
\end{remark}

\subsubsection*{Modular forms orthogonal to cusp forms}
We will introduce some notation involving orthogonal complements of cusp forms. First, let $E_{\ell,m}(\Gamma)$ be the orthogonal complement of $S_{\ell,m}(\Gamma)$ inside $M_{\ell,m}(\Gamma)$, so that
\begin{equation}\label{Elmdefeq}
 M_{\ell,m}(\Gamma)=S_{\ell,m}(\Gamma)\oplus E_{\ell,m}(\Gamma).
\end{equation}
Recall from \eqref{N1stardef2eq} that $M^*_{3,m}(\Gamma)=\{F\in N^1_{3,m}(\Gamma)\:|\:LF=E_-F=0\}$. We let $E^*_{3,m}(\Gamma)$ be the orthogonal complement of  $S_{3,m}(\Gamma)$ in $M^*_{3,m}(\Gamma)$, so that
\begin{equation}\label{Estardefeq}
 M^*_{3,m}(\Gamma)=S_{3,m}(\Gamma)\oplus E^*_{3,m}(\Gamma).
\end{equation}
Recall that in Sect.\ \ref{peterssonsec} we defined $\E_{\ell,m}(\Gamma)$ to be the orthogonal complement of  $N_{\ell,m}(\Gamma)^\circ$ in $N_{\ell,m}(\Gamma)$, so that
\begin{equation}\label{slashEdefeq}
 N_{\ell,m}(\Gamma)=N_{\ell,m}(\Gamma)^\circ\oplus\E_{\ell,m}(\Gamma).
\end{equation}
\begin{lemma}\label{EslashElemma}
 Let $\ell$ be a positive integer, and $m$ a non-negative integer. Then:
 \begin{enumerate}
  \item $E^*_{3,m}(\Gamma)\subset\E_{3,m}(\Gamma)$.
  \item $\E_{3,m}(\Gamma)\cap M^*_{3,m}(\Gamma)=E^*_{3,m}(\Gamma)$.
  \item $E_{\ell,m}(\Gamma)\subset\E_{\ell,m}(\Gamma)$.
  \item $\E_{\ell,m}(\Gamma)\cap M_{\ell,m}(\Gamma)=E_{\ell,m}(\Gamma)$.
 \end{enumerate}
\end{lemma}
\begin{proof}
(1) Let $F\in E^*_{3,m}(\Gamma)$ and $G\in N_{3,m}(\Gamma)^\circ$; we have to show that $\langle F,G\rangle=0$. We work instead with the corresponding automorphic forms $\Phi_F$, $\Phi_G$, and will show that $\langle\Phi_F,\Phi_G\rangle=0$. We may assume that $\Phi_G$ generates an irreducible module $L(\kappa)$. Recall from the definition of the space $M^*_{3,m}(\Gamma)$ that $\Phi_F$ generates either a module $L(\mu)$, where $\mu=(m+3,3)$, or a module $N(\lambda)^\vee$, where $\lambda=(m+3,1)$. Assume that $\langle\Phi_F,\Phi_G\rangle\neq0$; we will obtain a contradiction. Since the modules $\langle\Phi_F\rangle$ and $\langle\Phi_G\rangle\cong L(\kappa)$ pair non-trivially, we get a non-zero $\g_\C$-map
$$
 L(\mu)\longrightarrow L(\kappa)\qquad\text{or}\qquad N(\lambda)^\vee\longrightarrow L(\kappa).
$$
In either case we conclude $L(\kappa)\cong L(\mu)$, hence $\kappa=\mu$. It follows that $G$ is holomorphic, therefore an element of $S_{3,m}(\Gamma)$. Since $F\in E^*_{3,m}(\Gamma)$, we have $\langle F,G\rangle=0$, contradicting our assumption $\langle\Phi_F,\Phi_G\rangle\neq0$.

(2) follows from (1), (3) is proved in a way analogous to (1), and (4) follows from (3).
\end{proof}

\begin{theorem}[Structure theorem for modular forms orthogonal to cusp forms]\label{noncuspidalstructuretheorem}
 Let $\ell$ be a positive integer, and $m$ a non-negative integer.  Let the sets $\mathcal{X}_{\ell',m'}^{\ell,m}$ be defined as in \eqref{Xlmlmdefeq1} and \eqref{Xlmlmdefeq2}. Then we have a direct sum decomposition
 \begin{equation}\label{noncuspidalstructuretheoremeq1}
  \E_{\ell,m}(\Gamma)=\bigoplus_{\ell'=1}^\ell
  \bigoplus_{m'=0}^{\ell+m-\ell'}\;\sum_{X \in \mathcal{X}_{\ell', m'}^{\ell, m}} X(E^*_{\ell',m'}(\Gamma)),
 \end{equation}
 where
 \begin{equation}\label{noncuspidalstructuretheoremeq2}
  E^*_{\ell',m'}(\Gamma)=\begin{cases}
                          E_{\ell',m'}(\Gamma)&\text{if }\ell'\neq3,\\
                          \text{as in \eqref{Estardefeq}}&\text{if }\ell'=3.
                         \end{cases}
 \end{equation}
\end{theorem}
\begin{proof}
By Lemma \ref{EslashElemma}, $E^*_{\ell',m'}(\Gamma)\subset\E_{\ell',m'}(\Gamma)$ for all $\ell',m'$. Lemma \ref{preserveeissiegel} therefore implies that the right hand side is contained in the left hand side. The reverse inclusion follows in a straightforward way from Theorem \ref{allmodformstructuretheorem}.
\end{proof}

\section{Adelization and arithmeticity}\label{s:final}
\subsection{Adelization and automorphic representations}\label{s:adele}
Throughout this section, we let $G$ denote the group $\GSp_4$. Let $K_\infty$ denote the maximal compact subgroup of $\Sp_4(\R)$, and for each prime $p$, put $K_p = G(\Z_p)$. Recall that an automorphic form on $G(\A)$ is a smooth function on $G(\A)$ that is left $G(\Q)$-invariant, $\mathcal{Z}$-finite, $K_\infty$-finite and slowly increasing; here $\mathcal{Z}$ is the center of $\mathcal{U}(\mathfrak{g}'_\C)$, where $\mathfrak{g}'\cong\R\oplus\mathfrak{g}$ is the Lie algebra of $\GSp_4(\R)$. We let $\AA(G)$ denote the space of automorphic forms on $G(\A)$ and $\AA(G)^\circ$ denote the subspace of cusp forms on $G(\A)$. All the automorphic forms we will consider in the following will be annihilated by the center $\R$ of $\mathfrak{g}'$.

For each prime $p$, and each positive integer $N$, define a compact open subgroup $K_p^N$ of $G(\Z_p)$ by \begin{equation}\label{kpn}K_p^N =  \left\{g\in G(\Z_p)\;|\;g \equiv \mat{I_2}{}{}{aI_2} \pmod{N}, \ a\in \Z_p^\times\right\}.\end{equation} Note that our choice of $K_p^N$ satisfies the following properties:
\begin{itemize}
 \item $K_p^N = G(\Z_p)$ for all primes $p$ not dividing $N$,
 \item The multiplier map $\mu_2 : K_p^N \mapsto \Z_p^\times$ is surjective for all primes $p$,
 \item $\Gamma(N)=G(\Q) \bigcap G(\R)^+\prod_{p<\infty}K_p^N$.
\end{itemize}

As always, let $\ell$, $m$ denote integers with $m \ge 0$. Let $\Gamma$ be a congruence subgroup of $\Sp_4(\Q)$ and $F$ be an element of $C^\infty_{\ell,m}(\Gamma)$. Let $N$ be any integer such that $\Gamma(N) \subset \Gamma$.
By Lemma~\ref{FPhilemma}, we can attach to $F$ a function $\Phi$ on $\Sp_4(\R)$ that is left invariant by $\Gamma$.
By strong approximation, we can write any element $g \in G(\A) $ as
$$g = \lambda g_\Q g_\infty k_\mathfrak{f}, \qquad g_\Q\in G(\Q), \ g_\infty \in \Sp_4(\R),\ k_\mathfrak{f} \in \prod_{p}K_p^N, \ \lambda \in Z_{G}(\R)^+,$$
We define the \emph{adelization} $\Phi_F$ of $F$ to be the function on $G(\A)$ defined by $\Phi_F(g) = \Phi(g_\infty)$. This is well defined because of the way the groups $K_p^N$ were chosen. Furthermore, it is independent of the choice of $N$. By construction, it is clear that $\Phi_F(h g) = \Phi_F(g) \ \text{ for all } h \in G(\Q), g \in G(\A)$. It is also easy to see that the mapping $F \mapsto \Phi_F$ is linear. The following is immediate from Proposition~\ref{Nplmautformprop}.

\begin{proposition}\label{Nlmautformprop}
 Let $\Gamma$ be a congruence subgroup of $\Sp_4(\Q)$ and $F$ be an element of $N_{\ell, m}(\Gamma)$. Let $\Phi_F$ be the adelization of $F$. Then $\Phi_F\in\AA(G)$. If $F \in N_{\ell, m}(\Gamma)^\circ$, then $\Phi_F\in\AA(G)^\circ$.
\end{proposition}

Let $F \in N_{\ell, m}(\Gamma) $ and $\Phi_F \in \AA(G)$ be its adelization as defined above. Then $\Phi_F$ generates a representation $\pi_F$ under the natural right regular action\footnote{More precisely, one takes the right regular action of $G(\A_{\mathfrak{f}})$ together with the action of the Lie algebra at the infinite place.} of $G(\A)$. From the results of the previous sections it follows that any irreducible subquotient of $\pi_F$ is an irreducible \emph{automorphic} representation of $G(\A)$; it is cuspidal whenever $F \in N_{\ell, m}(\Gamma)^\circ$.

Let $X$, $\mathcal{X}$ be as in Lemma~\ref{slashliecommute}. Let $F \in N_{\ell, m}(\Gamma)$ be such that $\Phi_F$ generates a factorizable representation $\pi = \otimes_v \pi_v$ of $G(\A)$, and suppose that $\Phi_F$ corresponds to a factorizable vector $\phi = \otimes_v \phi_v$ inside $\pi$. Then, if $G:=XF \in N_{\ell',m'}(\Gamma)$, then  $\Phi_G$ is the vector inside $\pi$ corresponding to $\otimes_{p<\infty}\phi_v \otimes (X\phi_\infty)$. In particular, if $\pi$ is an irreducible automorphic representation, then $\Phi_G$ generates $\pi$. This is immediate from~\eqref{Nplusoperatorsnearholopropeq2}, the definition of the adelization map, and the fact that $X$ does not alter the components of $F$ at any of the finite places.


\begin{proposition}\label{piinfprop}
 Let $F \in M^*_{\ell, m}(\Gamma)$ and $ \pi_F$ be the $(\mathfrak{g},K_\infty)\times G(\A_{\mathfrak{f}})$-module generated by $\Phi_F$.  Let $\pi = \otimes_v \pi_v$ be any irreducible subquotient of $\pi_F$.
 \begin{enumerate}
  \item If $\ell \neq 3$, then $\pi_\infty \simeq L(\ell +m, \ell)$.
  \item If $\ell=3$, then $\pi_\infty$ is isomorphic to either $L(3 +m, 3)$ or $L(3 +m, 1)$.
\end{enumerate}
\end{proposition}
\begin{proof}
This can be derived in a straightforward way from our results in Sect.~\ref{nontemperedsec}.
\end{proof}

\subsection{Arithmeticity for nearly holomorphic forms}\label{s:arithmeticity}

Recall that any $F \in N_{\ell,m}(\Gamma) = \bigcup_{p\ge0}N_{\ell,m}^p(\Gamma)$ has a Fourier expansion as follows (note the difference in normalization between~\eqref{Fourierexpansioneq4new} and~\eqref{Fourierexpansioneq4}; this is for arithmetic purposes): \begin{equation}\label{degree2fourierexpansionnearholo} F(Z)=\sum_{Q \in M_2^{\sym}(\Q)}a(Q)e^{2\pi i\,{\rm Tr}(QZ)},\end{equation} where \begin{equation}\label{Fourierexpansioneq4new}
 a(Q)=\sum_{\alpha,\beta,\gamma}
 a_{\alpha,\beta,\gamma}(Q) \Big(\frac{y}{2 \pi \Delta}\Big)^\alpha\Big(\frac{v}{2\pi \Delta}\Big)^\beta\Big(\frac{y'}{2 \pi \Delta}\Big)^\gamma, \quad a_{\alpha,\beta,\gamma}(Q) \in W_m.
\end{equation}We note that  $a_{\alpha,\beta,\gamma}(Q) = 0$  unless $Q \in \frac{1}{N} M_2^{\sym}(\Z)$ for some integer $N$. Given any $\sigma \in \Aut(\C)$, we define a function $\leftexp{\sigma}F$ via the action of $\sigma$ on the elements $a_{\alpha,\beta,\gamma}(Q)$:

$$\leftexp{\sigma}F(Z)=\sum_{Q \in M_2^{\sym}(\Q)}\leftexp{\sigma}a(Q)e^{2\pi i\,{\rm Tr}(QZ)},$$ where
\begin{equation}\label{Fourierexpansioneq4newaut}
\leftexp{\sigma}a(Q):=\sum_{j=0}^m\,\sum_{\alpha,\beta,\gamma}
 \sigma(a_{j,\alpha,\beta,\gamma}(Q)) \Big(\frac{y}{2 \pi \Delta}\Big)^\alpha\Big(\frac{v}{2\pi \Delta}\Big)^\beta\Big(\frac{y'}{2 \pi \Delta}\Big)^\gamma S^{m-j}T^j.
\end{equation}

For any subfield $L$ of $\C$, define $N_{\ell,m}(\Gamma; L)$ to be the subspace of $N_{\ell,m}(\Gamma)$ consisting of the forms that are fixed by $\Aut(\C/L)$. Define $N_{\ell,m}(\Gamma; L)^\circ$,  $N^p_{\ell,m}(\Gamma; L)$, $N^p_{\ell,m}(\Gamma; L)^\circ$, $M_{\ell,m}(\Gamma; L)$,  $S_{\ell,m}(\Gamma; L)$ similarly.

We say that a congruence subgroup $\Gamma$ of $\Sp_4(\Q)$ is ``nice" if there exists a compact open subgroup $K_0$ of $G(\A_\mathfrak{f})$ with the following properties.

\begin{enumerate}
 \item $K_0 = \prod_{p<\infty}K_{0,p}$, where, for each prime $p$, $K_{0,p}$ is a compact open subgroup of $G(\Q_p)$ with $K_{0,p} =G(\Z_p)$ for almost all primes.
 \item For all $p$, and all $x \in  \Z_p^\times$, we have
  $$
   \diag(1,1,x,x)K_{0,p}\ \diag(1,1,x^{-1}, x^{-1}) = K_{0,p}.
  $$
 \item $K_0 \GSp_4(\R)^+ \:\cap\: \GSp_4(\Q) = \Gamma$.
\end{enumerate}
We note that all congruence subgroups that are naturally encountered in the theory, such as the principal, Siegel, Klingen, Borel or paramodular congruence subgroups, are nice in the above sense. The following result follows from \cite[Theorem 14.13]{shimura2000}.

\begin{theorem}[Shimura]\label{autcshimura}
 Let $\Gamma$ be a nice congruence subgroup of $\Sp_4(\Q)$. Then for all $p \ge 0$ we have the equalities
 $$
  N_{\ell,m}^p(\Gamma) = N_{\ell,m}^p(\Gamma; \Q) \otimes_{\Q} \C,
 $$
 $$
  N_{\ell,m}^p(\Gamma)^\circ = N_{\ell,m}^p(\Gamma; \Q)^\circ \otimes_{\Q} \C.
 $$
 In particular, the action of $\Aut(\C)$ preserves the above spaces.
\end{theorem}

\begin{remark}
Theorem 14.13 of \cite{shimura2000} had the added condition that $M_{k,0}(\Gamma; \overline{\Q}) \neq \{0\}$ for some $0<k \in\Z$. This is clearly true in our case. Indeed, we have $\Gamma \subset \gamma^{-1} \Gamma^{\rm para}(N)\gamma$ for some squarefree integer $N$ and some $\gamma \in G(\Q)$; this is because every compact open subgroup of $G(\Q_p)$ is either contained in a conjugate of $G(\Z_p)$ or in a conjugate of the local paramodular group at $p$. Let $F_1 \in S_{10,0}(\Sp_4(\Z); \Q)$ be the unique weight 10 cusp form of full level. Then $F =(\prod_{p|N}\theta_p)F_1$ belongs to $S_{10,0}(\Gamma^{\rm para}(N); \overline{\Q})$, where $\theta_p$ is as in~\cite{robertsschmidt06}; the fact that the Fourier coefficients are algebraic follow from the $q$-expansion principle. Observe that $\theta_p$ is injective by Theorem 7.2 of \cite{robertsschmidt06}, so that $F\neq0$. Finally, $F|_{k,0}\gamma$ is an element of $S_{10,0}(\Gamma; \overline{\Q})$.

\end{remark}

Let $\mathcal{X_+}$ be the free monoid consisting of all (finite) strings of the symbols $X_+$, $U$, $E_+$, and $D_+$ in the left column of Table~\ref{Nplusoperatorstable}. Clearly  $\mathcal{X_+}$ is a submonoid of the monoid $\mathcal{X}$ defined earlier, and furthermore contains all the subsets $\mathcal{X}_{\ell',m'}^{\ell,m}$ introduced for the purpose of stating the structure theorems. Each element $X \in \mathcal{X_+}$ is an operator that for any particular $\ell, m, p$, takes $N^p_{\ell, m}(\Gamma)$ to $N^{p_1}_{\ell_1, m_1}(\Gamma)$, where the integers $\ell_1, m_1, p_1$ (that depend on $\ell$, $m$, $p$ and $X$) can be easily calculated using Table~\ref{Nplusoperatorsnearholotable}. In particular, the non-negative integer $v= p_1-p$ depends only on $X$. Precisely, $v=1$ for $X_+$, $U$, and $E_+$; $v=2$ for $D_+$. For longer strings, $v$ can be calculated by adding up the contributions from the individual operators.

\begin{definition}For any $X \in \mathcal{X_+}$, we define the degree of $X$ to be the integer $v$ described above.
\end{definition}

The following key proposition, when combined with our structure theorems, allows us to transfer arithmeticity results from holomorphic forms to nearly holomorphic forms.

\begin{proposition}\label{proparithmeticitydiff}Let $X \in \mathcal{X_+}$ and let $v$ be the degree of $X$. Then, for all $F \in N_{\ell, m}(\Gamma)$, and all $\sigma \in \Aut(\C)$, we have $$\leftexp{\sigma}( (2 \pi)^{-v}XF ) =  (2 \pi)^{-v}X(\leftexp{\sigma}F).$$

\end{proposition}
\begin{proof}It suffices to prove this for each of the basic operators $X_+$, $U$, $E_+$, and $D_+$. Using equations \eqref{operatorsonfunctions1} -- \eqref{operatorsonfunctions8}, we note that the  action of the operators $X_+$, $U$, and $E_+$ on the component functions of $F$ are given by rational linear combinations from the following set $S$ of operators on $C^\infty(\H_2)$,
$$
 S = \Big\{\frac{y}{\Delta}, \frac{y'}{\Delta}, \frac{v}{\Delta}, 2i\frac{\partial}{\partial z}, 2i\frac{\partial}{\partial \tau}, 2i\frac{\partial}{\partial \tau'} \Big\}.
$$
Furthermore, the  action of the operator $D_+$ on the component functions of $F$ is given by rational linear combinations of the objects formed by taking the composition of exactly \emph{two operators} from the set $S$.

Therefore, to complete the proof, it suffices to show that for each element $Q \in M_2^{\sym}(\Q)$, each triple of non-negative integers $\alpha, \beta, \gamma$, and each operator $s\in S$, there exist \emph{rational numbers} $a_{\alpha', \beta', \gamma'}(Q)$  indexed by a finite set of triples of non-negative integers $\alpha', \beta', \gamma'$, such that
\begin{align*}
 &(2 \pi)^{-1}s\left(\Big(\frac{y}{2 \pi \Delta}\Big)^\alpha\Big(\frac{v}{2\pi \Delta}\Big)^\beta\Big(\frac{y'}{2 \pi \Delta}\Big)^\gamma  e^{2\pi i\,{\rm Tr}(QZ)} \right) \\ &= \sum_{\alpha',\beta',\gamma'}
 a_{\alpha',\beta',\gamma'}(Q) \Big(\frac{y}{2 \pi \Delta}\Big)^{\alpha'}\Big(\frac{v}{2\pi \Delta}\Big)^{\beta'}\Big(\frac{y'}{2 \pi \Delta}\Big)^{\gamma'} e^{2\pi i\,{\rm Tr}(QZ)}.
\end{align*}
This is an elementary calculation and can be easily verified for each element $s$ of $S$. We omit the details.
\end{proof}

\subsubsection*{Isotypic projections}

By our structure theorem, the space $N_{\ell, m}(\Gamma)$ decomposes as a direct sum as follows:
\begin{equation}\label{structuredecomparith}
 N_{\ell, m}(\Gamma) = \bigoplus_{\substack{0\le \ell' \le \ell \\ 0\le \ell'+m' \le \ell+m \\ m'\ge 0}}\sum_{X \in \mathcal{X}^{\ell, m}_{\ell', m'}}X(M^{*}_{\ell',m'}(\Gamma)),
\end{equation}
where we adopt the convention that $M^{*}_{\ell',m'}(\Gamma) := M_{\ell', m'}(\Gamma)$ whenever $\ell' \neq 3$. The identical decomposition holds for the cuspidal subspace.

\begin{definition}
For each quadruple of integers $\ell, m, \ell', m'$ with $m,m'$ non-negative, define an endomorphism $\mathfrak{p}_{\ell',m'}$ of $N_{\ell,m}(\Gamma)$ by the projection map
$$
 N_{\ell, m}(\Gamma) \longrightarrow \bigg(\sum_{X \in \mathcal{X}^{\ell, m}_{\ell', m'}}X(M^*_{\ell',m'}(\Gamma))\bigg) \subset N_{\ell, m}(\Gamma)
$$
given by the direct sum decomposition~\eqref{structuredecomparith}. In particular, if the set
$\mathcal{X}^{\ell, m}_{\ell', m'}$ is empty, we have $\mathfrak{p}_{\ell', m'} = 0$.
\end{definition}

\begin{lemma}\label{nearlyholprojbasicproplemma}
 Suppose that $F \in N_{\ell,m}(\Gamma)$. Then the following hold.
\begin{enumerate}
 \item Suppose that $F \in N_{\ell, m}(\Gamma)^\circ$, resp.\ $F \in \E_{\ell, m}(\Gamma)$. Then, $\mathfrak{p}_{\ell', m'}(F) \in  N_{\ell, m}(\Gamma)^\circ$,   resp.\ $\mathfrak{p}_{\ell', m'}(F) \in \E_{\ell, m}(\Gamma)$.
 \item We have
  $$
   F = \sum_{\ell' \ge 0,\:m'\ge0}\mathfrak{p}_{\ell', m'}(F).
  $$
  The above sum is orthogonal in the sense that if $(\ell_1', m_1') \neq (\ell_2', m_2'),$ and  $\mathfrak{p}_{\ell_1', m_1'}(F) \in N_{\ell, m}(\Gamma)^\circ$, then
  $$
   \Big\langle \mathfrak{p}_{\ell_1', m_1'}(F), \ \mathfrak{p}_{\ell_2', m_2'}(F)\Big\rangle = 0.
  $$
 \item Suppose that $F \in N_{\ell, m}(\Gamma)$, and $G \in S_{\ell', m'}(\Gamma)$. Then, for all $X \in \mathcal{X}_{\ell',m'}^{\ell, m}$,
   $$
    \langle F, XG\rangle = \Big\langle \mathfrak{p}_{\ell', m'}(F), XG\Big\rangle.
   $$
\end{enumerate}
\end{lemma}
\begin{proof}
All the parts follow directly from the structure theorems and our definition of the projection map. We omit the details.
\end{proof}

\begin{lemma}\label{autcgeneralholo}Let $\Gamma$ be a nice congruence subgroup of $\Sp_4(\Q)$. Then we have the equality
$$M^{*}_{\ell',m'}(\Gamma) = M^{*}_{\ell',m'}(\Gamma;\Q) \otimes_{\Q} \C.$$ In particular, the action of $\Aut(\C)$ preserves the above space.
\end{lemma}
\begin{proof}We only need to consider the case $\ell' =3$, since otherwise $M^{*}_{\ell',m'}(\Gamma) = M_{\ell',m'}(\Gamma)$ and this case has already been covered by Theorem~\ref{autcshimura}. So, assume $\ell'=3$. Let $F \in M^{*}_{3,m'}(\Gamma)$ and $\sigma \in \Aut(\C)$. It suffices to show that $\leftexp{\sigma}F \in  M^{*}_{3,m'}(\Gamma)$; see Lemma 3.19 of \cite{sahapet}. We already know from Theorem~\ref{autcshimura} that  $\leftexp{\sigma}F \in  N^1_{3,m'}(\Gamma).$ So, to complete the proof, we only need to show that $L(\leftexp{\sigma}F)=E_{-}(\leftexp{\sigma}F)=0$. But this is an immediate consequence of  Proposition~\ref{proparithmeticitydiff}.
\end{proof}

We now state our main arithmeticity result concerning this projection map.
\begin{proposition}\label{arithmeticityproj}
 For all quadruples $(\ell, m, \ell', m')$, all  $\sigma \in \Aut(\C)$, and all $F \in N_{\ell, m}(\Gamma)$, we have
 $$
  \mathfrak{p}_{\ell', m'}(\leftexp{\sigma}F) = \leftexp{\sigma}(\mathfrak{p}_{\ell', m'}(F)).
 $$
\end{proposition}
\begin{proof}
By shrinking $\Gamma$ if necessary, we may assume $\Gamma$ is nice. Using the structure theorem \ref{allmodformstructuretheorem}, write
$$
 F = \sum_{\ell', m'}\sum_{X \in \mathcal{X}^{\ell, m}_{\ell', m'}}X'(F_{X}), \qquad \text{where }X' = (2\pi)^{-v(X)}X
$$
and $F_X\in M^*_{\ell',m'}(\Gamma)$. Then, by Proposition \ref{proparithmeticitydiff},
$$
 \leftexp{\sigma}F = \sum_{\ell', m'}\sum_{X \in \mathcal{X}^{\ell, m}_{\ell', m'}}\leftexp{\sigma}(X'(F_X))= \sum_{\ell', m'}\sum_{X \in \mathcal{X}^{\ell, m}_{\ell', m'}}X'(\leftexp{\sigma}F_X).
$$
By Theorem \ref{autcshimura} and Lemma~\ref{autcgeneralholo}, the modular form $\leftexp{\sigma}F_X$ lies in $M^*_{\ell',m'}(\Gamma)$. Hence
$$
 \mathfrak{p}_{\ell', m'}(\leftexp{\sigma}F) = \sum_{X \in \mathcal{X}^{\ell, m}_{\ell', m'}}X'(\leftexp{\sigma}F_X)=  \sum_{X \in \mathcal{X}^{\ell, m}_{\ell', m'}}\leftexp{\sigma}X'(F_X)  = \leftexp{\sigma}(\mathfrak{p}_{\ell', m'}(F)).
$$
This completes the proof.
\end{proof}

\begin{remark}
In the special case $\ell'=\ell$, $m'=m$, Shimura defined the map $\mathfrak{p}_{\ell,m}:N_{\ell,m}(\Gamma) \to M_{\ell,m}(\Gamma)$ and called it the holomorphic projection map. He was able to prove results of $\Aut(\C)$-equivariance in this special case under the additional assumption that either $F \in N_{\ell,m}(\Gamma)^\circ$ or $m=0$; see ~\cite[Prop.\ 15.3, Prop.\ 15.6]{shimura2000}.
\end{remark}

\begin{definition}
Let $\mathfrak{q}$ denote the natural projection map from nearly holomorphic modular forms to nearly holomorphic cusp forms, i.e., $\mathfrak{q}:\oplus_{\ell,m}N_{\ell,m}(\Gamma) \rightarrow \oplus_{\ell, m}N_{\ell,m}(\Gamma)^\circ$ is obtained from the orthogonal direct sum decomposition  $$N_{\ell,m}(\Gamma) = N_{\ell,m}(\Gamma)^\circ \oplus \E_{\ell, m}(\Gamma).$$
\end{definition}
\begin{definition}
 Define $\mathfrak{p}^\circ_{\ell', m'} = \mathfrak{q} \circ \mathfrak{p}_{\ell', m'}$.
\end{definition}
Thus,
$$
 \mathfrak{p}^\circ_{\ell',m'} : N_{\ell, m}(\Gamma) \rightarrow \sum_{X \in \mathcal{X}^{\ell, m}_{\ell', m'}}X(S_{\ell',m'}(\Gamma)) \subset N_{\ell, m}(\Gamma)^\circ.
$$
If $F \in N_{\ell, m}(\Gamma)^\circ$, then $\mathfrak{p}^\circ_{\ell',m'}(F) = \mathfrak{p}_{\ell',m'}(F)$. It is clear that for all  $F \in N_{\ell, m}(\Gamma)$, $G \in S_{\ell', m'}(\Gamma)$, $X \in \mathcal{X}_{\ell',m'}^{\ell, m}$, we have
$$
 \langle F, XG\rangle = \big\langle \mathfrak{p}_{\ell', m'}(F), XG\big\rangle = \big\langle \mathfrak{p}^\circ_{\ell', m'}(F), XG\big\rangle.
$$
Furthermore, if $F \in N_{\ell, m}(\Gamma)$ and we write, using the structure theorem,
$$
 F  = \sum_{\ell', m'}\sum_{X \in \mathcal{X}^{\ell, m}_{\ell', m'}}X(F_{X}),
$$
then
$$
 \mathfrak{p}^\circ_{\ell', m'}(F) = \sum_{X \in \mathcal{X}^{\ell, m}_{\ell', m'}}X(\mathfrak{q}(F_{X})).
$$
Recall that $E_{\ell,m}(\Gamma)$ denotes the orthogonal complement of $S_{\ell,m}(\Gamma)$ in $M_{\ell, m}(\Gamma)$ and has the property that $E_{\ell,m}(\Gamma)=\E_{\ell, m}(\Gamma) \cap M_{\ell,m}(\Gamma)$; see Lemma \ref{EslashElemma}.

 \begin{definition} Given a number field $L$, we say that $E_{\ell,m}(\Gamma)$ is $L$-rational  if $$E_{\ell,m}(\Gamma) = E_{\ell,m}(\Gamma; L) \otimes_L \C.$$
 \end{definition}

\begin{remark}\label{arithmeticitycuspremark}
If $E_{\ell,m}(\Gamma)$ is $L$-rational, then for all $F \in M_{\ell,m}(\Gamma)$, $\sigma \in \Aut(\C/L)$, we have $\leftexp{\sigma}(\mathfrak{q}(F)) = \mathfrak{q}(\leftexp{\sigma}F).$
\end{remark}

 \begin{remark}The results of Harris (see~\cite{hareis}) imply that if $\ell>4$ (so that we are in the absolutely convergent range, and so $E_{\ell,m}(\Gamma)$ is spanned by holomorphic Siegel and Klingen Eisenstein series) and $\Gamma$ is nice, then $E_{\ell,m}(\Gamma)$ is $L$-rational for some number field $L$. It is unclear to us if we can always take $L =\Q$ in this case, though we suspect this to be the case.
 \end{remark}

\begin{proposition}\label{arithmeticitycusppro}
 Suppose that $\ell'>3$ and $E_{\ell',m'}(\Gamma)$ is $L$-rational. Then, for all $F \in N_{\ell, m}(\Gamma)$ and $\sigma \in \Aut(\C/L)$,
 $$
  \mathfrak{p}^\circ_{\ell', m'}(\leftexp{\sigma}F) = \leftexp{\sigma}(\mathfrak{p}^\circ_{\ell', m'}(F)).
 $$
\end{proposition}
\begin{proof}
The proof is essentially identical to that of Proposition~\ref{arithmeticityproj}.
\end{proof}

We end this section with an arithmeticity result for ratios of Petersson inner products.
\begin{proposition}\label{pterssonratio}
 Let $F \in S_{\ell, m}(\Gamma)$ have the property that for all $G \in S_{\ell, m}(\Gamma)$ and all $\sigma \in \Aut(\C)$, we have
 $$
  \sigma \left( \frac{\langle G, F\rangle}{\langle F,F \rangle}  \right) =  \frac{\langle \leftexp{\sigma}G, \leftexp{\sigma}F\rangle}{\langle \leftexp{\sigma}F,\leftexp{\sigma}F \rangle}.
 $$
 Let $\ell_1, m_1$ be integers such that $\mathcal{X}_{\ell, m}^{\ell_1, m_1}$ is a singeleton set equal to $\{X \}$. Then for all $G \in N_{\ell_1, m_1}(\Gamma)^\circ$, and all $\sigma \in \Aut(\C)$, we have
 $$
  \sigma \left( \frac{\langle G, \ XF\rangle}{\langle XF, \ XF \rangle}  \right) =  \frac{\langle \leftexp{\sigma}G, \ \leftexp{\sigma}XF\rangle}{\langle \leftexp{\sigma}XF, \ \leftexp{\sigma}XF \rangle}.
 $$
\end{proposition}
\begin{remark}
It is expected that whenever $\ell \ge 6$, all Hecke eigenforms $F$ in $S_{\ell, m}(\Gamma)$ with coefficients in a CM field have the property required in the above proposition. This has been proved in many special cases, e.g., when $\Gamma = \Sp_4(\Z)$ (see ~\cite{takei}).
\end{remark}
\begin{proof}
By (3) of Lemma \ref{nearlyholprojbasicproplemma},
$$
 \frac{\langle G, XF\rangle}{\langle XF,XF \rangle}  =  \frac{\langle \mathfrak{p}_{\ell,m}(G), XF\rangle}{\langle XF,XF \rangle}.
$$
Now, $\mathfrak{p}_{\ell,m}(G) = XG'$ for some $G' \in S_{\ell,m}(\Gamma)$. By Proposition \ref{peterssonequivprop},
$$
 \sigma\left( \frac{\langle G, XF\rangle}{\langle XF,XF \rangle}  \right)= \sigma \left( \frac{\langle X(G'), XF\rangle}{\langle XF,XF \rangle}  \right) = \sigma \left( \frac{\langle G', F\rangle}{\langle F,F \rangle}  \right).
$$
Similarly, using Proposition \ref{proparithmeticitydiff},
$$
 \frac{\langle \leftexp{\sigma}G, \leftexp{\sigma}XF\rangle}{\langle \leftexp{\sigma}XF,\leftexp{\sigma}XF \rangle}  = \frac{\langle \leftexp{\sigma}G', \leftexp{\sigma}F\rangle}{\langle \leftexp{\sigma}F,\leftexp{\sigma}F \rangle}.
$$
Now the result follows from the property of $F$ assumed in the statement of the proposition.
\end{proof}

\begin{remark}The condition that $\mathcal{X}_{\ell, m}^{\ell_1, m_1}$ is a singleton set is satisfied when $\ell_1 = \ell+m$ and $m_1=0$, provided $m$ is even. In this case,  we have $\mathcal{X}_{\ell, m}^{\ell_1, m_1} = \{U^{m/2} \}$. The application of the above proposition in this special case will be of crucial importance in our upcoming work.
\end{remark}

\begin{proposition}Let $F$ be as in Proposition~\ref{pterssonratio}.
Assume further that $\ell>3$ and $E_{\ell,m}(\Gamma)$ is $L$-rational for some number field $L$.

Let $\ell_1, m_1$ be integers such that $\mathcal{X}_{\ell, m}^{\ell_1, m_1}$ is a singeleton set equal to $\{X \}$.
Then for all $G \in N_{\ell_1, m_1}(\Gamma)$, and all $\sigma \in \Aut(\C/L)$, we have $$ \sigma \left( \frac{\langle G, \ XF\rangle}{\langle XF, \ XF \rangle}  \right) =  \frac{\langle \leftexp{\sigma}G, \ \leftexp{\sigma}XF\rangle}{\langle \leftexp{\sigma}XF, \ \leftexp{\sigma}XF \rangle}.
$$
\end{proposition}
\begin{proof}
The proof is identical to Proposition~\ref{pterssonratio}, except that we use $\mathfrak{p}^\circ_{\ell,m}$.
\end{proof}

\bibliography{nearholo}{}

\end{document}